\documentclass[a4paper,11pt]{article}
\usepackage{amsmath,amsthm,amssymb,enumitem,xcolor}
\usepackage[hang]{footmisc}

\usepackage[nosort,nocompress,noadjust]{cite}

\usepackage[bookmarks=false,hyperfootnotes=false,colorlinks,linktoc=all,
    linkcolor={red!60!black},
    citecolor={blue!50!black},
    urlcolor={blue!80!black}]{hyperref}

\renewcommand{\eqref}[1]{\hyperref[#1]{(\ref{#1})}}

\newlist{enumlist}{enumerate}{2}
\setlist[enumlist,1]{labelindent=0cm,label=(\roman*),ref=(\roman*),labelwidth=4.5ex,labelsep=1ex,leftmargin=5.5ex,align=right,topsep=0.5ex,itemsep=1ex,parsep=1ex}
\setlist[enumlist,2]{labelindent=0cm,label=\alph*),ref=\arabic*,labelwidth=5ex,labelsep=0.5ex,leftmargin=5.5ex,align=left,topsep=0.5ex,itemsep=1ex,parsep=1ex}

\newlist{enumlistI}{enumerate}{1}
\setlist[enumlistI]{labelindent=0cm,label=(\Roman*),ref=(\Roman*),labelwidth=4.5ex,labelsep=1ex,leftmargin=5.5ex,align=right,topsep=0.5ex,itemsep=1ex,parsep=1ex}

\newlist{enumlistprime}{enumerate}{1}
\setlist[enumlistprime]{labelindent=0cm,label=(\roman*)$^{\prime}$,ref=(\roman*)$^{\prime}$,labelwidth=4.5ex,labelsep=1ex,leftmargin=5.5ex,align=right,topsep=0.5ex,itemsep=1ex,parsep=1ex}

\newlist{enumlistprimeprime}{enumerate}{1}
\setlist[enumlistprimeprime]{labelindent=0cm,label=(\roman*)$^{\prime\prime}$,ref=(\roman*)$^{\prime\prime}$,labelwidth=4.5ex,labelsep=1ex,leftmargin=5.5ex,align=right,topsep=0.5ex,itemsep=1ex,parsep=1ex}

\newlist{enumlistprimeprimeprime}{enumerate}{1}
\setlist[enumlistprimeprimeprime]{labelindent=0cm,label=(\roman*)$^{\prime\prime\prime}$,ref=(\roman*)$^{\prime\prime\prime}$,labelwidth=4.5ex,labelsep=1ex,leftmargin=5.5ex,align=right,topsep=0.5ex,itemsep=1ex,parsep=1ex}

\newlist{enuma}{enumerate}{1}
\setlist[enuma]{labelindent=0cm,label=(\alph*),ref=(\alph*),labelwidth=4.5ex,labelsep=1ex,leftmargin=5.5ex,align=right,topsep=0.5ex,itemsep=1ex,parsep=1ex}

\newlist{itemlist}{itemize}{1}
\setlist[itemlist]{labelindent=0cm,label=$\bullet$,labelwidth=2.5ex,labelsep=0.5ex,leftmargin=3ex,align=left,topsep=0.5ex,itemsep=1ex,parsep=1ex}

\numberwithin{equation}{section}

{\theoremstyle{definition}\newtheorem{definitiona}{Definition}[section]
\newtheorem{remarka}[definitiona]{Remark}
\newtheorem{examplea}[definitiona]{Example}}

\newtheorem{propositiona}[definitiona]{Proposition}
\newtheorem{lemmaa}[definitiona]{Lemma}
\newtheorem{theorema}[definitiona]{Theorem}
\newtheorem{corollarya}[definitiona]{Corollary}

\newtheorem{letterthma}{Theorem}
\renewcommand{\theletterthma}{\Alph{letterthma}}

{\theoremstyle{definition}
\newtheorem{letterremarka}[letterthma]{Remark}
\renewcommand{\theletterremarka}{\Alph{letterthma}}
}

\newenvironment{definition}[1][]{\begin{definitiona}[#1]\setlist*[enumlist,1]{label=(\roman*),ref=\thedefinitiona(\roman*)}\setlist*[enuma]{label=(\alph*),ref=\thedefinitiona(\alph*)}}{\end{definitiona}}
\newenvironment{remark}[1][]{\begin{remarka}[#1]\setlist*[enumlist,1]{label=(\roman*),ref=\theremarka(\roman*)}\setlist*[enuma]{label=(\alph*),ref=\theremarka(\alph*)}}{\end{remarka}}

\newenvironment{proposition}[1][]{\begin{propositiona}[#1]\setlist*[enumlist,1]{label=(\roman*),ref=\thepropositiona(\roman*)}\setlist*[enuma]{label=(\alph*),ref=\thepropositiona(\alph*)}}{\end{propositiona}}
\newenvironment{lemma}[1][]{\begin{lemmaa}[#1]\setlist*[enumlist,1]{label=(\roman*),ref=\thelemmaa(\roman*)}\setlist*[enuma]{label=(\alph*),ref=\thelemmaa(\alph*)}}{\end{lemmaa}}
\newenvironment{theorem}[1][]{\begin{theorema}[#1]\setlist*[enumlist,1]{label=(\roman*),ref=\thetheorema(\roman*)}\setlist*[enuma]{label=(\alph*),ref=\thetheorema(\alph*)}}{\end{theorema}}
\newenvironment{corollary}[1][]{\begin{corollarya}[#1]\setlist*[enumlist,1]{label=(\roman*),ref=\thecorollarya(\roman*)}\setlist*[enuma]{label=(\alph*),ref=\thecorollarya(\alph*)}}{\end{corollarya}}
\newenvironment{letterthm}[1][]{\begin{letterthma}[#1]\setlist*[enumlist,1]{label=(\roman*),ref=\theletterthma(\roman*)}\setlist*[enuma]{label=(\alph*),ref=\theletterthma(\alph*)}}{\end{letterthma}}
\newenvironment{letterremark}[1][]{\begin{letterremarka}[#1]\setlist*[enumlist,1]{label=(\roman*),ref=\theletterthma(\roman*)}\setlist*[enuma]{label=(\alph*),ref=\theletterremarka(\alph*)}}{\end{letterremarka}}

\newcommand{\C}{\mathbb{C}}
\newcommand{\Z}{\mathbb{Z}}
\newcommand{\N}{\mathbb{N}}

\newcommand{\F}{\mathbb{F}}
\newcommand{\T}{\mathbb{T}}
\newcommand{\Q}{\mathbb{Q}}

\newcommand{\cC}{\mathcal{C}}

\newcommand{\cH}{\mathcal{H}}
\newcommand{\cZ}{\mathcal{Z}}
\newcommand{\cG}{\mathcal{G}}
\newcommand{\cK}{\mathcal{K}}
\newcommand{\cJ}{\mathcal{J}}
\newcommand{\cF}{\mathcal{F}}
\newcommand{\cA}{\mathcal{A}}
\newcommand{\cB}{\mathcal{B}}

\newcommand{\cU}{\mathcal{U}}
\newcommand{\cM}{\mathcal{M}}
\newcommand{\cN}{\mathcal{N}}

\newcommand{\cV}{\mathcal{V}}

\newcommand{\cP}{\mathcal{P}}

\newcommand{\cQ}{\mathcal{Q}}

\newcommand{\al}{\alpha}
\newcommand{\be}{\beta}
\newcommand{\vphi}{\varphi}
\newcommand{\om}{\omega}
\newcommand{\Om}{\Omega}
\newcommand{\si}{\sigma}

\newcommand{\Ad}{\operatorname{Ad}}
\newcommand{\Ext}{\operatorname{Ext}}
\newcommand{\Ker}{\operatorname{Ker}}

\newcommand{\Hom}{\operatorname{Hom}}
\newcommand{\Aut}{\operatorname{Aut}}

\newcommand{\Stab}{\operatorname{Stab}}

\newcommand{\Norm}{\operatorname{Norm}}
\newcommand{\res}{\operatorname{res}}
\newcommand{\Haar}{\operatorname{Haar}}
\newcommand{\BHom}{\operatorname{BHom}}

\newcommand{\ev}{\operatorname{ev}}

\newcommand{\Gtil}{\widetilde{G}}

\newcommand{\Lambdatil}{\widetilde{\Lambda}}
\newcommand{\cMtil}{\widetilde{\mathcal{M}}}
\newcommand{\Atil}{\widetilde{A}}

\newcommand{\deltatil}{\widetilde{\delta}}

\newcommand{\Chat}{\widehat{C}}
\newcommand{\Dhat}{\widehat{D}}

\newcommand{\ot}{\otimes}
\newcommand{\id}{\mathord{\text{\rm id}}}
\newcommand{\ovt}{\mathbin{\overline{\otimes}}}
\newcommand{\actson}{\curvearrowright}

\newcommand{\dpr}{^{\prime\prime}}
\newcommand{\op}{^\text{\rm op}}

\newcommand{\fc}{_\text{\rm fc}}
\newcommand{\ab}{_\text{\rm ab}}
\newcommand{\sym}{_\text{\rm sym}}
\newcommand{\tor}{_\text{\rm tor}}
\newcommand{\ve}{^\text{\rm v}}

\newcommand{\bim}[3]{\mathord{\raisebox{-0.4ex}[0ex][0ex]{\scriptsize $#1$}{#2}\hspace{-0.2ex}\raisebox{-0.4ex}[0ex][0ex]{\scriptsize $#3$}}}

\begin{document}

\begin{center}
{\boldmath\LARGE\bf W$^*$-superrigidity for groups with infinite center}

\vspace{1ex}

{\sc by Milan Donvil\footnote{KU Leuven, Department of Mathematics, Leuven (Belgium), milan.donvil@kuleuven.be\\ Supported by PhD grant 1162024N funded by the Research Foundation Flanders (FWO)} and Stefaan Vaes\footnote{KU~Leuven, Department of Mathematics, Leuven (Belgium), stefaan.vaes@kuleuven.be\\ Supported by FWO research projects G090420N and G016325N of the Research Foundation Flanders and by Methusalem grant METH/21/03 –- long term structural funding of the Flemish Government.}}
\end{center}

\begin{abstract}\noindent
We construct discrete groups $G$ with infinite center that are nevertheless W$^*$-superrigid, meaning that the group von Neumann algebra $L(G)$ fully remembers the group $G$. We obtain these rigidity results both up to isomorphisms and up to virtual isomorphisms of the groups and their von Neumann algebras. Our methods combine rigidity results for the quotient of these groups by their center with rigidity results for their $2$-cohomology.
\end{abstract}

\section{Introduction}

Starting with \cite{IPV10} and making use of Popa's deformation\slash rigidity theory, several families of discrete groups $G$ were shown to be \emph{W$^*$-superrigid}: their group von Neumann algebra $L(G)$ fully remembers the group $G$, in the sense that $L(G) \cong L(\Lambda)$ for an arbitrary discrete group $\Lambda$ if and only if $G \cong \Lambda$. Such a W$^*$-superrigidity property was established for several generalized wreath product groups $G = (\Z/2\Z)^{(I)} \rtimes \Gamma$ in \cite{IPV10}, for many left-right wreath product groups $G = (\Z/2\Z)^{(\Gamma)} \rtimes (\Gamma \times \Gamma)$ in \cite{BV12}, and in \cite{CIOS21}, for a family of wreath-like products with Kazhdan's property~(T).

The furthest away from W$^*$-superrigidity are infinite abelian groups $G$, because their group von Neumann algebras $L(G)$ are all trivially isomorphic with the unique diffuse abelian von Neumann algebra $L^\infty([0,1])$. As an immediate consequence, no direct product $C \times G$ of an infinite abelian group $C$ with a discrete group $G$ can be W$^*$-superrigid, since $L(C_1 \times G) \cong L(C_2 \times G)$ for all infinite abelian groups $C_i$. Note, however, that it was proven in \cite{CFQT24} that for certain groups $G$, this is the only obstruction to W$^*$-superrigidity: if for such a group $G$ and an infinite abelian group $C_1$, we have $L(C_1 \times G) \cong L(\Lambda)$, there exists an infinite abelian group $C_2$ with $\Lambda \cong C_2 \times G$.

The previous paragraph shows that \emph{trivial central extensions} $e \to C \to C \times G \to G \to e$ are never W$^*$-superrigid when $C$ is infinite. In the main result of this paper, we show that many left-right wreath product groups $G$ admit \emph{nontrivial central extensions} $e \to C \to \Gtil \to G \to e$ such that $\Gtil$ is W$^*$-superrigid. These and the examples in \cite{CFQOT24} are the first W$^*$-superrigid groups with infinite center.

We prove two versions of this result. We consider W$^*$-superrigidity up to isomorphisms, proving that $L(\Gtil) \cong L(\Lambda)$ if and only if $\Gtil \cong \Lambda$, but we also consider W$^*$-superrigidity up to \emph{virtual isomorphisms}. By definition, two II$_1$ factors $M$ and $P$ are called virtually isomorphic if there exists an $M$-$P$-bimodule that is \emph{bifinite}, i.e.\ finitely generated as a left $M$-module and as a right $P$-module. For arbitrary finite von Neumann algebras $M$ and $P$, there are two variants of this notion, requiring the existence of a nonzero, resp.\ faithful, bifinite bimodule. We discuss this in detail in Section \ref{sec.virtual-iso}. Two groups are said to be \emph{commensurable} if they admit isomorphic finite index subgroups, and said to be \emph{virtually isomorphic} if they are commensurable up to finite kernels, i.e.\ isomorphic after passing to finite index subgroups and taking the quotient by finite normal subgroups. We then prove that the existence of a nonzero, resp.\ faithful, bifinite $L(\Gtil)$-$L(\Lambda)$-bimodule is equivalent to $\Lambda$ being virtually isomorphic, resp.\ commensurable, to $\Gtil$.

The first W$^*$-superrigidity result for group von Neumann algebras up to virtual isomorphisms was established in \cite{DV24}. In that paper, we also allowed the group von Neumann algebras to be twisted by a $2$-cocycle, proving that for many left-right wreath product groups $G$, we have that $L_\mu(G)$ is (virtually) isomorphic to any other twisted group von Neumann algebra $L_\om(\Lambda)$ if and only if $(G,\mu)$ is (virtually) isomorphic to $(\Lambda,\om)$. Since the group von Neumann algebra $L(\Gtil)$ of a central extension $e \to C \to \Gtil \to G \to e$ is naturally decomposed as a direct integral of twisted group von Neumann algebras $L_\mu(G)$ (see \eqref{eq.direct-integral-intro}), the results of \cite{DV24} are a key ingredient in the proof of the main result in this paper.

Even more so, in Sections \ref{sec.generic-iso-superrigid} and \ref{sec.generic-virtual-iso-superrigid}, we prove two meta rigidity theorems \ref{thm.inherit-iso-superrigidity} and \ref{thm.inherit-virtual-iso-superrigidity}. For \emph{arbitrary} central extensions $e \to C \to \Gtil \to G \to e$, where $C$ is a torsion free abelian group and the twisted group von Neumann algebras $L_\mu(G)$ satisfy (virtual) W$^*$-superrigidity, we describe all countable groups $\Lambda$ whose group von Neumann algebra $L(\Lambda)$ is (virtually) isomorphic to $L(\Gtil)$. Under extra rigidity assumptions on the $2$-cohomology $H^2(G,C)$, we prove that this forces $\Lambda$ to be (virtually) isomorphic to $\Gtil$. As an application, we then provide the following uncountable family of groups with infinite center that are W$^*$-superrigid, both for isomorphisms and for virtual isomorphisms.

\begin{letterthm}\label{thm.main-A}
Let $\cP$ be a set of prime numbers with $2 \in \cP$. Define $C = \Z[\cP^{-1}]$ as the subring of $\Q$ generated by $\cP^{-1}$. Put $\Gamma = C \ast \Z$ and consider the left-right wreath product group $G = (\Z/2\Z)^{(\Gamma)} \rtimes (\Gamma \times \Gamma)$.

There are infinitely many central extensions $0 \to C \to \Gtil \to G \to e$ such that $\Gtil$ is W$^*$-superrigid for isomorphisms and for virtual isomorphisms: for every countable group $\Lambda$, we have that

\begin{enuma}
\item\label{thm.main-A.1} $p L(\Gtil) p \cong L(\Lambda)$ for a nonzero projection $p \in L(\Gtil)$ iff $p=1$ and $\Lambda$ and $\Gtil$ are isomorphic~;
\item\label{thm.main-A.2} there exists a nonzero bifinite $L(\Gtil)$-$L(\Lambda)$-bimodule iff $\Lambda$ and $\Gtil$ are virtually isomorphic~;
\item\label{thm.main-A.3} there exists a faithful bifinite $L(\Gtil)$-$L(\Lambda)$-bimodule iff $\Lambda$ and $\Gtil$ are commensurable.
\end{enuma}
\end{letterthm}

The infinitely many central extensions $\Gtil$ appearing in Theorem \ref{thm.main-A} are not only distinct as central extensions, but are actually not even virtually isomorphic as groups.

In Section \ref{sec.lack-of-superrigidity}, we explain why Theorem \ref{thm.main-A} is formulated for these specific abelian groups $C = \Z[\cP^{-1}]$. We prove in Proposition \ref{prop.new-no-go-result-2} that left-right wreath product groups $G = (\Z/2\Z)^{(\Gamma)} \rtimes (\Gamma \times \Gamma)$ with $\Gamma$ icc have \emph{no} W$^*$-superrigid central extension $0 \to C \to \Gtil \to G \to e$ unless the group $C$ is $2$-divisible. It is thus natural to use $C = \Z[\cP^{-1}]$ with $2 \in \cP$.

Moreover, we prove in Proposition \ref{prop.new-no-go-result-1} that W$^*$-superrigidity also fails if the $2$-cocycle $\Om \in H^2(G,C)$ associated with a central extension lacks a natural surjectivity property. In order to have $2$-cocycles with such a surjectivity property, we need to have $C$ as a free factor of $\Gamma$.

\begin{letterremark}\label{rem.main}
Since the group $C$ in Theorem \ref{thm.main-A} is an infinite abelian group, the trivial extension $C \times G$ obviously is not W$^*$-superrigid, as discussed above. But, as we explain in Section \ref{sec.proof-main}, more is true for the groups $G$ appearing in Theorem \ref{thm.main-A}.

First assume that $\cP$ is a proper set of prime numbers, i.e.\ $C \neq \Q$. Then,
\begin{enuma}
\item the group $G$ also admits infinitely many central extensions $0 \to C \to \Gtil \to G \to e$ that are W$^*$-superrigid for virtual isomorphisms (as in Theorem \ref{thm.main-A.2} and \ref{thm.main-A.3}), but that are not W$^*$-superrigid for isomorphisms, because there exist groups $\Lambda \not\cong \Gtil$ with $L(\Lambda) \cong L(\Gtil)$~;

\item and the group $G$ admits infinitely many nonisomorphic central extensions $0 \to C \to \Gtil \to G \to e$ that are not W$^*$-superrigid in any sense, because there exist groups $\Lambda$ that are not virtually isomorphic to $\Gtil$, but nevertheless satisfy $L(\Lambda) \cong L(\Gtil)$.
\end{enuma}

Next assume that $\cP$ consists of all prime numbers, so that $C = \Q$.
\begin{enuma}[resume]
\item Apart from the trivial extension $C \times G$, the group $G$ has precisely one other exceptional central extension $0 \to C \to \Gtil \to G \to e$ (up to isomorphism of the group $\Gtil$) that is not W$^*$-superrigid in any sense, but all other central extensions of $G$ by $C$ satisfy all the isomorphism and virtual isomorphism W$^*$-superrigidity properties of Theorem \ref{thm.main-A}.
\end{enuma}

Finally, we replace $\Gamma = C \ast \Z$ by $\Gamma = C \ast C$, where $C = \Z[\cP^{-1}]$ and $\cP$ is a set of prime numbers with $2 \in \cP$.
\begin{enuma}[resume]
\item If $C \neq \Q$, then every nontrivial central extension of $G$ by $C$ satisfies the virtual isomorphism W$^*$-superrigidity properties of Theorem \ref{thm.main-A.2} and \ref{thm.main-A.3}, while there still are infinitely many central extensions that are not isomorphism W$^*$-superrigid.

\item If $C = \Q$, every nontrivial central extension of $G$ by $C$ satisfies all the isomorphism and virtual isomorphism W$^*$-superrigidity properties of Theorem \ref{thm.main-A}.
\end{enuma}
\end{letterremark}

In Theorem \ref{thm.rigidity-trivial-extensions}, we also prove a generic rigidity theorem for trivial central extensions $C \times G$, as in \cite{CFQT24}, and provide new examples in Proposition \ref{prop.examples-rigidity-trivial-extensions}.

We conclude this introduction by outlining the proof of Theorem \ref{thm.main-A}. As we explain in \eqref{eq.central-decomp}, the group von Neumann algebra $L(\Gtil)$ of any central extension of $G$ by $C$, with associated $2$-cocycle $\Om \in H^2(G,C)$, has a direct integral decomposition
\begin{equation}\label{eq.direct-integral-intro}
L(\Gtil) = \int^\oplus_{\mu \in \Chat} L_{\mu \circ \Om}(G) \, d\mu \; ,
\end{equation}
where we integrate w.r.t.\ the Haar measure on the compact Pontryagin dual group $\Chat$. In \cite[Theorem A]{DV24}, we proved W$^*$-superrigidity, both up to isomorphisms and up to virtual isomorphisms, for the twisted group von Neumann algebra $L_\om(G)$, $\om \in H^2(G,\T)$. That is a first key ingredient in the proof of Theorem \ref{thm.main-A}.

Given a countable group $\Lambda$ and an isomorphism $p L(\Gtil) p \cong L(\Lambda)$, or merely a nonzero bifinite $L(\Gtil)$-$L(\Lambda)$-bimodule, our first step is to prove that $\Lambda$ is itself, up to a virtual isomorphism, a central extension $0 \to D \to \Lambda \to \Lambda_0 \to e$, in which $D$ is at the same time equal to the center of $\Lambda$ and the virtual center $\Lambda\fc$, i.e.\ the normal amenable subgroup of elements having a finite conjugacy class.

To obtain this first step, note that $L(\Lambda\fc) \subset L(\Lambda)$ is an amenable von Neumann subalgebra. Denoting by $(v_s)_{s \in \Lambda}$ the unitaries generating $L(\Lambda)$, we have that $(\Ad v_s)_{s \in \Lambda}$ defines an action of $\Lambda$ by automorphisms of $L(\Lambda\fc)$ that is compact, meaning that the closure of $\Lambda$ inside $\Aut L(\Lambda\fc)$ is compact. Through the (virtual) isomorphism of $L(\Lambda)$ with $L(\Gtil)$, we transport $L(\Lambda\fc)$ to an amenable subalgebra of (an amplification of) $L(\Gtil)$ having a large normalizer that acts compactly. In Section \ref{sec.weakly-compact}, we piece together several existing results from Popa's deformation\slash theory, as obtained in \cite{Ioa06,OP07,PV12}, to conclude that $L(\Lambda\fc)$ has an intertwining bimodule into the center of $L(\Gtil)$, which is $L(C)$.

Using Lemma \ref{lem.transfer-embedding-in-center}, we come back to $L(\Lambda)$ and conclude that inside $L(\Lambda)$, there must be an intertwining bimodule of $L(\Lambda\fc)$ into the center of $L(\Lambda)$. In Proposition \ref{prop.structure-non-icc}, we deduce that up to finite index considerations, $\Lambda\fc$ coincides with the center of $D = \cZ(\Lambda)$. Up to a virtual isomorphism, we thus obtain the central extension $0 \to D \to \Lambda \to \Lambda_0 \to e$ and we denote by $\Om' \in H^2(\Lambda_0,D)$ the associated $2$-cocycle. We get the direct integral decomposition
$$L(\Lambda) = \int^\oplus_{\om \in \Dhat} L_{\om \circ \Om'}(\Lambda_0) \, d\om \; .$$
Using Proposition \ref{prop.iso-on-center}, we deduce that the original (virtual) isomorphism between $L(\Gtil)$ and $L(\Lambda)$ gives rise to nonnegligible Borel sets $\cU \subset \Chat$, $\cV \subset \Dhat$ and a nonsingular isomorphism $\Delta : \cU \to \cV$ such that for all $\mu \in \cU$, we have that $L_{\mu \circ \Om}(G)$ is (virtually) isomorphic to $L_{\om \circ \Om'}(\Lambda_0)$.

By the W$^*$-superrigidity theorem proved in \cite[Theorem A]{DV24}, it follows that the groups $G$ and $\Lambda_0$ are (virtually) isomorphic and that, for every $\mu \in \cU$, there exists a (virtual) isomorphism between $(G,\mu \circ \Om)$ and $(\Lambda_0,\Delta(\mu) \circ \Om')$.

In order to reach the conclusion of Theorem \ref{thm.main-A}, we have to prove that $(G,\Om)$ is (virtually) isomorphic to $(\Lambda_0,\Om')$. We thus need to lift the virtual isomorphisms between the $\T$-valued $2$-cocycles $\mu \circ \Om$ and $\Delta(\mu) \circ \Om'$ to a virtual isomorphism $\rho$ between $C$ and $D$ and a virtual isomorphism between $\rho \circ \Om$ and $\Om'$. For this, we need a careful analysis of the $2$-cohomology group $H^2(G,C)$ and the group homomorphism $\Chat \to H^2(G,\T) : \mu \mapsto \mu \circ \Om$. We do this analysis separately for the case of isomorphisms and the case of virtual isomorphisms in Sections \ref{sec.generic-iso-superrigid}, resp.\ \ref{sec.generic-virtual-iso-superrigid}. We actually prove in these sections generic meta-versions of Theorem \ref{thm.main-A.1}, resp.\ \ref{thm.main-A.2} and \ref{thm.main-A.3}, for groups satisfying the conclusions of \cite[Theorem A]{DV24} and satisfying a number of hypotheses on their $2$-cohomology. In Section \ref{sec.elementary-2-cohom-comput}, we collect the needed $2$-cohomology computations for left-right wreath product groups, so as to prove that they satisfy the hypotheses of our meta-theorems.

In Section \ref{sec.lack-of-superrigidity}, we prove several no-go results, providing central extensions that are not W$^*$-superrigid, which will be used to prove Remark \ref{rem.main}. In Section \ref{sec.proof-main}, we put all pieces together and prove Theorem \ref{thm.main-A}, as a consequence of the more precise Theorem \ref{thm.more-general}, and we prove the statements in Remark \ref{rem.main}.

Our work on this paper was done in parallel with and independently of the work \cite{CFQOT24}. In that paper, the authors prove that a family of central extensions $e \to C \to \Gtil \to G \to e$ \emph{with Kazhdan's property (T)} is W$^*$-superrigid, up to isomorphisms. As in \cite{CIOS21}, the quotient groups $G$ are wreath-like products with property~(T). An important part of the paper \cite{CFQOT24} deals with W$^*$-superrigidity for the twisted group von Neumann algebras $L_\mu(G)$, which in the setting of our paper was already done in \cite{DV24}. Also, the paper \cite{CFQOT24} does not consider virtual isomorphisms. Therefore, the overlap between this paper and \cite{CFQOT24} is small and limited to some of the methodology appearing in Section \ref{sec.generic-iso-superrigid}; see Remark \ref{rem.relation-to-CFQOT24} for a more detailed discussion.

Unless we explicitly say the contrary, in this paper, all groups are countable and discrete. We systematically write the group operation in an abelian group additively, mainly because the additive group $\Z[\cP^{-1}]$ plays a key role. The only exception is our notation of the group $\T$, viewed as the unit circle in $\C$ and thus written multiplicatively.

\section{Preliminaries}

\subsection{Bifinite bimodules and intertwining-by-bimodules}

Let $P$ be a finite von Neumann algebra and denote its standard Hilbert space as $L^2(P)$. A right Hilbert $P$-module $H_P$ is said to be of \emph{finite type} if it is isomorphic to a closed $P$-submodule of $L^2(P)^{\oplus n}$ for some $n \in \N$. This is equivalent to saying that $H_P$ is isomorphic with $p(\C^n \ot L^2(P))$ for some projection $p \in M_n(\C) \ot P$. When $P$ is countably decomposable, in particular when $P$ has separable predual as is the case in all concrete applications in this paper, a right Hilbert $P$-module $H_P$ is of finite type if and only if it is finitely generated, in the sense that there exists a finite subset $\cF \subset H$ such that $\cF \cdot P$ has dense linear span in $H$.

Let $M$ and $P$ be finite von Neumann algebras and $\bim{M}{H}{P}$ a Hilbert $M$-$P$-bimodule. Assume that $H$ is of finite type as a right Hilbert $P$-module, so that $H_P \cong p(\C^n \ot L^2(P))$. The left $M$-action is then expressed by a unital normal $*$-homomorphism $\vphi : M \to p(M_n(\C) \ot P)p$ and $\bim{M}{H}{P}$ is isomorphic with the $M$-$P$-bimodule
$$\bim{\vphi(M)}{p(\C^n \ot L^2(P))}{P} \quad\text{with}\quad a \cdot \xi \cdot b = \vphi(a) \xi b \quad\text{for all $a \in M$, $b \in P$.}$$
Also recall that a Hilbert $M$-$P$-bimodule $\bim{M}{H}{P}$ is said to be \emph{bifinite} if $H$ is of finite type as a right Hilbert $P$-module and also as a left Hilbert $M$-module.

In \cite[Section 2]{Pop03}, Popa introduced the fundamental concept of \emph{intertwining-by-bimodules}, which we recall here. For a finite von Neumann algebra $M$, a nonzero projection $p \in M_n(\C) \ot M$ and von Neumann subalgebras $A \subset p(M_n(\C) \ot M)p$ and $B \subset M$, one says that $A \prec_M B$ if $p(\C^n \ot L^2(M))$ admits a nonzero Hilbert $A$-$B$-subbimodule that is of finite type as a right Hilbert $B$-module. One says that $A \prec^f_M B$ if $Aq \prec_M B$ for every nonzero projection $q \in A' \cap p(M_n(\C) \ot M)p$. Finally recall from \cite{Pop03} that $A \prec_M B$ if and only if there exist a projection $r \in M_k(\C) \ot B$, a unital normal $*$-homomorphism $\vphi : A \to r(M_k(\C) \ot B)r$ and a nonzero partial isometry $v \in p(M_{n,k}(\C) \ot M)r$ satisfying $a v = v \vphi(a)$ for all $a \in A$.

We use the following folklore result on bifinite bimodules between finite von Neumann algebras with nontrivial center. For completeness, we include a proof.

\begin{proposition}\label{prop.iso-on-center}
Let $M$ and $P$ be finite von Neumann algebras and $\bim{M}{H}{P}$ a bifinite Hilbert $M$-$P$-bimodule. Then $H$ can be written as the direct sum of Hilbert $M$-$P$-subbimodules $H_i \subset H$ with the following property: denoting by $z_i \in \cZ(M)$ the left support of $H_i$ and by $r_i \in \cZ(P)$ the right support of $H_i$, there is a $*$-isomorphism $\al_i : \cZ(M) z_i \to \cZ(P) r_i$ such that $a \xi = \xi \al_i(a)$ for all $a \in \cZ(M) z_i$ and $\xi \in H_i$.
\end{proposition}

Before proving Proposition \ref{prop.iso-on-center}, we need a lemma.

\begin{lemma}\label{lem.equal-under-q}
Let $M$ be a finite von Neumann algebra, $p \in M_n(\C) \ot M$ a projection and $A \subset p(M_n(\C) \ot M)p$ an abelian von Neumann subalgebra such that $(1 \ot \cZ(M))p \subset A$.

Then $A \prec_M \cZ(M)$ if and only if there exists a nonzero projection $q \in A$ such that $A q = (1 \ot \cZ(M))q$.
\end{lemma}
\begin{proof}
If $A q = (1 \ot \cZ(M))q$, it is immediate that $A \prec_M \cZ(M)$.

To prove the converse, assume that $A \prec_M \cZ(M)$. Take $k \in \N$, a nonzero projection $r \in M_k(\C) \ot \cZ(M)$, a unital normal $*$-homomorphism $\vphi : A \to r(M_k(\C) \ot \cZ(M))r$ and a nonzero $v \in p(M_{n,k}(\C) \ot M)r$ such that $a v = v \vphi(a)$ for all $a \in A$.

Denote by $D_k(\C) \subset M_k(\C)$ the diagonal subalgebra. Then $\vphi(A)$ is an abelian von Neumann subalgebra of a corner of the finite type~I algebra $M_k(\C) \ot \cZ(M)$ with maximal abelian subalgebra $D_k(\C) \ot \cZ(M)$. So, $\vphi(A)$ can be unitarily conjugated into $D_k(\C) \ot \cZ(M)$. This means that we can choose $w \in M_k(\C) \ot \cZ(M)$ such that $ww^* = r$,
$$w^*w = \sum_{i=1}^k e_{ii} \ot z_i \quad\text{and}\quad w^* \vphi(a) w = \sum_{i=1}^k e_{ii} \ot \vphi_i(a) \quad\text{for all $a \in A$,}$$
where $z_i \in \cZ(M)$ are projections and $\vphi_i : A \to \cZ(M) z_i$ are unital normal $*$-homomorphisms. Write the matrix $vw \in M_{n,k}(\C) \ot M$ as $vw = \sum_{i=1}^k v_i (e_{1i} \ot 1)$ with $v_i \in \C^n \ot M$. Then, $v_i \in p(\C^n \ot Mz_i)$ and $a v_i = v_i \vphi_i(a)$ for all $a \in A$. Since $ww^*=r$, we have that $vw \neq 0$. Fix $i$ such that $v_i \neq 0$. In particular, $z_i \neq 0$. We now write $v$, $z$ and $\vphi$ instead of $v_i$, $z_i$ and $\vphi_i$. We have thus found a projection $z \in \cZ(M)$, a nonzero $v \in p(\C^n \ot Mz)$ and a unital normal $*$-homomorphism $\vphi : A \to \cZ(M)z$ such that $a v = v \vphi(a)$ for all $a \in A$.

Since $\vphi(A) \subset \cZ(M)$, it follows that $a v = v \vphi(a) = (1 \ot \vphi(a))v$ for all $a \in A$. So, $a vv^* = (1 \ot \vphi(a))vv^*$ for all $a \in A$. Using the unique trace preserving conditional expectation $E_A$ and using that $(1 \ot \cZ(M))p \subset A$, we find that $a E_A(vv^*) = (1 \ot \vphi(a)) E_A(vv^*)$ for all $a \in A$. Defining $q \in A$ as the support projection of $E_A(vv^*)$, we get that $A q \subset (1 \ot \cZ(M))q$. Since $(1 \ot \cZ(M))p \subset A$, we also have that $(1 \ot \cZ(M))q \subset Aq$ and the lemma is proven.
\end{proof}

\begin{proof}[{Proof of Proposition \ref{prop.iso-on-center}}]
Since we can use a maximality argument, it suffices to find a nonzero Hilbert $M$-$P$-subbimodule $H_0 \subset H$ with left support $z_0 \in \cZ(M)$, right support $r_0 \in \cZ(P)$ and a $*$-isomorphism $\al : \cZ(M) z_0 \to \cZ(P) r_0$ such that $a \xi = \xi \al(a)$ for all $a \in \cZ(M) z_0$ and $\xi \in H_0$.

Since $H_P$ is of finite type as a right Hilbert $P$-module, $\bim{M}{H}{P} \cong \bim{\vphi(M)}{p(\C^n \ot L^2(P))}{P}$ where $\vphi : M \to p (M_n(\C) \ot P)p$ is a unital normal $*$-homomorphism. Denote $P_1 = p (M_n(\C) \ot P)p$. Since $H$ is of finite type as a left Hilbert $M$-module, also $\bim{\vphi(M)}{L^2(P_1)}{P_1}$ is of finite type as a left Hilbert $M$-module. It follows that $P_1 \prec_{P_1} \vphi(M)$. By \cite[Lemma 3.5]{Vae07}, we can take relative commutants and conclude that $\vphi(M)' \cap P_1 \prec_{P_1} \cZ(P_1)$. Note that $\cZ(P_1) = (1 \ot \cZ(P))p$.

Define the abelian von Neumann algebra $A = \vphi(\cZ(M)) \vee (1 \ot \cZ(P))p$. Since $A \subset \vphi(M)' \cap P_1$, we get that $A \prec_{P_1} (1 \ot \cZ(P))p$. By Lemma \ref{lem.equal-under-q}, we find a nonzero projection $q \in A$ such that $A q = (1 \ot \cZ(P))q$. In particular, $\vphi(\cZ(M)) q \subset (1 \ot \cZ(P))q$. Define $H_1 := q(\C^n \ot L^2(P))$ and note that $H_1$ is a nonzero Hilbert $M$-$P$-subbimodule of $H$. Denote by $z_1 \in \cZ(M)$ and $r_1 \in \cZ(P)$ its left and right support. Since $\vphi(\cZ(M)) q \subset (1 \ot \cZ(P))q$, we find a faithful unital normal $*$-homomorphism $\al : \cZ(M) z_1 \to \cZ(P) r_1$ such that $\vphi(a)q = (1 \ot \al(a))q$ for all $a \in \cZ(M) z_1$. This implies that $a \xi = \xi \al(a)$ for all $a \in \cZ(M) z_1$ and $\xi \in H_1$.

Repeating the same reasoning starting from $H_1$, we find a nonzero Hilbert $M$-$P$-subbimodule $H_0 \subset H_1$, with left support $z_0 \in \cZ(M)z_1$, right support $r_0 \in \cZ(P)r_1$ and a faithful unital normal $*$-homomorphism $\be : \cZ(P) r_0 \to \cZ(M) z_0$ such that $\be(b) \xi = \xi b$ for all $b \in \cZ(P) r_0$ and $\xi \in H_0$. It follows that $\al(z_0) = r_0$ and that the restriction of $\al$ to $\cZ(M) z_0$ is the inverse of $\be$. So the proposition is proven.
\end{proof}

As explained at the end of the introduction, given a virtual isomorphism between finite von Neumann algebras $M$ and $P$ and given a von Neumann subalgebra $A \subset M$, we want to deduce from an intertwining of $A$ into the center $\cZ(P)$ inside $P$, that there also exists an intertwining of $A$ into $\cZ(M)$ inside $M$.

\begin{lemma}\label{lem.transfer-embedding-in-center}
Let $M$ and $P$ be finite von Neumann algebras and let $\bim{M}{H}{P} = \bim{\vphi(M)}{p(\C^n \ot L^2(P))}{P}$ be a bifinite Hilbert $M$-$P$-bimodule with left support $z \in \cZ(M)$. Let $A \subset M$ be a von Neumann subalgebra. If $\vphi(A) \prec^f_P \cZ(P)$, then $A z \prec^f_M \cZ(M)$.
\end{lemma}
\begin{proof}
Since $\bim{M}{H}{P}$ is bifinite, we can also write $\bim{M}{H}{P} \cong \bim{M}{(M_{1,k}(\C) \ot L^2(M))q}{\psi(P)}$ for some unital normal $*$-homomorphism $\psi : P \to q(M_k(\C) \ot M)q$. Since $\vphi(A) \prec^f_P \cZ(P)$, we also find that
$$(\id \ot \psi)\vphi(A) \prec^f_M \psi(\cZ(P)) \; .$$
By Proposition \ref{prop.iso-on-center}, we have that $\psi(\cZ(P)) \prec^f_M \cZ(M)$. It follows that $(\id \ot \psi)\vphi(A) \prec^f_M \cZ(M)$.

Since $\bim{M}{(H \ot_P \overline{H})}{M} \cong \bim{\vphi(M)}{p(M_n(\C) \ot L^2(P))p}{\vphi(M)}$, we can use $\vphi$ to define an $M$-bimodular isometry $L^2(M z) \to H \ot_P \overline{H}$. Since we also have that
$$\bim{M}{(H \ot_P \overline{H})}{M} \cong \bim{(\id \ot \psi)\vphi(M)}{p'(\C^{n} \ot \C^{k} \ot L^2(M))}{M} \; ,$$
where $p' = (\id \ot \psi)(p)$, this $M$-bimodular isometry becomes an element $V \in M_{nk,1}(\C) \ot M$ such that $VV^* \leq p'$, $V^* V = z$ and $(\id \ot \psi)\vphi(a) V = V a$ for all $a \in Mz$.

Since $(\id \ot \psi)\vphi(A) \prec^f_M \cZ(M)$, also $(\id \ot \psi)\vphi(A) VV^* \prec^f_M \cZ(M)$ and thus $A z \prec^f_M \cZ(M)$.
\end{proof}

\subsection{Commensurability and virtual isomorphism}\label{sec.virtual-iso}

\begin{definition}\label{def.virtual-iso}
We say that two discrete groups $G$ and $\Lambda$ are \emph{commensurable} if they admit isomorphic finite index subgroups. We say that $G$ and $\Lambda$ are \emph{virtually isomorphic} if there exist finite index subgroups $G_1 < G$, $\Lambda_1 < \Lambda$ and finite normal subgroups $N \lhd G_1$, $\Sigma \lhd \Lambda_1$ such that $G_1/N \cong \Lambda_1 / \Sigma$.
\end{definition}

Two II$_1$ factors $M$ and $P$ are called virtually isomorphic if there exists a nonzero bifinite $M$-$P$-bimodule $\bim{M}{H}{P}$. Since $M$ and $P$ are factors, the left $M$-action and the right $P$-action on $H$ are automatically faithful.

When $M$ and $P$ are arbitrary finite von Neumann algebras, we say that an $M$-$P$-bimodule $\bim{M}{H}{P}$ is \emph{faithful} if the left $M$-action and the right $P$-action on $H$ are both faithful. One could say that two tracial von Neumann algebras $M$ and $P$ are virtually isomorphic if there exists a faithful bifinite $M$-$P$-bimodule. This defines an equivalence relation on the class of finite von Neumann algebras, while the relation of existence of a nonzero bifinite $M$-$P$-bimodule does not define an equivalence relation on the class of finite von Neumann algebras. This can be seen by taking arbitrary finite von Neumann algebras $M$ and $P$, and noting that the latter relation holds between $M$ and $M \oplus P$, and also between $P$ and $M \oplus P$, but not necessarily between $M$ and $P$.

As can be seen immediately from the formulation of Theorem \ref{thm.main-A.2} and \ref{thm.main-A.3}, using this terminology of virtual isomorphism for von Neumann algebras would be extremely confusing. In the rest of this paper, we therefore write everywhere explicitly that there ``exists a nonzero\slash faithful bifinite bimodule''.

The following lemma is certainly well-known. For completeness, we include a proof.

\begin{lemma}\label{lem.from-virtual-iso-to-bimodule}
Let $G$ and $\Lambda$ be countable groups.
\begin{enuma}
\item\label{lem.from-virtual-iso-to-bimodule.1} If $G$ and $\Lambda$ are virtually isomorphic, there exists a nonzero bifinite Hilbert $L(G)$-$L(\Lambda)$-bimodule.
\item\label{lem.from-virtual-iso-to-bimodule.2} If $G$ and $\Lambda$ are commensurable, there exists a faithful bifinite Hilbert $L(G)$-$L(\Lambda)$-bimodule.
\end{enuma}
\end{lemma}
\begin{proof}
(a) Take finite index subgroups $G_0 < G$, $\Lambda_0 < \Lambda$, finite normal subgroups $\Sigma \lhd G_0$ and $\Sigma' \lhd \Lambda_0$ and a group isomorphism $\delta : G_0/\Sigma \to \Lambda_0 / \Sigma'$. Using $\delta$ to define the left action, we turn $H = \ell^2(\Lambda_0/\Sigma')$ into an $L(G_0/\Sigma)$-$L(\Lambda_0/\Sigma')$-bimodule that is singly generated as a left Hilbert module and singly generated as a right Hilbert module. We may view $H$ as a bifinite $L(G_0)$-$L(\Lambda_0)$-bimodule. Note that $H$ is nonzero but that the left, resp.\ right action is not faithful if $\Sigma$, resp.\ $\Sigma'$ is nontrivial.

The inclusion of $G_0$ into $G$ turns $\ell^2(G)$ into a faithful bifinite $L(G)$-$L(G_0)$-bimodule. We similarly view $\ell^2(\Lambda)$ as a faithful bifinite $L(\Lambda_0)$-$L(\Lambda)$-bimodule. Then,
$$\ell^2(G) \ot_{L(G_0)} H \ot_{L(\Lambda_0)} \ell^2(\Lambda)$$
is the required nonzero bifinite $L(G)$-$L(\Lambda)$-bimodule.

(b) If $G$ and $\Lambda$ have isomorphic finite index subgroups, in the proof of (a), we may take $\Sigma = \{e\} = \Sigma'$ and the resulting bifinite bimodule is faithful.
\end{proof}

\subsection{Second cohomology of discrete groups}\label{sec.second-cohomology-of-groups}

We recall a few basic concepts about group cohomology. Given a countable group $\Gamma$ and an abelian group $A$, one denotes by $Z^2(\Gamma,A)$ the group of normalized $2$-cocycles with values in $A$, i.e.\ maps $\Omega : \Gamma \times \Gamma \to A$ satisfying
$$\Om(g,h) + \Om(gh,k) = \Om(g,hk) + \Om(h,k) \quad\text{and}\quad \Om(g,e) = 0 = \Om(e,g) \quad\text{for all $g,h,k \in \Gamma$.}$$
Denoting by $C^1(\Gamma,A)$ the group of all maps $F : \Gamma \to A$ with $F(e) = 0$, one has the coboundary homomorphism
$$\partial : C^1(\Gamma,A) \to Z^2(\Gamma,A) : (\partial F) (g,h) = F(g) + F(h) - \overline{F(gh)} \; .$$
One defines $B^2(\Gamma,A) = \partial(C^1(\Gamma,A))$ and $H^2(\Gamma,A)$ as the quotient $Z^2(\Gamma,A) / B^2(\Gamma,A)$.

Given $\Om \in Z^2(\Gamma,A)$, the product $(a,g)\cdot (b,h) = (a+b+\Om(g,h),gh)$ turns $G = A \times \Gamma$ into a group and defines a central extension $0 \to A \to G \overset{q}{\longrightarrow} \Gamma \to e$. For every such central extension, we can choose a map $\vphi : \Gamma \to G$ satisfying $q \circ \vphi = \id$ and define $\Om \in Z^2(G,A)$ by $\vphi(g)\vphi(h) = \vphi(g,h) \Om(g,h)$. This defines the classical bijection between
$H^2(\Gamma,A)$ and central extensions, where two such central extensions $0 \to A \to G_i \overset{q_i}{\longrightarrow} \Gamma \to e$ are identified if there exists an isomorphism $\theta : G_1 \to G_2$ satisfying $\theta(c) = c$ for all $c \in A$ and $q_2 \circ \theta = q_1$.

When $\Gamma$ is abelian and $0 \to A \to G \to \Gamma \to 0$ is a central extension, then $G$ is abelian if and only if the associated $2$-cocycle $\Om \in H^2(\Gamma,A)$ is \emph{symmetric}: $\Om(g,h) = \Om(h,g)$ for all $g,h \in \Gamma$. In this way, we identify $\Ext^1(\Gamma,A)$, the group of abelian extensions of $\Gamma$ by $A$, with the group $H^2\sym(\Gamma,A)$ of symmetric $2$-cocycles divided by the coboundaries.

When the abelian coefficient group $A$ is the circle $\T$, we write the group operation as multiplication. With the topology of pointwise convergence, one gets that $C^1(\Gamma,\T)$ and $Z^2(\Gamma,\T)$ are compact abelian groups and that $\partial$ is continuous, so that its image is closed. With the quotient topology, $H^2(\Gamma,\T)$ then becomes a compact abelian group.

We say that a cocycle $\Om \in Z^2(\Gamma,\T)$ is of finite type if there exists a finite dimensional projective representation $\pi : \Gamma \to \cU(d)$ such that $\pi(g) \pi(h) = \Om(g,h) \pi(gh)$ for all $g,h \in \Gamma$.

Throughout the paper, we often make use of the universal coefficient theorem for group cohomology (see e.g.\ \cite[Theorem 3.8]{BT82} for an elementary treatment). It says that the Schur multiplier $M(\Gamma)$ may be interpreted as the Pontryagin dual of the compact abelian group $H^2(\Gamma,\T)$ and thus, that the sequence
\begin{equation}\label{eq.univ-coeff}
0 \to \Ext^1(\Gamma\ab,A) \overset{\Upsilon}{\longrightarrow} H^2(\Gamma,A) \overset{\Theta}{\longrightarrow} \Hom(\Chat,H^2(\Gamma,\T)) \to 1
\end{equation}
is exact and (non-canonically) split. The homomorphism $\Theta$ is defined by $\Theta(\Om) : \mu \mapsto \mu \circ \Om$, and the homomorphism $\Upsilon$ is given by identifying as above $\Ext^1(\Gamma\ab,A)$ with $H^2\sym(\Gamma\ab,A)$ and writing $\Upsilon(\Om) = \Om \circ \pi\ab$, where $\pi\ab : \Gamma \to \Gamma\ab$ is the abelianization homomorphism.

\subsection{Direct integral decompositions}

We make use of the theory of direct integrals of von Neumann algebras with separable predual and refer to \cite[Section IV.8]{Tak79} for the necessary background. In particular, if $M$ is any von Neumann algebra with separable predual, if $(X,\mu)$ is a standard probability space and if we are given a unital embedding $L^\infty(X,\mu) \subset \cZ(M)$, by \cite[Theorem IV.8.21]{Tak79}, we have a canonical decomposition
\begin{equation}\label{eq.abstract-direct-integral}
M = \int_X^\oplus M_x \, d\mu(x)
\end{equation}
and we have $L^\infty(X,\mu) = \cZ(M)$ if and only if a.e.\ $M_x$ is a factor.

We apply this in particular to the group von Neumann algebra $L(G)$ of a central extensions $0 \to C \to G \to \Gamma \to e$. Since $C < \cZ(G)$, we have that $L(C) \subset \cZ(L(G))$. Choose a lift $\theta : \Gamma \to G$ for the quotient homomorphism $G \to \Gamma$, and define the $2$-cocycle $\Om \in Z^2(\Gamma,C)$ by $\theta(g)\theta(h) = \Om(g,h) \theta(gh)$ for all $g,h \in \Gamma$. We use the bijection $C \times \Gamma \to G : (c,g) \mapsto c \theta(g)$ to identify $\ell^2(G) = \ell^2(C) \ot \ell^2(\Gamma)$. We equip the Pontryagin dual $\Chat$ with its Haar probability measure and identify $\ell^2(C) = L^2(\Chat)$ through the Fourier transform. Denoting for every $c \in C$ by $\ev_c \in L^\infty(\Chat)$ the unitary $\om \mapsto \om(c)$, we thus find the unitary
$$U : \ell^2(G) \to L^2(\Chat) \ot \ell^2(\Gamma) : U \delta_{c \theta(g)} = \ev_c \ot \delta_g \quad\text{for all $c \in C$, $g \in \Gamma$.}$$
Since $U\lambda_c U^* = \ev_c \ot 1$ for all $c \in C$, we get that $U L(C) U^* = L^\infty(\Chat) \ot 1$. For every $\om \in \Chat$, we have the scalar $2$-cocycle $\om \circ \Om \in Z^2(\Gamma,\T)$. For every $g \in \Gamma$, the unitary $U \lambda_{\theta(g)} U^*$ belongs to $L^\infty(\Chat) \ot B(\ell^2(\Gamma))$ and is given by the function $\Chat \to B(\ell^2(\Gamma)) : \om \mapsto \lambda_{\om \circ \Om}(g)$, where for every $\mu \in Z^2(\Gamma,\T)$, $\lambda_\mu$ denotes the $\mu$-regular representation that generates the twisted group von Neumann algebra $L_\mu(\Gamma)$. We have thus identified the abstract direct integral decomposition of \eqref{eq.abstract-direct-integral} with the concrete direct integral decomposition
\begin{equation}\label{eq.concrete-direct-integral}
L(G) = \int_{\Chat}^\oplus L_{\om \circ \Omega}(\Gamma) \, d\om \; ,
\end{equation}
where the integral is taken w.r.t.\ the Haar measure on $\Chat$.

\subsection{The Kurosh theorems}\label{sec.kurosh}

By the Kurosh subgroup theorem (see e.g.\ \cite[Theorem 14]{Ser80}), every subgroup $\Lambda$ of a free product $\Gamma = \Gamma_1 \ast \Gamma_2$ is freely generated by copies of $\Z$ and conjugates of subgroups of $\Gamma_i$, $i=1,2$. If $\Lambda < \Gamma$ has finite index, there are only finitely many of each.

The following elementary lemma is an immediate consequence of the Kurosh subgroup theorem. It is possible to fully describe $\Aut(\Gamma_1 \ast \Gamma_2)$, but we do not need that result in our paper.

\begin{lemma}\label{lem.countable-aut}
If $\Gamma = \Gamma_1 \ast \Gamma_2$ is the free product of two nontrivial abelian groups with countable automorphism group $\Aut \Gamma_i$ for all $i \in \{1,2\}$, then $\Aut \Gamma$ is countable.
\end{lemma}
\begin{proof}
If $\Gamma_1 \cong \Z \cong \Gamma_2$, then $\Gamma \cong \F_2$ and $\Aut \Gamma$ is countable. By symmetry, we may thus assume that $\Gamma_1 \not\cong \Z$.

Let $\delta \in \Aut \Gamma$. By the Kurosh subgroup theorem, after replacing $\delta$ by $(\Ad g) \circ \delta$ for some $g \in \Gamma$, we get that $\delta(\Gamma_1) < \Gamma_i$ for some $i \in \{1,2\}$. Since $\Gamma_i$ commutes with $\delta(\Gamma_1)$ and $\delta$ is an automorphism, it follows that $\delta(\Gamma_1) = \Gamma_i$.

If $\Gamma_2 \cong \Z$, then $i=1$. So, $\delta$ is fully determined by an automorphism of $\Gamma_1$, an inner automorphism and the value of $\delta(a_0) \in \Gamma$, where $a_0$ is a generator of $\Gamma_2$. So if $\Aut \Gamma_1$ is countable, also $\Aut \Gamma$ is countable.

If $\Gamma_2 \not\cong \Z$, we similarly find that $\delta(\Gamma_2) = h \Gamma_j h^{-1}$ for some $h \in \Gamma$ and $j \in \{1,2\}$. Since $\Aut \Gamma_1$ and $\Aut \Gamma_2$ are countable, we again conclude that $\Aut \Gamma$ is countable.
\end{proof}

\subsection{\boldmath Groups in class $\cC$}\label{sec.class-C}

Recall from \cite[Definition 3.1]{PV21}, that a group $\Gamma$ is said to belong to class $\cC$ if $\Gamma$ is nonamenable, weakly amenable, biexact and if every nontrivial element $g \in \Gamma \setminus \{e\}$ has an amenable centralizer $C_\Gamma(g)$. This class $\cC$ contains all free products $\Gamma = \Gamma_1 \ast \Gamma_2$ of amenable groups with $|\Gamma_1| \geq 2$ and $|\Gamma_2| \geq 3$.

\section{The structure of non icc group von Neumann algebras}

Let $\Lambda$ be a countable group. Denote by $\Lambda\fc$ the \emph{virtual center} of $\Lambda$, i.e.\ the subgroup of elements having a finite conjugacy class. By \cite[Satz 1]{Kan69} and Schur's theorem, the group von Neumann algebra $L(\Lambda)$ has a nonzero type~I direct summand if and only if $\Lambda$ is virtually isomorphic to an abelian group (in the sense of Definition \ref{def.virtual-iso}). By \cite[Satz 2]{Kan69}, the group von Neumann algebra $L(\Lambda)$ is of type~I if and only if $\Lambda$ is commensurable to an abelian group.

Note that a finite von Neumann algebra $M$ has a nonzero type~I direct summand if and only if $M \prec_M \cZ(M)$, while $M$ is of type~I if and only if $M \prec^f_M \cZ(M)$. In the following proposition, we prove a generalization of the results of \cite{Kan69}, characterizing when $L(\Lambda\fc) \prec \cZ(L(\Lambda))$. This proposition is one of the key ingredients in the proof of Theorem \ref{thm.main-A}, as we will use it to decompose the unknown group von Neumann algebra $L(\Lambda)$ as a direct integral of twisted group von Neumann algebras over its center.

As a further illustration, which is not needed in the rest of the paper, we deduce from Proposition \ref{prop.structure-non-icc} the results of \cite{Kan69}. Note that Proposition \ref{prop.structure-non-icc} also generalizes \cite[Proposition 6.1]{DV24} and that the methods to prove Proposition \ref{prop.structure-non-icc} are quite similar to the methods used in \cite{CFQT24} to analyze how a group von Neumann algebra decomposes as a direct integral over its center in a group-like manner.

\begin{proposition}\label{prop.structure-non-icc}
Let $\Lambda$ be a countable group with virtual center $\Lambda\fc$ and $z \in \cZ(L(\Lambda))$ a nonzero central projection. Denote by $(v(s))_{s \in \Lambda}$ the unitaries generating $L(\Lambda)$ and define $\Sigma = \{s \in \Lambda \mid v(s) z = z\}$.
\begin{enuma}
\item\label{prop.structure-non-icc.1} $\Sigma$ is a finite normal subgroup of $\Lambda$.
\item\label{prop.structure-non-icc.2} The following statements are equivalent.
    \begin{enumlist}
    \item $L(\Lambda\fc)z \prec^f_{L(\Lambda)} \cZ(L(\Lambda))$
    \item There exists a finite index subgroup $\Lambda_1 < \Lambda$ containing $\Sigma$ such that the group $\Lambda_0 = \Lambda_1 / \Sigma$ satisfies $\Lambda_{0,\text{\rm fc}} = \cZ(\Lambda_0)$.
    \end{enumlist}
\item \label{prop.structure-non-icc.3} If the conditions in (b) hold and if we write $\cZ(L(\Lambda)) z = L^\infty(X,\mu)$, there are measurable maps $d : X \to \N$ and $n : X \to \N$, a measurable field of finite index subgroups $\Lambda_x < \Lambda/\Lambda\fc$ and a measurable field of $2$-cocycles $\om_x \in H^2(\Lambda_x,\T)$ such that
    $$L(\Lambda) z \cong \int_X^\oplus (M_{d(x)n(x)}(\C) \ot L_{\om_x}(\Lambda_x)) \, d\mu(x) \; .$$
    Moreover, $n(x) = [\Lambda/\Lambda\fc:\Lambda_x]$ for a.e.\ $x \in X$ and in the direct integral decomposition
    $$L(\Lambda\fc)z \cong \int_X^\oplus A_x \, d\mu(x) \; ,$$
    we have that $A_x \cong M_{d(x)}(\C) \ot \C^{n(x)}$ for a.e.\ $x \in X$.
\end{enuma}
\end{proposition}
\begin{proof}
Throughout the proof, we denote by $\tau$ the canonical trace on $L(\Lambda)$.

(a) Since
$$\tau(z) = \tau(z^* z) = \sum_{s \in \Lambda} |\tau(v(s) z)|^2 \geq |\Sigma| \, |\tau(z)|^2 \; ,$$
it follows that $\Sigma$ is a finite subgroup of $\Lambda$. Since $z$ commutes with every $v(s)$, $s \in \Lambda$, we get that $\Sigma$ is a normal subgroup of $\Lambda$.

(b, ii $\Rightarrow$ i) Define $\Lambda_2 = \bigcap_{g \in \Lambda} g \Lambda_1 g^{-1}$. Since $\Lambda_1 < \Lambda$ has finite index and $\Sigma$ is a normal subgroup of $\Lambda$, we get that $\Lambda_2 < \Lambda$ is a finite index normal subgroup and $\Sigma < \Lambda_2$. Write $\Lambda_3 = \Lambda_2 / \Sigma$. Since $\Lambda_3 < \Lambda_0$ has finite index, we get that
$$\cZ(\Lambda_3) \subset \Lambda_{3,\text{\rm fc}} = \Lambda_{0,\text{\rm fc}} \cap \Lambda_3 = \cZ(\Lambda_0) \cap \Lambda_3 \subset \cZ(\Lambda_3) \; .$$
So, $\Lambda_{3,\text{\rm fc}} = \cZ(\Lambda_3)$. Replacing $\Lambda_1$ by $\Lambda_2$, we may thus assume from the start that $\Lambda_1$ is a normal subgroup of $\Lambda$.

Since $\Lambda_1 < \Lambda$ has finite index, also $\Lambda_{1,\text{\rm fc}} = \Lambda\fc \cap \Lambda_1 < \Lambda\fc$ has finite index. It follows that $L(\Lambda\fc) \prec^f L(\Lambda_{1,\text{\rm fc}})$ inside $L(\Lambda\fc)$ and thus, a fortiori, also inside $L(\Lambda)$.

Write $B = L(\Lambda_{1,\text{\rm fc}})z$. Since $z$ is a central projection in $L(\Lambda)$, we also get that $L(\Lambda\fc)z \prec^f B$ inside $L(\Lambda)$. Note that
$$B = L(\Lambda_{1,\text{\rm fc}})z = L(\Lambda_{0,\text{\rm fc}})z = L(\cZ(\Lambda_0))z$$
is abelian. Also, for every $s \in \Lambda_1$, we have that $v(s) z \in L(\Lambda_0) z$, so that $v(s)$ commutes with $B$. Therefore, conjugation by $v(s)$ defines an action of the finite group $\Lambda/\Lambda_1$ on the abelian von Neumann algebra $B$.

For any action of a finite group $H$ on an abelian von Neumann algebra $D$ with fixed point algebra $D^H$, we can find projections $z_k \in D$ that sum up to $1$ and satisfy $D z_k =  D^H z_k$ for all $k$. The fixed point algebra of the action of $\Lambda/\Lambda_1$ on $B$ equals $B \cap \cZ(L(\Lambda))z$. Since $\Lambda/\Lambda_1$ is finite, we thus conclude that $B \prec^f B \cap \cZ(L(\Lambda))z$ inside $B$ and thus, a fortiori, also inside $L(\Lambda)$. Since we already obtained that $L(\Lambda\fc)z \prec^f B$, we conclude that (a) holds.

(c) We prove that condition (i) in (b) implies the conclusion of (c). Write $M = L(\Lambda)$. Note that $\cZ(M) \subset L(\Lambda\fc)$. Writing $\cZ(M) z = L^\infty(X,\mu)$, we thus have well-defined direct integral decompositions
$$L(\Lambda\fc)z = \int_X^\oplus A_x \, d\mu(x) \quad\text{and}\quad Mz = \int_X^\oplus M_x \, d\mu(x) \; ,$$
in which $A_x \subset M_x$ are von Neumann subalgebras and $\mu$ is the measure on $X$ given by restricting $\tau$ to $\cZ(M)$. In particular, $\mu(X) = \tau(z)$. We also write $\tau(\cdot \, z) = \int^\oplus_X \tau_x \, d\mu(x)$, where for every $x \in X$, $\tau_x$ is a normal tracial state on $M_x$.

Since $L(\Lambda\fc)z \prec^f_M \cZ(M)$ and $\cZ(M)$ is abelian, the von Neumann algebra $L(\Lambda\fc)z$ is of type~I and $\cZ(L(\Lambda\fc))z \prec^f_M \cZ(M)$. Since also $\cZ(M) \subset \cZ(L(\Lambda\fc))$, Lemma \ref{lem.equal-under-q} provides a sequence of projections $q_n \in \cZ(L(\Lambda\fc))z$ such that $\cZ(L(\Lambda\fc)) q_n = \cZ(M) q_n$ for all $n$ and $\sum_n q_n = z$. Writing $q_n = \int^\oplus_X q_{n,x} \, d\mu(x)$, we find that for a.e.\ $x \in X$, $(q_{n,x})_n$ is a sequence of minimal projections in $\cZ(A_x)$ with $\sum_n q_{n,x} = 1$. So for a.e.\ $x \in X$, $\cZ(A_x)$ is discrete.

As explained above, $L(\Lambda\fc)z$ is of type~I. So, a.e.\ $A_x$ is of type I and we conclude that a.e.\ $A_x$ is a discrete von Neumann algebra.

Write $v(s) z = \int^\oplus_X v(s,x) \, d\mu(x)$ with $v(s,x) \in \cU(M_x)$. Since $(\Ad v(s))_{s \in \Lambda}$ is a $\tau$-preserving action on $L(\Lambda\fc)$ with fixed point algebra $\cZ(M)$, we find that for a.e.\ $x \in X$, $(\Ad v(s,x))_{s \in \Lambda}$ is an ergodic $\tau_x$-preserving action on $A_x$. Every discrete tracial von Neumann algebra that admits an ergodic trace preserving action is of the form $M_d(\C) \ot \C^n$. We thus find measurable maps $d : X \to \N$ and $n : X \to \N$ such that $A_x \cong M_{d(x)}(\C) \ot \C^{n(x)}$.

Choose a projection $r \in \cZ(L(\Lambda\fc))z$ such that, writing $r = \int^\oplus_X r_x \, d\mu(x)$, we have that $r_x \in \cZ(A_x)$ is a minimal central projection. We can then define
$$\Gamma_x = \{ s \in \Lambda \mid v(s,x) r_x = r_x v(s,x)\} \;\; , \quad\text{so that $[\Lambda:\Gamma_x] = n(x)$,}$$
and measurably choose unitaries $\rho(s,x) \in A_x r_x \cong M_{d(x)}(\C)$ for all $s \in \Gamma_x$ s.t.\ $\Ad v(s,x) r_x = \Ad \rho(s,x)$ on $A_x r_x$. By construction, $\Lambda\fc < \Gamma_x$. We define $\Lambda_x = \Gamma_x / \Lambda\fc$ and measurably choose lifts $\theta_x : \Lambda_x \to \Gamma_x$. We also choose a projection $p \in L(\Lambda\fc)r$ such that, writing $p = \int^\oplus_X p_x \, d\mu(x)$, we have that $p_x$ is a minimal projection in $A_x r_x$.

Write $\pi_x(s) = v(\theta_x(s),x) \rho(\theta_x(s),x)^* p_x$, for all $s \in \Lambda_x$. Whenever $s \in \Lambda \setminus \Lambda\fc$, we have that $E_{L(\Lambda\fc)}(v(s)) = 0$, so that $E_{A_x}(v(s,x)) = 0$ for a.e.\ $x \in X$. It follows that for a.e.\ $x \in X$, we have that $\tau_x(\pi_x(s)) = 0$ for all $s \in \Lambda_x \setminus \{e\}$. By construction, $\pi_x$ is a projective representation of $\Lambda_x$. So, $\pi_x$ induces an isomorphism $p_x M_x p_x \cong L_{\om_x}(\Lambda_x)$, where $\om_x \in H^2(\Lambda_x,\T)$ is a measurable field of $2$-cocycles. Then also $M_x \cong M_{d(x)n(x)}(\C) \ot L_{\om_x}(\Lambda_x)$ so that (c) is proven.

(b, i $\Rightarrow$ ii) By the paragraphs above, the conclusions of (c) hold. For every $n \in \N$ and $n$-tuple $s \in \Lambda\fc \times \cdots \times \Lambda\fc$, define the measurable subset $\cU(n,s) \subset X$ by
$$\cU(n,s) = \bigl\{ x \in X \bigm| d(x)^2 n(x) = n \quad\text{and}\quad \det\bigl(\tau_x(v(s_i,x)^* v(s_j,x))_{i,j = 1,\ldots,n}\bigr) \neq 0 \bigr\} \; .$$
By definition, $\cU(n,s)$ consists of the elements $x \in X$ such that $(v(s_i,x))_{i=1,\ldots,n}$ is a vector space basis for $A_x$. Since a.e.\ $A_x$ is finite dimensional, we have that
$$X = \bigcup_{n \in \N} \bigcup_{s \in \Lambda\fc^n} \cU(n,s) \; ,$$
up to measure zero. Writing $\Lambda\fc$ as the union of an increasing sequence of finite subsets $\cF_n \subset \Lambda\fc$, we thus find an increasing sequence of measurable subsets $\cU_n \subset X$ such that $X = \bigcup_n \cU_n$ up to measure zero and $A_x$ is linearly spanned by $v(s,x)_{s \in \cF_n}$ for all $x \in \cU_n$.

Define the central projections $z_n \in \cZ(M)z$ by $z_n = 1_{\cU_n}$. Note that the sequence $z_n$ is increasing and $z_n \to z$ strongly. Then $\Sigma_n = \{s \in \Lambda \mid v(s) z_n = z_n\}$ is a decreasing sequence of finite normal subgroups of $\Lambda$ with $\bigcap_n \Sigma_n = \Sigma$. We can thus fix $n_0 \in \N$ such that $\Sigma_{n_0} = \Sigma$. Since $\cF_{n_0}$ is a finite subset of $\Lambda\fc$, the centralizer $\Lambda_2 := C_\Lambda(\cF_{n_0})$ of $\cF_{n_0}$ inside $\Lambda$ has finite index in $\Lambda$. Define $\Lambda_1 := \Sigma \cdot \Lambda_2$ and $\Lambda_0 := \Lambda_1 / \Sigma$. We prove that $\Lambda_{0,\text{\rm fc}} = \cZ(\Lambda_0)$.

Note that $\Lambda_0 = \Lambda_2 / (\Lambda_2 \cap \Sigma)$. The inclusion $\cZ(\Lambda_0) \subset \Lambda_{0,\text{\rm fc}}$ is trivial. Conversely, take $h \in \Lambda_{0,\text{\rm fc}}$. Write $h = g (\Lambda_2 \cap \Sigma)$ with $g \in \Lambda_2$. Since $\Sigma < \Lambda$ is a finite normal subgroup and $h \in \Lambda_{0,\text{\rm fc}}$, we have that $g \in \Lambda\fc$. Choose an arbitrary $k \in \Lambda_2$. Define the commutator $c = gkg^{-1}k^{-1}$. We have to prove that $c \in \Sigma$. Since $g \in \Lambda\fc$, we have $v(g,x) \in A_x$ for a.e.\ $x \in X$. It follows that for a.e.\ $x \in \cU_{n_0}$, we have that $v(g,x) \in \operatorname{span}\{v(s,x) \mid s \in \cF_{n_0}\}$. Since $k \in \Lambda_2$, the unitaries $v(k,x)$ commute with $v(s,x)$ for all $s \in \cF_{n_0}$. So, $v(k,x)$ and $v(g,x)$ commute for a.e.\ $x \in \cU_{n_0}$. This means that $v(c) z_{n_0} = z_{n_0}$. Thus, $c \in \Sigma_{n_0} = \Sigma$ and the equality $\Lambda_{0,\text{\rm fc}} = \cZ(\Lambda_0)$ follows.
\end{proof}

\begin{corollary}[\makebox{\cite[S\"{a}tze 1 \& 2]{Kan69}}]\label{cor.Kaniuth}
Let $\Lambda$ be a countable group.
\begin{enuma}
\item $L(\Lambda)$ has a nonzero type~I direct summand if and only if $\Lambda$ is virtually isomorphic to an abelian group in the sense of Definition \ref{def.virtual-iso}.
\item $L(\Lambda)$ is of type~I if and only if $\Lambda$ is commensurable to an abelian group.
\end{enuma}
\end{corollary}
\begin{proof}
Write $M = L(\Lambda)$. First note that if $\Lambda\fc < \Lambda$ has infinite index, then $M \not\prec_M L(\Lambda\fc)$. Since $\cZ(M) \subset L(\Lambda\fc)$, we certainly get that $M \not\prec_M \cZ(M)$, so that $M$ is of type II. We may thus assume that $\Lambda\fc < \Lambda$ has finite index.

(a, $\Rightarrow$) Since $M$ has a nonzero type~I direct summand, we can take a nonzero central projection $z \in \cZ(M)$ such that $Mz \prec^f_M \cZ(M)$. A fortiori, $L(\Lambda\fc)z \prec^f_M \cZ(M)$. By Proposition \ref{prop.structure-non-icc.2}, $\Lambda\fc$ is virtually isomorphic to an abelian group. Since $\Lambda\fc < \Lambda$ has finite index, also $\Lambda$ is virtually isomorphic to an abelian group.

(a, $\Leftarrow$) Take a finite index subgroup $\Lambda_1 < \Lambda$ and a finite normal subgroup $\Sigma_1 \lhd \Lambda_1$ such that $\Lambda_1 / \Sigma_1$ is abelian. Let $g_1,\ldots,g_k \in \Lambda$ be representatives for $\Lambda/\Lambda_1$ with $g_1 = e$. Define $\Lambda_2 = \bigcap_{i=1}^k g_i \Lambda_1 g_i^{-1}$ and $\Sigma_2 = \Sigma_1 \cap \Lambda_2$. Then $\Lambda_2 < \Lambda$ still has finite index, $\Sigma_2 < \Lambda_2$ is a finite normal subgroup and $\Lambda_2 / \Sigma_2$ is abelian. Now $\Lambda_2$ is also a normal subgroup of $\Lambda$. Let $h_1,\ldots,h_n \in \Lambda$ be representatives for $\Lambda/\Lambda_2$ and define $\al_i \in \Aut \Lambda_2$ by $\al_i = \Ad h_i$. Write $\Sigma = \al_1(\Sigma_2) \cdots \al_n(\Sigma_2)$. Then $\Sigma < \Lambda_2$ is a finite normal subgroup, $\Lambda_2 / \Sigma$ is still abelian and now, $\Sigma$ is normal in $\Lambda$. Define $z = |\Sigma|^{-1} \sum_{s \in \Sigma} v_s$.

Then $z$ is a projection in the center of $L(\Lambda)$ and $L(\Lambda) z = L(\Lambda/\Sigma)$. Since $\Lambda_2 / \Sigma < \Lambda/\Sigma$ has finite index, we have that $L(\Lambda/\Sigma) \prec^f L(\Lambda_2/\Sigma)$. Since $\Lambda_2/\Sigma$ is abelian, it follows that $L(\Lambda/\Sigma)$ is of type~I. So, $L(\Lambda)z$ is of type I.

(b, $\Rightarrow$) Since $M$ is of type~I, we have that $M \prec^f_M \cZ(M)$. Applying Proposition \ref{prop.structure-non-icc.2} with $z=1$, we find a finite index subgroup $\Lambda_1 < \Lambda\fc$ such that $\Lambda_{1,\text{\rm fc}}$ is abelian. Since $\Lambda\fc < \Lambda$ has finite index, also $\Lambda_1 < \Lambda$ has finite index. It follows that $\Lambda_{1,\text{\rm fc}} = \Lambda\fc \cap \Lambda_1$, which has finite index in $\Lambda$. So, $\Lambda$ has an abelian subgroup of finite index.

(b, $\Leftarrow$) We find an abelian finite index subgroup $\Lambda_1<\Lambda$. Since $L(\Lambda) \prec^f L(\Lambda_1)$ and $L(\Lambda_1)$ is abelian, it follows that $L(\Lambda)$ is of type I.
\end{proof}

\section{\boldmath Isomorphism W$^*$-superrigidity}\label{sec.generic-iso-superrigid}

There are by now, starting with \cite{IPV10}, several classes of W$^*$-superrigid icc groups $G$. In \cite{DV24}, the first families of groups were shown to be W$^*$-superrigid even when we allow arbitrary $2$-cocycle twists. We provide in this section sufficient conditions to transfer such cocycle W$^*$-superrigidity of an icc group $G$ to W$^*$-superrigidity of certain central extensions
\begin{equation}\label{eq.central-extension}
0 \to C \to \Gtil \to G \to e \quad\text{with corresponding $2$-cocycle $\Om \in H^2(G,C)$.}
\end{equation}
In Lemma \ref{lem.about-all-our-conditions}, we prove that the groups in Theorem \ref{thm.main-A} satisfy these sufficient conditions.

We fix a countable group $G$ and a countable abelian group $C$. Since we cannot produce W$^*$-superrigidity out of thin air, the first condition is the following.
\begin{enumlist}
\item\label{cond.iso.1} The group $G$ satisfies the following cocycle W$^*$-superrigidity property: if $\Lambda$ is any countable group, $\mu \in H^2(G,\T)$ and $\om \in H^2(\Lambda,\T)$ are arbitrary $2$-cocycles and $p \in L_\mu(G)$ is a nonzero projection such that $p L_\mu(G) p \cong L_\om(\Lambda)$, then $p=1$ and there exists a group isomorphism $\delta : G \to \Lambda$ with $\om \circ \delta = \mu$ in $H^2(G,\T)$.
\end{enumlist}

In \cite[Theorem 6.15]{DV24}, left-right wreath product groups $G = (\Z/2\Z)^{(\Gamma)} \rtimes (\Gamma \times \Gamma)$ for $\Gamma$ in class $\cC$ (see Section \ref{sec.class-C}) were shown to satisfy condition \ref{cond.iso.1}.

In Theorem \ref{thm.inherit-iso-superrigidity}, we prove a rigidity result for \emph{arbitrary} central extensions $\Gtil$ of an icc group $G$ by a torsion free abelian group $C$ with associated $2$-cocycle $\Om \in H^2(G,C)$, when $G$ satisfies condition \ref{cond.iso.1} and the more technical conditions \ref{ncond.iso.2} and \ref{ncond.iso.3} below. We describe all countable groups $\Lambdatil$ such that $L(\Lambdatil) \cong L(\Gtil)$. These groups $\Lambdatil$ must themselves be central extensions of $G$ by an abelian $D$ such that the associated $2$-cocycle $\Om' \in H^2(G,D)$ is in a specific way ``equivalent'' to $\Om$ (cf.\ Definition \ref{def.equiv-2-cocycle}). Under extra assumptions on $H^2(G,C)$ and $\Om$, this forces $D \cong C$ and $\Om' = \Om$, so that W$^*$-superrigidity follows.

Consider a central extension as in \eqref{eq.central-extension}, and assume that $G$ is icc. We claim that $\Gtil\fc = \cZ(\Gtil) = C$. The inclusions $C \subset \cZ(\Gtil) \subset \Gtil\fc$ are trivial. Conversely, assume that $g \in \Gtil$ has a finite conjugacy class. Then the image of $g$ in $G$ also has a finite conjugacy class, and is thus trivial, so that $g \in C$. As explained in \eqref{eq.concrete-direct-integral}, we get the direct integral decomposition
\begin{equation}\label{eq.central-decomp}
L(\Gtil) = \int^\oplus_{\Chat} L_{\mu \circ \Om}(G) \, d\mu \; ,
\end{equation}
where we integrate w.r.t.\ the Haar measure on $\Chat$. So as a group, $\Gtil$ is determined by the $2$-cocycle $\Om \in H^2(G,C)$, but as a group von Neumann algebra, $L(\Gtil)$ only ``sees'' $\mu \circ \Om$ for $\mu \in \Chat$. It is thus natural to consider the following equivalences between $2$-cocycles.

\begin{definition}\label{def.equiv-2-cocycle}
Let $G$ be a countable group and $C$, $D$ countable abelian groups. Let $\Om \in H^2(G,C)$ and $\Om' \in H^2(G,D)$.
\begin{enuma}
\item We write $\Om \approx \Om'$ if $C=D$ and $\om \circ \Om = \om \circ \Om'$ in $H^2(G,\T)$ for all $\om \in \Chat$.
\item We write $\Om \sim^t \Om'$ for $t > 0$ if there exists an isomorphism $\zeta : C_0 \to D_0$ between subgroups $C_0 < C$ and $D_0 < D$ such that $[C:C_0] t = [D:D_0]$ and
\begin{equation}\label{eq.cocycle-equal}
\mu \circ \Om = \om \circ \Om' \quad\text{in $H^2(G,\T)$ whenever $\mu \in \Chat$, $\om \in \Dhat$ satisfy $\om \circ \zeta = \mu|_{C_0}$.}
\end{equation}
\end{enuma}
\end{definition}

Note that in (b), we are not assuming that $C_0 < C$ and $D_0 < D$ have finite index. Also note that $\approx$ is an equivalence relation on $H^2(G,C)$. In Lemma \ref{lem.prop-of-sim-t}, we observe that if $\Om \sim^t \Om'$ and $\Om' \sim^s \Om\dpr$, then $\Om \sim^{ts} \Om\dpr$ and $\Om' \sim^{t^{-1}} \Om$.

We need the following extra conditions on $G$, which are of a technical nature and are needed to deal with the more subtle operator algebraic issues in the proof of Theorem \ref{thm.inherit-iso-superrigidity}.

\begin{enumlist}[resume]
\item\label{ncond.iso.2} The automorphism group $\Aut G$ is countable.
\end{enumlist}

Recall that a subgroup $G_0 < G$ of a discrete group is said to be \emph{co-amenable} if $G/G_0$ admits a $G$-invariant mean.

\begin{enumlist}[resume]
\item\label{ncond.iso.3} There is a family of non co-amenable subgroups $(G_i)_{i \in I}$ of $G$ for which the following holds: if $(N,\tau)$ is any tracial von Neumann algebra, $p \in N \ovt L(G)$ a nonzero projection, $B \subset p(N \ovt L(G))p$ an amenable von Neumann subalgebra and $\cG \subset \cN_{p(N \ovt L(G))p}(B)$ a subgroup such that $(\Ad v)_{v \in \cG}$ has compact closure in $\Aut B$ and $P := (B \cup \cG)\dpr$ is strongly nonamenable relative to $N \ovt L(G_i)$ for all $i \in I$, then $B \prec^f N \ot 1$.
\end{enumlist}

Recall from Section \ref{sec.second-cohomology-of-groups} the group homomorphism
\begin{equation}\label{eq.hom-Theta}
\Theta : H^2(G,C) \to \Hom(\Chat,H^2(G,\T)) : (\Theta(\Om))(\mu) = \mu \circ \Om \quad\text{for all $\Om \in H^2(G,C)$, $\mu \in \Chat$,}
\end{equation}
where $\Hom(\Chat,H^2(G,\T))$ denotes the abelian group of continuous homomorphisms between the compact groups $\Chat$ and $H^2(G,\T)$. Also recall from Section \ref{sec.second-cohomology-of-groups} that we denote by $\Ext^1(G\ab,C)$ the group of abelian extensions of the abelianization $G\ab$ by $C$.

For every countable group $\Gamma$, we denote by $\tau_\Gamma$ the canonical tracial state on $L(\Gamma)$.

\begin{theorem}\label{thm.inherit-iso-superrigidity}
Let $0 \to C \to \Gtil \to G \to e$ be any central extension of countable groups with associated $2$-cocycle $\Om \in H^2(G,C)$.

Define $\cH_{\sim^t}$, resp.\ $\cH_{\approx}$, as the class of countable groups that are isomorphic to a central extension of $G$ by an abelian $D$ and $\Om' \in H^2(G,D)$ satisfying $\Om \sim^t \Om'$, resp.\ $\Om \approx \Om'$.

Let $\Lambdatil$ be any countable group and $t \in (0,1]$.
\begin{enuma}
\item\label{thm.inherit-iso-superrigidity.1} If $\Lambdatil$ belongs to $\cH_{\sim^t}$, there exists a projection $p \in L(\Gtil)$ with $\tau_{\Gtil}(p)=t$ and $p L(\Gtil) p \cong L(\Lambdatil)$.
\end{enuma}
Assume that $C$ is torsion free.
\begin{enuma}[resume]
\item\label{thm.inherit-iso-superrigidity.2} If $G$ is icc and satisfies conditions \ref{cond.iso.1}, \ref{ncond.iso.2} and \ref{ncond.iso.3}, then the converse of \ref{thm.inherit-iso-superrigidity.1} holds: $p L(\Gtil) p \cong L(\Lambdatil)$ for some projection $p \in L(\Gtil)$ with $\tau_{\Gtil}(p)=t$ iff $\Lambdatil$ belongs to $\cH_{\sim^t}$.

\item\label{thm.inherit-iso-superrigidity.3} If $\Theta(\Om) \in \Hom(\Chat,H^2(G,\T))$ is faithful and $\Lambdatil \in \cH_{\sim^t}$, then $t = 1$ and $\Lambdatil \in \cH_\approx$.

\item\label{thm.inherit-iso-superrigidity.4} If $\Ext^1(G\ab,C) = 0$, then every group in $\cH_\approx$ is isomorphic to $\Gtil$.
\end{enuma}
\end{theorem}

So when all conditions \ref{cond.iso.1}, \ref{ncond.iso.2}, \ref{ncond.iso.3}, the faithfulness of $\Theta(\Om)$ and the vanishing of $\Ext^1(G\ab,C)$ hold, the group $\Gtil$ is W$^*$-superrigid: we have $p L(\Gtil) p \cong L(\Lambdatil)$ if and only if $p=1$ and $\Gtil \cong \Lambdatil$.

We can rephrase Theorem \ref{thm.inherit-iso-superrigidity} in the following way. The groups that are isomorphic to resp.\ $\Gtil$, a group in $\cH_\approx$, a group in $\cH_{\sim^1}$, and the groups whose group von Neumann algebra is isomorphic to $L(\Gtil)$, form larger and larger families of groups, and the rigidity results \ref{thm.inherit-iso-superrigidity.2}-\ref{thm.inherit-iso-superrigidity.4} say which of these families are equal.

As we explain in Lemma \ref{lem.extend-characters-C}, the faithfulness of $\Theta(\Om)$ is equivalent to the condition that every group homomorphism $\Gtil \to \T$ is equal to $1$ on $C$. In Proposition \ref{prop.new-no-go-result-1.1}, we explain that the faithfulness of $\Theta(\Om)$ is a necessary assumption for the validity of Theorem \ref{thm.inherit-iso-superrigidity.3}. In Proposition \ref{prop.new-no-go-result-3.2}, we show that a condition as $\Ext^1(G\ab,C) = 0$ in Theorem \ref{thm.inherit-iso-superrigidity.4} is essential: we give examples of groups $\Lambdatil$ in $\cH_\approx$ that are not isomorphic to $\Gtil$.

\subsection{Proof of Theorem \ref{thm.inherit-iso-superrigidity}}

Before proving Theorem \ref{thm.inherit-iso-superrigidity}, we need the following lemma that will also be used in Section \ref{sec.generic-virtual-iso-superrigid}.

\begin{lemma}\label{lem.intertwine-virtual-center-to-center}
Let $G$ be any countable group satisfying condition \ref{ncond.iso.3}. Let $C$ be any countable abelian group and $e \to C \to \Gtil \to G \to e$ any central extension. Let $\Lambdatil$ be any countable group and $\bim{L(\Lambdatil)}{H}{L(\Gtil)}$ a nonzero bifinite bimodule with left support $z \in \cZ(L(\Lambdatil))$.

Then inside $L(\Lambdatil)$, we have $L(\Lambdatil\fc)z \prec^f \cZ(L(\Lambdatil))$.
\end{lemma}
\begin{proof}
Write $P = L(\Lambdatil)$ and $M = L(\Gtil)$. Take a projection $p \in M_n(\C) \ot M$ and a unital normal faithful $*$-homomorphism $\vphi : Pz \to p (M_n(\C) \ot M)p$ such that $\bim{P}{H}{M} \cong \bim{\vphi(P)}{p(\C^n \ot L^2(M))}{M}$. Denote by $\pi : \Gtil \to G$ the quotient homomorphism and denote by $(w_h)_{h \in \Gtil}$ and $(u_g)_{g \in G}$ the canonical unitaries generating $L(\Gtil)$, resp.\ $L(G)$. Define the normal $*$-homomorphism
$$\psi : L(\Gtil) \to L(\Gtil) \ovt L(G) : \psi(w_h) = w_h \ot u_{\pi(h)} \quad\text{for all $h \in \Gtil$.}$$
Take the subgroups $(G_i)_{i \in I}$ given by condition \ref{ncond.iso.3} and write $\Gtil_i = \pi^{-1}(G_i)$. Denote by $(v_s)_{s \in \Lambdatil}$ the unitaries generating $L(\Lambdatil)$. Write $B = L(\Lambdatil\fc)z$.

Since $\Lambdatil\fc$ is an amenable group, $B$ is amenable. Since $\Lambdatil\fc$ consists of the elements of $\Lambdatil$ having a finite conjugacy class, we can write $\Lambdatil\fc$ as the union of finite disjoint subsets $W_k \subset \Lambdatil\fc$ that are conjugation invariant: $s W_k s^{-1} = W_k$ for all $s \in \Lambdatil$. We can then view the unitary representation $(\Ad v_s)_{s \in \Lambdatil}$ on $\ell^2(\Lambdatil\fc)$ as the direct sum of the finite-dimensional representations on $\ell^2(W_k)$. It follows that the closure of $(\Ad v_s z)_{s \in \Lambdatil}$ in $\Aut(B)$ is compact.

Note that the unitaries $(v_s z)_{s \in \Lambdatil}$ generate $P z$. We claim that for all $i \in I$, $\vphi(Pz)$ is strongly nonamenable relative to $L(\Gtil_i)$ inside $M=L(\Gtil)$. Assume the contrary and take $i \in I$ and a nonzero projection $q \in \vphi(Pz)' \cap p(M_n(\C) \ot M)p$ such that $\vphi(Pz)q$ is amenable relative to $L(\Gtil_i)$ inside $M$. Define $r \in \cZ(M)$ such that $1 \ot r$ is the central support of the projection $q \in M_n(\C) \ot M$.

Note that $q \cdot H$ is finitely generated as a left Hilbert $P$-module. So, $(M_{1,n}(\C) \ot r L^2(M))q$ is an $Mr$-$\vphi(Pz)q$-bimodule that is finitely generated as a right Hilbert $\vphi(Pz)q$-module. By definition of the projection $r$, the left action of $Mr$ is faithful. By the definition of the intertwining relation $\prec$, we get that $M r \prec^f \vphi(Pz)q$ inside $M$. A fortiori, $Mr$ is amenable relative to $\vphi(Pz)q$ inside $M$. By \cite[Proposition 2.4(3)]{OP07}, we get that $Mr$ is amenable relative to $L(\Gtil_i)$ inside $M$. The basic construction $\langle M, e_{L(\Gtil_i)}\rangle$ is defined as the commutant of the right action of $L(\Gtil_i)$ on $\ell^2(\Gtil)$ and contains in this way $\ell^\infty(\Gtil/\Gtil_i)$ as multiplication operators. We thus find a $G$-invariant state on $\ell^\infty(G/G_i) = \ell^\infty(\Gtil/\Gtil_i)$, contradicting the assumption that $G_i$ is not co-amenable in $G$. So the claim is proven.

We now make use twice of \cite[Lemma 2.2]{DV24} that allows to transfer relative amenability and intertwining bimodules along nondegenerate commuting squares (see also the discussion preceding \cite[Lemma 2.2]{DV24} for the necessary terminology). Since
$$\begin{array}{ccc}
L(\Gtil) \ovt L(G_i) & \subset & L(\Gtil) \ovt L(G) \\ \cup & & \cup \\ \psi(L(\Gtil_i)) & \subset & \psi(L(\Gtil))
\end{array}
$$
is a nondegenerate commuting square and $\vphi(Pz)$ is strongly nonamenable relative to $L(\Gtil_i)$, it follows from \cite[Lemma 2.2]{DV24} that $(\id \ot \psi)\vphi(Pz)$ is strongly nonamenable relative to $M \ovt L(G_i)$ inside $M \ovt L(G)$. Using condition \ref{ncond.iso.3}, we conclude that $(\id \ot \psi)\vphi(B) \prec^f M \ot 1$ inside $M \ovt L(G)$. Since also
$$\begin{array}{ccc}
L(\Gtil) \ot 1 & \subset & L(\Gtil) \ovt L(G) \\ \cup & & \cup \\ \psi(L(C)) & \subset & \psi(L(\Gtil))
\end{array}$$
is a nondegenerate commuting square, by a second application of \cite[Lemma 2.2]{DV24}, we get that $\vphi(B) \prec^f L(C)$ inside $L(\Gtil)$. Since $L(C) \subset \cZ(L(\Gtil))$, it follows from Lemma \ref{lem.transfer-embedding-in-center} that $B \prec^f \cZ(L(\Lambdatil))$ inside $L(\Lambdatil)$.
\end{proof}

\begin{proof}[{Proof of Theorem \ref{thm.inherit-iso-superrigidity.1}}]
Take a central extension $\Lambdatil$ of $G$ by $D$ associated with $\Om' \in H^2(G,D)$ such that $\Om \sim^t \Om'$. Take an isomorphism $\zeta : C_0 \to D_0$ between subgroups $C_0 < C$ and $D_0 < D$ such that $[C:C_0] t = [D:D_0]$ and such that \eqref{eq.cocycle-equal} holds. In particular, $\mu \circ \Om = 1$ in $H^2(G,\T)$ for all $\mu \in \Chat$ that are equal to $1$ on $C_0$. We thus get a well-defined continuous group homomorphism $\lambda_0 : \widehat{C_0} \to H^2(G,\T)$ satisfying $\lambda_0(\mu|_{C_0}) = \mu \circ \Om$ for all $\mu \in \Chat$.

Choose a Borel lift $\theta_0 : \widehat{C_0} \to \Chat$ for the quotient homomorphism $\Chat \to \widehat{C_0} : \mu \mapsto \mu|_{C_0}$. Denote by $S < \Chat$ the kernel of this quotient homomorphism. Then, $\theta : S \times \widehat{C_0} \to \Chat : \theta(s,\mu) = s \theta_0(\mu)$ is a bijective Borel map. Since $\theta_*(\text{Haar}_S \times \text{Haar}_{\widehat{C_0}})$ is a translation invariant probability measure on $\Chat$, it must be equal to $\text{Haar}_{\Chat}$. Since $\theta(s,\mu) \circ \Om = \lambda_0(\mu)$, it then follows from \eqref{eq.concrete-direct-integral} that $\theta$ induces a trace preserving isomorphism
\begin{align*}
(L(\Gtil),\tau_{\Gtil}) &\cong \int_{\Chat}^\oplus (L_{\mu \circ \Om}(G),\tau_G) \, d\mu \cong \int_{S \times \widehat{C_0}}^\oplus (L_{\theta(s,\mu) \circ \Om}(G),\tau_G) \, d(s,\mu) \\
&\cong (L^\infty(S),\text{Haar}_S) \ovt \int_{\widehat{C_0}}^\oplus (L_{\lambda_0(\mu)}(G),\tau_G) \, d\mu \; ,
\end{align*}
where $\Chat$, $S$ and $\widehat{C_0}$ are equipped with their Haar probability measures. If $[C:C_0] < +\infty$, the group $S$ is finite and has order $[C:C_0]$, so that $(L^\infty(S),\text{Haar}_S)$ is isomorphic with $A_1 = \C^{[C:C_0]}$ with the uniform tracial state. If $[C:C_0]=+\infty$, we get that $(S,\text{Haar}_S)$ is a nonatomic standard probability space, so that $(L^\infty(S),\text{Haar}_S)$ is isomorphic with $A_1=L^\infty([0,1])$ with the Lebesgue measure. We have thus found a trace preserving isomorphism
\begin{equation}\label{eq.my-iso-for-Gtil}
L(\Gtil) \cong A_1 \ovt Q_1 \quad\text{where}\quad Q_1 = \int_{\widehat{C_0}}^\oplus L_{\lambda_0(\mu)}(G) \, d\mu \; ,
\end{equation}
and where $Q_1$ is equipped with the trace given by integration over the Haar measure of $\widehat{C_0}$ of the canonical tracial states on $L_{\lambda_0(\mu)}(G)$.

Since $\om \circ \Om' = 1$ in $H^2(G,\T)$ for all $\om \in \Dhat$ that are equal to $1$ on $D_0$, we similarly define $\rho_0 : \widehat{D_0} \to H^2(G,\T)$ such that $\rho_0(\om|_{D_0}) = \om \circ \Om'$ for all $\om \in \Dhat$. As in \eqref{eq.my-iso-for-Gtil}, we have that
$$L(\Lambdatil) \cong A_2 \ovt Q_2 \quad\text{where}\quad Q_2 = \int_{\widehat{D_0}}^\oplus L_{\rho_0(\om)}(G) \, d\om \; ,$$
where $A_2 = \C^{[D:D_0]}$ when $[D:D_0]<+\infty$ and $A_2 = A_1$ otherwise. Define the continuous isomorphism $\gamma : \widehat{C_0} \to \widehat{D_0}$ by $\gamma(\mu) = \mu \circ \zeta^{-1}$. By \eqref{eq.cocycle-equal}, $\rho_0(\gamma(\mu)) = \lambda_0(\mu)$ for all $\mu \in \widehat{C_0}$. So, $\gamma$ induces an isomorphism $Q_1 \cong Q_2$.

Note that we are assuming that $t \in (0,1]$. When $[C:C_0] = +\infty$, also $[D:D_0] = +\infty$ and we can choose a projection $p \in A_1$ with trace $t$. Since $A_1 p \cong A_2$, it follows that $q L(\Gtil) q \cong L(\Lambdatil)$, where $q = p \ot 1$. When $[C:C_0] < + \infty$ and $[D:D_0]  = [C:C_0] t$, we can take a projection $p \in A_1 = \C^{[C:C_0]}$ of trace $t$ and conclude again that $A_1 p \cong A_2$, so that $q L(\Gtil) q \cong L(\Lambdatil)$, where $q = p \ot 1$.
\end{proof}

\begin{proof}[{Proof of Theorem \ref{thm.inherit-iso-superrigidity.2}}]
Assume that $p \in L(\Gtil)$ is a projection such that $p L(\Gtil) p \cong L(\Lambdatil)$ and $\tau_{\Gtil}(p) = t$. Write $P = L(\Lambdatil)$ and $M = L(\Gtil)$, and assume that $\al : P \to pMp$ is a $*$-isomorphism. Then $\bim{\al(P)}{p L^2(M)}{M}$ is a bifinite $P$-$M$-bimodule and the left $P$-action is faithful. By condition \ref{ncond.iso.3} and Lemma \ref{lem.intertwine-virtual-center-to-center}, we find that $L(\Lambdatil\fc) \prec^f_P \cZ(P)$. We define $(X,\eta)$ such that $L^\infty(X,\eta) = \cZ(P)$. We write $\Lambda = \Lambdatil/\Lambdatil\fc$. By Proposition \ref{prop.structure-non-icc.3}, we find measurable fields of finite index subgroups $\Lambda_x < \Lambda$ and $2$-cocycles $\om_x \in H^2(\Lambda_x,\T)$ such that
$$P = \int^\oplus_X (M_{d(x)n(x)}(\C) \ot L_{\om_x}(\Lambda_x)) \, d\eta(x) \quad\text{and}\quad L(\Lambdatil\fc) \cong \int^\oplus_X (M_{d(x)}(\C) \ot \C^{n(x)}) \, d\eta(x) \; ,$$
with $n(x) = [\Lambda:\Lambda_x]$. Since $G$ is icc, we have $\Gtil\fc = \cZ(\Gtil) = C$ and get the direct integral decomposition \eqref{eq.central-decomp}. Since the direct integral decomposition of a von Neumann algebra over its center is unique (see e.g.\ \cite[Theorem IV.8.23]{Tak79}), the isomorphism $\al: P \to pMp$ gives rise to a nonnegligible measurable subset $\cU \subset \Chat$, a measurable family of nonzero projections $(p_\mu)_{\mu \in \cU}$ in $L_{\mu \circ \Om}(G)$ and a nonsingular isomorphism $\Delta : \cU \to X$ such that $p = \int_\cU^\oplus p_\mu \, d\mu$ and
$$p_\mu L_{\mu \circ \Om}(G) p_\mu \cong M_{d(x)n(x)}(\C) \ot L_{\om_x}(\Lambda_x) \quad\text{with $x = \Delta(\mu)$, for a.e.\ $\mu \in \cU$.}$$
By condition \ref{cond.iso.1}, we find that for a.e.\ $\mu \in \cU$ and $x = \Delta(\mu)$, we have $p_\mu = 1$, $d(x)=n(x)=1$ and $(G,\mu \circ \om) \cong (\Lambda_x,\om_x)$. Since $[\Lambda:\Lambda_x] = n(x)$, we also conclude that $\Lambda_x = \Lambda$. Since $p_\mu = 1$ for a.e.\ $\mu \in \cU$, in the direct integral decomposition \eqref{eq.central-decomp}, the projection $p$ corresponds to $1_\cU$. Since the identification \eqref{eq.central-decomp} is trace preserving, $\tau_{\Gtil}(p)$ equals the Haar measure of $\cU$.

So for a.e.\ $x \in (X,\eta)$, we have $d(x) = n(x) = 1$. This means that $L(\Lambdatil\fc) = L^\infty(X,\eta) = \cZ(P)$. We conclude that $\Lambdatil\fc = \cZ(\Lambdatil)$. We denote $D = \cZ(\Lambdatil)$, so that $\cZ(P) = L(D)$, and we identify $(X,\eta)$ with $\Dhat$ equipped with its Haar probability measure. Since $D = \cZ(\Lambdatil)$, we also find the central extension $0 \to D \to \Lambdatil \to \Lambda \to e$. We denote by $\Om' \in H^2(\Lambda,D)$ the corresponding $2$-cocycle. We can then rewrite the direct integral decomposition for $P$ as
$$P = \int^\oplus_{\Dhat} L_{\om \circ \Om'}(\Lambda) \, d\om$$
and obtain the nonsingular isomorphism $\Delta : \cU \to \Dhat$ such that for a.e.\ $\mu \in \cU$, there exists an isomorphism $\delta_\mu : G \to \Lambda$ satisfying $\Delta(\mu) \circ \Om' \circ \delta_\mu = \mu \circ \Om$ in $H^2(G,\T)$.

By condition \ref{ncond.iso.2} we can enumerate all isomorphisms $\delta_n : G \to \Lambda$. Define the measurable sets $\cU_n \subset \cU$ by
$$\cU_n = \bigl\{ \mu \in \cU \bigm| \Delta(\mu) \circ \Om' \circ \delta_n = \mu \circ \Om \;\;\text{in $H^2(G,\T)$}\; \bigr\} \; .$$
Then, $\cU = \bigcup_n \cU_n$ up to measure zero. By removing the indices for which $\cU_n$ has measure zero and by making the $\cU_n$ smaller, we find a sequence of isomorphisms $\delta_n : G \to \Lambda$ and disjoint nonnegligible measurable subsets $\cU_n \subset \cU$ such that $\Delta(\mu) \circ \Om' \circ \delta_n = \mu \circ \Om$ in $H^2(G,\T)$ for all $\mu \in \cU_n$ and $\cU = \bigcup_n \cU_n$ up to measure zero.

Define the continuous group homomorphisms
$$\lambda: \Chat \to H^2(G,\T) : \lambda(\mu) = \mu \circ \Om \quad\text{and}\quad \rho : \Dhat \to H^2(\Lambda,\T) : \rho(\om) = \om \circ \Om' \; .$$
Define the subgroups $C_0 < C$ and $D_0 < D$ such that the kernel of $\lambda$ equals $\widehat{C/C_0}$ and the kernel of $\rho$ equals $\widehat{D/D_0}$. Define the faithful homomorphisms $\lambda_0 : \widehat{C_0} \to H^2(G,\T)$ and $\rho_0 : \widehat{D_0} \to H^2(\Lambda,\T)$ such that $\lambda_0(\mu|_{C_0}) = \lambda(\mu)$ and $\rho_0(\om|_{D_0}) = \rho(\om)$ for all $\mu \in \Chat$, $\om \in \Dhat$.

We also define the isomorphisms $\al_n : H^2(G,\T) \to H^2(\Lambda,\T)$ by $\al_n(\psi) = \psi \circ \delta_n^{-1}$ for all $\psi \in H^2(G,\T)$. By construction, $\al_n(\lambda(\mu)) = \rho(\Delta(\mu))$ for all $\mu \in \cU_n$ and all $n \in \N$. It follows that for every $n \in \N$, $(\al_n \circ \lambda)^{-1}(\rho(\Dhat))$ is a closed subgroup $\Chat$ that contains $\cU_n$, so that it is nonnegligible and thus open. Since $C$ is torsion free, the group $\Chat$ is connected and we conclude that $\al_n(\lambda(\Chat)) \subset \rho(\Dhat)$ for all $n \in \N$.

Denote $\cV_n = \Delta(\cU_n)$. Then $\cV_n \subset \rho^{-1}(\al_n(\lambda(\Chat)))$, so that the closed subgroup $\al_n(\lambda(\Chat))$ of $\rho(\Dhat)$ is nonnegligible and thus open. Since $\al_n(\lambda(\Chat))$ is connected, we conclude that the connected component $K$ of the identity in $\rho(\Dhat)$ is an open subgroup of $\rho(\Dhat)$ and $\al_n(\lambda(\Chat)) = K$ for all $n \in \N$. It follows that $\cV_n \subset \rho^{-1}(K)$ for all $n \in \N$. Since $\bigcup_{n \in \N} \cV_n = \Dhat$ up to measure zero, we conclude that $\Dhat \subset \rho^{-1}(K)$. It follows that $K = \rho(\Dhat)$. We get for all $n \in \N$ that
$$\al_n(\lambda_0(\widehat{C_0})) = \al_n(\lambda(\Chat)) = \rho(\Dhat) = \rho_0(\widehat{D_0}) \; .$$
This induces a sequence of group isomorphisms $\zeta_n : C_0 \to D_0$ such that $\al_n(\lambda_0(\mu)) = \rho_0(\mu \circ \zeta_n^{-1})$ for all $\mu \in \widehat{C_0}$ and all $n \in \N$.

Only using $n=1$, note that $\delta_1$ induces an isomorphism between $\Lambdatil$ and the central extension of $G$ by $D$ associated with $\Om' \circ \delta_1$. Below we prove that $[C:C_0] \tau_{\Gtil}(p) = [D:D_0]$. Since $t = \tau_{\Gtil}(p)$ and using $\zeta_1 : C_0 \to D_0$, we get that $\Om \sim^t \Om'$, so that $\Lambdatil \in \cH_{\sim^t}$.

For every compact group $L$, we denote by $\Haar_L$ the unique Haar probability measure on $L$. As explained above, $\tau_{\Gtil}(p) = \Haar_{\Chat}(\cU)$.

Fix $n \in \N$. We have that $\al_n(\lambda(\mu)) = \rho(\Delta(\mu))$ for all $\mu \in \cU_n$, where $\Delta : \cU_n \to \cV_n$ and $\al_n : H^2(G,\T) \to H^2(\Lambda,\T)$ are bijective. By Lemma \ref{lem.characterize-finite-kernel} below, $[C:C_0] <+\infty$ iff the restriction of $\lambda$ to $\cU_n$ is finite-to-one. Similarly, $[D:D_0] < +\infty$ iff the restriction of $\rho$ to $\cV_n$ is finite-to-one. So, $[C:C_0]<+\infty$ iff $[D:D_0]<+\infty$. To conclude the proof of (b), we may thus assume that $[C:C_0]<+\infty$ and $[D:D_0]<+\infty$, and we have to prove that $\Haar_{\Chat}(\cU) = [C:C_0]^{-1} [D:D_0]$.

Denote by $S < \Dhat$ the kernel of the quotient homomorphism $\Dhat \to \widehat{D_0} : \om \mapsto \om|_{D_0}$. Choose a Borel lift $\theta_0 : \widehat{D_0} \to \Dhat$ for this quotient homomorphism. As explained in the beginning of the proof of \ref{thm.inherit-iso-superrigidity.1}, the map $\theta : S \times \widehat{D_0} \to \Dhat : (s,\om) = s \theta_0(\om)$ is a bijective Borel map that is measure preserving if we equip $S$, $\widehat{D_0}$ and $\Dhat$ with their respective Haar probability measures. Note that $S$ is a finite group of order $[D:D_0]$, whose Haar measure is thus the normalized counting measure. It follows that
\begin{align}
\Haar_{\Dhat}(\cV_n) &= \int_{\widehat{D_0}} \Haar_S\bigl(\{s \in S \mid s \theta_0(\psi) \in \cV_n \}\bigr) \, d\Haar_{\widehat{D_0}}(\psi) \notag\\
&= [D:D_0]^{-1} \int_{\widehat{D_0}} \# \{s \in S \mid s \theta_0(\psi) \in \cV_n \} \, d\Haar_{\widehat{D_0}}(\psi) \notag\\
&= [D:D_0]^{-1} \int_{\widehat{D_0}} \# \{\om \in \cV_n \mid \om|_{D_0} = \psi\} \, d\Haar_{\widehat{D_0}}(\psi) \; .\label{eq.good-Haar-equality}
\end{align}
The isomorphisms $\zeta_n : C_0 \to D_0$ also induce isomorphisms $\be_n : \widehat{C_0} \to \widehat{D_0} : \be_n(\mu) = \mu \circ \zeta_n^{-1}$. By definition, for all $\mu \in \cU_n$ and $\psi \in \widehat{C_0}$, we have $\al_n(\lambda(\mu)) = \rho(\Delta(\mu))$ and $\al_n(\lambda_0(\psi)) = \rho_0(\be_n(\psi))$. Therefore, $\Delta(\mu)|_{D_0} = \be_n(\psi)$ iff $\mu|_{C_0} = \psi$, so that for every fixed $\psi \in \widehat{C_0}$, the map $\Delta$ restricts to a bijection between the finite sets
$$\{\mu \in \cU_n \mid \mu|_{C_0} = \psi\} \quad\text{and}\quad \{\om \in \cV_n \mid \om|_{D_0} = \be_n(\psi)\} \; .$$
Since $\be_n$ preserves the Haar measure, it then follows from \eqref{eq.good-Haar-equality} that
\begin{align*}
\Haar_{\Dhat}(\cV_n) &= [D:D_0]^{-1} \int_{\widehat{C_0}} \# \{\mu \in \cU_n \mid \mu|_{C_0} = \psi\} \, d\Haar_{\widehat{C_0}}(\psi) \\ &= [D:D_0]^{-1} [C:C_0] \Haar_{\Chat}(\cU_n) \; .
\end{align*}
Summing over $n$ and using that $\Dhat = \bigcup_{n \in \N} \cV_n$ up to measure zero, we get that
$$1 = [D:D_0]^{-1} [C:C_0] \Haar_{\Chat}(\cU) \; ,$$
so that the proof of (b) is complete.
\end{proof}

\begin{proof}[{Proof of Theorem \ref{thm.inherit-iso-superrigidity.3}}]
Assume that $\Lambdatil$ is a central extension of $G$ by a countable abelian $D$ and $\Om' \in H^2(G,D)$ satisfying $\Om \sim^t \Om'$. Choose subgroups $C_0 < C$ and $D_0 < D$, and an isomorphism $\zeta : C_0 \to D_0$ such that $[C:C_0] t = [D:D_0]$ and such that \eqref{eq.cocycle-equal} holds. In particular, $\mu \circ \Om = 1$ in $H^2(G,\T)$ for all $\mu \in \Chat$ that are equal to $1$ on $C_0$. Since $\Theta(\Om)$ is assumed to be faithful, we get that $C_0 = C$. So, $[D:D_0] = t \leq 1$. It follows that $t=1$ and $D_0 = D$. Also, $\zeta^{-1} \circ \Om' \in H^2(G,C)$ satisfies $\Om \approx \zeta^{-1} \circ \Om'$, so that $\Lambdatil$ belongs to $\cH_\approx$.
\end{proof}

\begin{proof}[{Proof of Theorem \ref{thm.inherit-iso-superrigidity.4}}]
Assume that $\Lambdatil$ is a central extension of $G$ by $C$ and $\Om' \in H^2(G,C)$ satisfying $\Om \approx \Om'$. This means that $\Om - \Om'$ belongs to the kernel of $\Theta$. By the universal coefficient theorem \eqref{eq.univ-coeff} and the assumption that $\Ext^1(G\ab,C) = 0$, it follows that $\Om - \Om' = 0$ in $H^2(G,C)$. So, $\Om = \Om'$ in $H^2(G,C)$ and $\Gtil \cong \Lambdatil$.
\end{proof}

\begin{lemma}\label{lem.characterize-finite-kernel}
Let $K$ be a second countable compact group with Haar probability measure $\mu$. Let $K_0 < K$ be a closed subgroup and consider the quotient homomorphism $\lambda : K \to K/K_0$.
If $\cU \subset K$ is a nonnegligible Borel set and the restriction $\lambda|_\cU$ is finite-to-one, then $K_0$ is finite.
\end{lemma}
\begin{proof}
By making $\cU$ smaller, we may assume that $\lambda|_{\cU}$ is injective and still $\mu(\cU) > 0$. Since $\lambda(k x) = \lambda(x)$ for all $k \in K_0$ and $x \in K$, it follows that $k \cU \cap \cU = \emptyset$ for all $k \in K_0 \setminus \{e\}$. Since $\mu(k \cU) = \mu(\cU)$ for all $k \in K_0$ and $\mu$ is a probability measure, it follows that $K_0$ is finite.
\end{proof}

\subsection{Further remarks and comments}

The faithfulness of $\Theta(\Om)$ appears as an assumption in Theorem \ref{thm.inherit-iso-superrigidity.3}. It can be equivalently expressed as follows.

\begin{lemma}\label{lem.extend-characters-C}
Let $G$ be a group and $C$ an abelian group. Let $\Om \in H^2(G,C)$ with corresponding central extension $0 \to C \to \Gtil \to G \to e$. Consider $\Theta(\Om) \in \Hom(\Chat,H^2(G,\T))$. Then the kernel of $\Theta(\Om)$ equals the image of the restriction homomorphism $\res : \Hom(\Gtil,\T) \to \Chat$.
\end{lemma}
\begin{proof}
Fix a lifting map $\vphi : G \to \Gtil$ such that $\vphi(g)\vphi(h) = \vphi(gh) \Om(g,h)$ for all $g,h \in G$. If $\om \in \Hom(\Gtil,\T)$ and $\mu = \om|_C$, we get that $\mu \circ \Om = \partial(\om \circ \vphi)$, so that $\mu \in \Ker \Theta(\Om)$. If $\mu \in \Chat$ belongs to the kernel of $\Theta(\Om)$, we take $\eta : G \to \T$ such that $\mu \circ \Om = \partial \eta$ and define $\om \in \Hom(\Gtil,\T)$ by $\om(c\vphi(g)) = \mu(c) \eta(g)$ for all $c \in C$, $g \in G$. By construction, $\om|_C = \mu$.
\end{proof}

\begin{remark}\label{rem.relation-to-CFQOT24}
Assume that $G$ is an icc hyperbolic property (T) group satisfying condition \ref{cond.iso.1}. If $0 \to C \to \Gtil \to G \to e$ is any central extension with $\Gtil\ab = 0$ and $C$ torsion free, all conditions appearing in Theorem \ref{thm.inherit-iso-superrigidity} are satisfied: by property (T), condition \ref{ncond.iso.2} holds; by \cite[Theorem 1.4]{PV12} and hyperbolicity of $G$, condition \ref{ncond.iso.3} holds w.r.t.\ the trivial subgroup $\{e\} \subset G$; by Lemma \ref{lem.extend-characters-C} and the triviality of $\Gtil\ab$, the faithfulness of $\Theta(\Om)$ holds; since $\Gtil\ab = 0$, also $G\ab = 0$ so that $\Ext^1(G\ab,C) = 0$.

In \cite{CFQOT24}, such icc hyperbolic property (T) groups satisfying condition \ref{cond.iso.1} are constructed as wreath-like products, and they are shown to have central extensions $0 \to C \to \Gtil \to G \to e$ with $\Gtil\ab = 0$ and with $\Gtil$ still having property (T). They thus obtain W$^*$-superrigid groups $\Gtil$ with infinite center and with property (T). As explained in the introduction, the results of \cite{CFQOT24} were obtained in parallel with and independently of our results. They prove more directly W$^*$-superrigidity of $\Gtil$. Quite naturally, they also use the direct integral decomposition \eqref{eq.central-decomp}, resulting in a small overlap between \cite{CFQOT24} and this Section \ref{sec.generic-iso-superrigid} of our paper.
\end{remark}

\begin{remark}\label{rem.relax-iso}
In the proof of Theorem \ref{thm.inherit-iso-superrigidity.2}, we did not use conditions \ref{cond.iso.1} and \ref{ncond.iso.2} in their full generality. Given a concrete central extension $0 \to C \to \Gtil \to G \to e$ with associated $2$-cocycle $\Om \in H^2(G,C)$, we only used that condition \ref{cond.iso.1} holds for the specific $2$-cocycles in $H^2(G,\T)$ of the form $\mu \circ \Om$, $\mu \in \Chat$.

Secondly, in our application of condition \ref{ncond.iso.2}, it would be sufficient to have that $\Aut G / \cA$ is countable, where $\cA < \Aut G$ is the subgroup of automorphisms $\delta$ satisfying $\mu \circ \Om \circ \delta = \mu \circ \Om$ in $H^2(G,\T)$ for all $\mu \in \Chat$.

This observation will be used in the proof of Proposition \ref{prop.new-no-go-result-3}.
\end{remark}

For completeness, we also prove the following lemma about the relation $\sim^t$ introduced in Definition \ref{def.equiv-2-cocycle}.

\begin{lemma}\label{lem.prop-of-sim-t}
Let $G$ be a countable group and, for $i \in \{1,2,3\}$, let $C_i$ be countable abelian groups and $\Om_i \in H^2(G,C_i)$.
\begin{enuma}
\item If $\Om_1 \sim^t \Om_2$ for $t > 0$, then $\Om_2 \sim^{t^{-1}} \Om_1$.
\item If $\Om_1 \sim^t \Om_2$ and $\Om_2 \sim^s \Om_3$ for $t,s > 0$, then $\Om_1 \sim^{ts} \Om_3$.
\end{enuma}
\end{lemma}
\begin{proof}
(a) This follows immediately from the definition, because $[C_1:D_1] t = [C_2:D_2]$ iff $[C_2:D_2] t^{-1} = [C_1:D_1]$, also when the indices are infinite.

(b) Take an isomorphism $\al : D_1 \to D_2$ between subgroups $D_1 < C_1$ and $D_2 < C_2$ such that $[C_1:D_1] t = [C_2:D_2]$ and $\om_1 \circ \Om_1 = \om_2 \circ \Om_2$ in $H^2(G,\T)$ whenever the $\om_i \in \widehat{C_i}$ satisfy $\om_2 \circ \al = \om_1|_{D_1}$.

Similarly, take an isomorphism $\be : D'_2 \to D_3$ between subgroups $D'_2 < C_2$ and $D_3 < C_3$ such that $[C_2:D'_2]s = [C_3:D_3]$ and $\om_2 \circ \Om_2 = \om_3 \circ \Om_3$ in $H^2(G,\T)$ whenever the $\om_i \in \widehat{C_i}$ satisfy $\om_3 \circ \be = \om_2|_{D'_2}$.

Define $D'_1 = \al^{-1}(D_2 \cap D'_2)$ and $D'_3 = \be(D_2 \cap D'_2)$, so that the restriction $\gamma$ of $\be \circ \al$ to $D'_1$ defines an isomorphism between $D'_1$ and $D'_3$. It is easy to check that $[C_1:D'_1] t s = [C_3:D'_3]$.

Take $\om_1 \in \widehat{C_1}$ and $\om_3 \in \widehat{C_3}$ such that $\om_3 \circ \gamma = \om_1|_{D'_1}$. Then, there is a unique and well-defined homomorphism $\om_2 : D_2 + D'_2 \to \T$ satisfying $\om_2(a) = \om_1(\al^{-1}(a))$ for all $a \in D_2$ and $\om_2(b) = \om_3(\be(b))$ for all $b \in D'_2$. Extend $\om_2$ to a character on $C_2$. By construction, $\om_2 \circ \al = \om_1|_{D_1}$, so that $\om_1 \circ \Om_1 = \om_2 \circ \Om_2$ in $H^2(G,\T)$. Also, $\om_3 \circ \be = \om_2|_{D'_2}$, so that $\om_2 \circ \Om_2 = \om_3 \circ \Om_3$ in $H^2(G,\T)$. It thus follows that $\om_1 \circ \Om_1 = \om_3 \circ \Om_3$ in $H^2(G,\T)$ and we conclude that $\Om_1 \sim^{ts} \Om_3$.
\end{proof}

\subsection{Rigidity of trivial central extensions}

As explained in the introduction, direct products $C \times G$ of an infinite abelian group $C$ with a discrete group $G$ are never W$^*$-superrigid, because $L(C \times G) \cong L(D \times G)$ for any other infinite abelian group $D$. Nevertheless in \cite{CFQT24}, it is proven that for certain groups $G$, this is the only obstruction to W$^*$-superrigidity, meaning that any countable group $\Lambda$ with $L(\Lambda) \cong L(C \times G)$ must be of the form $\Lambda \cong D \times G$.

In the specific case of trivial central extensions, we can prove a more general variant of Theorem \ref{thm.inherit-iso-superrigidity.2}. In this way, we obtain new examples of rigidity of trivial central extensions in Proposition \ref{prop.examples-rigidity-trivial-extensions}. To prove this variant of Theorem \ref{thm.inherit-iso-superrigidity.2}, the following weaker variant of condition \ref{cond.iso.1} suffices.

\begin{enumlistI}
\item\label{cond.iso.I} If $p \in L(G)$ is a projection and $p L(G) p \cong L_\om(\Lambda)$ for some countable group $\Lambda$ and $\om \in H^2(\Lambda,\T)$, then $p=1$, $\om = 1$ in $H^2(\Lambda,\T)$ and $\Lambda \cong G$.
\end{enumlistI}

\begin{theorem}\label{thm.rigidity-trivial-extensions}
Let $G$ be a countable icc group satisfying conditions \ref{cond.iso.I} and \ref{ncond.iso.3}. Denote by $G\ab$ its abelianization.
\begin{enuma}
\item\label{thm.rigidity-trivial-extensions.1} If $G\ab$ is free abelian, then rigidity holds: for any countable abelian group $C$, countable group $\Lambda$ and projection $p \in L(C \times G)$, we have that $pL(C \times G)p \cong L(\Lambda)$ if and only if $p \in L(C)$ and $\Lambda \cong D \times G$ where $D$ is a countable abelian group with $|D| = |C|\tau(p)$.
\item\label{thm.rigidity-trivial-extensions.2} If $G\ab$ is not free abelian, then rigidity fails: there exists a countable group $\Lambda$ such that $L(\Z \times G) \cong L(\Lambda)$, but $\Lambda$ is not isomorphic to $D \times G$ for any abelian group $D$.
\end{enuma}
\end{theorem}

Note that in the formulation of the theorem, we follow the convention that also the trivial group $\{0\}$ is free abelian.

\begin{proof}
(a) Assume that $p L(C \times G) p \cong L(\Lambda)$. We identify $L(C) = L^\infty(Y)$ and view $L(C \times G)$ as the direct integral of the constant field $L(G)$ over $Y$. In this way, $p$ decomposes as a measurable family of projections $p_x \in L(G)$, $x \in Y$. We define $X \subset Y$ as the set of $x \in Y$ such that $p_x \neq 0$. Then $pL(C \times G)p$ is the direct integral of the factors $(p_x L(G) p_x)_{x \in X}$.

We denote by $(v_s)_{s \in \Lambda}$ the canonical unitaries generating $L(\Lambda)$. Since $L(\Lambda\fc)$ is an amenable von Neumann subalgebra of $L(\Lambda) \cong p L(C \times G) p$ on which the action $(\Ad v_s)_{s \in \Lambda}$ is compact, it follows from condition \ref{ncond.iso.3} that $L(\Lambda\fc) \prec^f \cZ(L(\Lambda))$. Write $\Lambda_1 = \Lambda/\Lambda\fc$. By Proposition \ref{prop.structure-non-icc.3}, we find measurable maps $d : X \to \N$, $n : X \to \N$ and measurable families of finite index subgroups $\Lambda_x < \Lambda_1$ and $2$-cocycles $\om_x \in H^2(\Lambda_x,\T)$ such that
$$p_x L(G) p_x \cong M_{d(x)n(x)}(\C) \ot L_{\om_x}(\Lambda_x) \quad\text{for a.e.\ $x \in X$,}$$
and such that $n(x) = [\Lambda_1 : \Lambda_x]$ and $L(\Lambda\fc)$ is the direct integral of $M_{d(x)}(\C) \ot \C^{n(x)}$.

By condition \ref{cond.iso.I}, we have that $p_x = 1$, $d(x) = 1 = n(x)$, $\om_x = 1$ in $H^2(\Lambda_x,\T)$ and $G \cong \Lambda_x$ for a.e.\ $x \in X$. This means that $p \in L(C)$, that $\Lambda\fc = \cZ(\Lambda)$ and that $G \cong \Lambda_1$, so that $\Lambda$ is a central extension $0 \to D \to \Lambda \to G \to e$ whose associated $2$-cocycle $\Om \in H^2(G,D)$ satisfies $\om \circ \Om = 1$ in $H^2(G,\T)$ for all $\om \in \Dhat$. We also get that $L(C) p \cong L(D)$ so that $|D| = |C| \tau(p)$.

Since $\om \circ \Om = 1$ in $H^2(G,\T)$ for all $\om \in \Dhat$, it follows from the universal coefficient theorem \eqref{eq.univ-coeff} that $\Om = \Upsilon(\Psi)$ for some $\Psi \in \Ext^1(G\ab,D)$. Since $G\ab$ is a free abelian group, $\Ext^1(G\ab,D)=0$ and thus $\Psi = 0$ in $\Ext^1(G\ab,D)$, so that $\Om = 0$ in $H^2(G,D)$. This means that $\Lambda \cong D \times G$.

(b) Denote $L = \Z^{(\N)}$ and choose a surjective homomorphism $\theta : L \to G\ab$. Since $G\ab$ is not free abelian and $L$ is torsion free, the kernel $C := \Ker \theta$ is infinite. Since all subgroups of $L$ are free abelian, the extension $0 \to C \to L \to G\ab \to 0$ is not split and thus defines a nonzero element $\Psi \in \Ext^1(G\ab,C)$. With the notation of \eqref{eq.univ-coeff}, define $\Om \in H^2(G,C)$ by $\Om = \Upsilon(\Psi)$. Then $\Theta(\Om) = 1$, meaning that $\om \circ \Om = 1$ in $H^2(G,\T)$ for all $\om \in \Chat$. Denote by $0 \to C \to \Lambda \to G \to e$ the central extension that corresponds to $\Om$.

Since $\om \circ \Om = 1$ in $H^2(G,\T)$ for all $\om \in \Chat$, we get that $L(\Lambda) \cong L(C) \ovt L(G)$. Since $C$ is infinite, it follows that $L(\Lambda) \cong L(\Z \times G)$.

Since $G$ has trivial center, $\cZ(\Lambda) = C$. Since $\Om \neq 0$ in $H^2(G,C)$, the extension $0 \to \cZ(\Lambda) \to \Lambda \to \Lambda/\cZ(\Lambda) \to e$ is not split. So, $\Lambda$ is not isomorphic to $D \times G$ for any abelian group $D$.
\end{proof}

\begin{proposition}\label{prop.examples-rigidity-trivial-extensions}
Let $\Gamma$ be a countable group in class $\cC$ (see \ref{sec.class-C}). Define $\Sigma_1 < (\Z/2\Z)^{(\Gamma)}$ as the index $2$ subgroup of elements $x$ satisfying $\sum_{g \in \Gamma} x_g = 0$. Define $G_1 = \Sigma_1 \rtimes (\Gamma \times \Gamma)$.
\begin{enuma}
\item If the abelianization of $\Gamma$ is trivial, the conclusion of Theorem \ref{thm.rigidity-trivial-extensions.1} holds for $G_1$, so that we have rigidity of trivial central extensions.
\item If the abelianization of $\Gamma$ is nontrivial, the conclusion of Theorem \ref{thm.rigidity-trivial-extensions.2} holds for $G_1$, so that rigidity of trivial central extensions fails.
\end{enuma}
\end{proposition}

Note that by taking free products of amenable groups with trivial abelianization, class $\cC$ contains numerous groups with trivial abelianization.

\begin{proof}
We first prove that $G_1$ satisfies condition \ref{cond.iso.I}. Assume that $p L(G_1) p \cong L_\om(\Lambda)$. By \cite[Lemma 6.11]{DV24}, the virtual center $\Lambda\fc$ is finite. Since $G_1$ is icc, $L_\om(\Lambda)$ is a factor. It then follows from \cite[Proposition 6.1]{DV24} that $L_\om(\Lambda\fc) \cong M_d(\C) \ot \C^k$ and that for a minimal projection $q \in L_\om(\Lambda\fc)$, we have that $q L_\om(\Lambda) q \cong L_{\om_1}(\Lambda_1)$, where $\Lambda_1 < \Lambda/\Lambda\fc$ is an index $k$ subgroup and $\om_1 \in H^2(\Lambda_1,\T)$ is a canonical variant of $\om$.

To the projection $q$ corresponds a projection $p_1 \leq p$ in $L(G_1)$ such that $p_1 L(G_1) p_1 \cong L_{\om_1}(\Lambda_1)$. By construction, $\Lambda_1$ is icc. We can now follow the proof of \cite[Theorem 6.12]{DV24} and conclude that $p_1 = 1$, $\om_1 = 1$ in $H^2(\Lambda_1,\T)$ and $\Lambda_1 \cong G$. Since $p_1 = 1$, we get that $p=1$ and $q=1$. Since $q=1$, we get that $d=1=k$, so that $\Lambda\fc = \{e\}$. This also implies that $(\Lambda_1,\om_1) = (\Lambda,\om)$. We have thus proven that $p=1$, $\om = 1$ in $H^2(\Lambda,\T)$ and $\Lambda \cong G$.

By Lemma \ref{lem.left-right-wreath-product-satisfies-4}, the group $G_1$ satisfies condition \ref{ncond.iso.3}.

We finally prove the following dichotomy: if $\Gamma\ab$ is trivial, also $G_{1,\text{\rm ab}}$ is trivial, while if $\Gamma\ab$ is nontrivial, then $G_{1,\text{\rm ab}}$ is a nontrivial group that is not free abelian. Denote by $(\al_{(g,h)})_{(g,h) \in \Gamma \times \Gamma}$ the action of $\Gamma \times \Gamma$ on $\Sigma_1$. Then $G_{1,\text{\rm ab}}$ is given by $\Sigma_2 \times \Gamma\ab \times \Gamma\ab$, where $\psi : \Sigma_1 \to \Sigma_2$ is the largest quotient satisfying $\psi \circ \al_{(g,h)} = \psi$ for all $g,h \in \Gamma$. In the proof of Lemma \ref{lem.about-all-our-conditions.6}, we have seen that $\Sigma_2$ is nontrivial if and only if $\Gamma$ admits a subgroup of index $2$, i.e.\ admits $\Z/2\Z$ as a quotient. So if $\Gamma\ab$ is trivial, also $G_{1,\text{\rm ab}}$ is trivial, while if $\Gamma\ab$ is nontrivial, then $G_{1,\text{\rm ab}}$ has elements of order $2$. So the dichotomy is proven.

Both (a) and (b) now follow from Theorem \ref{thm.rigidity-trivial-extensions}.
\end{proof}

\section{\boldmath Virtual isomorphism W$^*$-superrigidity}\label{sec.generic-virtual-iso-superrigid}

In \cite{DV24}, the first W$^*$-superrigidity theorem up to virtual isomorphism was obtained. In the same way as we did in Section \ref{sec.generic-iso-superrigid} for isomorphism W$^*$-superrigidity, we provide in this section sufficient conditions to transfer virtual isomorphism W$^*$-superrigidity of $G$ to certain central extensions $0 \to C \to \Gtil \to G \to e$. Unsurprisingly these sufficient conditions are virtual versions of the sufficient conditions appearing in Section \ref{sec.generic-iso-superrigid}.

We fix an icc group $G$. The following conditions are the appropriate virtual versions of conditions \ref{cond.iso.1}, \ref{ncond.iso.2} and \ref{ncond.iso.3}.

\begin{enumlistprime}
\item\label{cond.iso.1prime} The group $G$ satisfies the following virtual isomorphism cocycle W$^*$-superrigidity property: if $\Lambda$ is any countable group and $\mu \in H^2(G,\T)$, $\om \in H^2(\Lambda,\T)$ are arbitrary $2$-cocycles such that there exists a nonzero bifinite $L_\mu(G)$-$L_\om(\Lambda)$-bimodule, then the pairs $(G,\mu)$ and $(\Lambda,\om)$ are virtually isomorphic, meaning that there exists a finite index subgroup $\Lambda_0 < \Lambda$ and a group homomorphism $\delta : \Lambda_0 \to G$ with finite kernel and image of finite index such that the $2$-cocycle $\om|_{\Lambda_0} \, \overline{\mu \circ \delta}$ is of finite type.
\end{enumlistprime}

By \cite[Theorem A]{DV24}, all left-right wreath products $G = (\Z/2\Z)^{(\Gamma)} \rtimes (\Gamma \times \Gamma)$ with $\Gamma$ in class $\cC$ (see Section \ref{sec.class-C}) satisfy condition \ref{cond.iso.1prime}.

In order to go from finite type $2$-cocycles to actual coboundaries, we also need the following.
\begin{enumlistprimeprime}[start=1]
\item\label{ncond.iso.1primeprime} For every finite index subgroup $G_0 < G$ and finite type $2$-cocycle $\Phi \in H^2(G_0,\T)$, there exists a finite index subgroup $G_1 < G_0$ such that $\Phi|_{G_1} = 1$ in $H^2(G_1,\T)$.
\end{enumlistprimeprime}

We replace condition \ref{ncond.iso.2} in Section \ref{sec.generic-iso-superrigid} by its natural virtual counterpart, essentially saying that the commensurator of $G$ is countable.

\begin{enumlistprime}[resume]
\item\label{ncond.iso.2prime} The set of triples $(G_0,G_1,\delta)$ where $G_0,G_1 < G$ are finite index subgroups and $\delta : G_0 \to G_1$ is a group isomorphism, is countable.
\end{enumlistprime}

The third condition (iii)$^\prime$ is identical to the condition \ref{ncond.iso.3}, so we do not repeat it here.

We thus take an icc group $G$ satisfying conditions \ref{cond.iso.1prime}, \ref{ncond.iso.1primeprime}, \ref{ncond.iso.2prime} and \ref{ncond.iso.3}, and we consider an \emph{arbitrary} central extension $\Gtil$ of $G$ by a torsion free abelian group $C$ with associated $2$-cocycle $\Om \in H^2(G,C)$. In Theorem \ref{thm.inherit-virtual-iso-superrigidity}, we describe all countable groups $\Lambdatil$ for which there exists a nonzero\slash faithful bifinite $L(\Gtil)$-$L(\Lambdatil)$-bimodule. We prove that up to virtual isomorphism\slash commensurability these groups are themselves central extensions whose associated $2$-cocycle is ``equivalent'' with $\Om$. Under extra assumptions on $H^2(G,C)$ and $\Om$, this forces $\Lambdatil$ to be virtually isomorphic\slash commensurable to $\Gtil$, so that W$^*$-superrigidity follows.

We thus introduce the following variant of Definition \ref{def.equiv-2-cocycle}, providing the appropriate notion of equivalence between $2$-cocycles.

\begin{definition}\label{def.equiv-2-cocycle-bis}
Let $G$ be a countable group and $C$, $D$ countable abelian groups. Let $\Om \in H^2(G,C)$ and $\Om' \in H^2(G,D)$.
\begin{enuma}
\item\label{def.equiv-2-cocycle-bis.1} We write $\Om \sim \Om'$ if $\Om \sim^t \Om'$ for some $t > 0$ (see Definition \ref{def.equiv-2-cocycle}). So, $\Om \sim \Om'$ iff there exists an isomorphism $\zeta : C_0 \to D_0$ between subgroups $C_0 < C$ and $D_0 < D$ such that
    \begin{equation}\label{eq.cocycle-equal-bis}
\mu \circ \Om = \om \circ \Om' \quad\text{in $H^2(G,\T)$ whenever $\mu \in \Chat$, $\om \in \Dhat$ satisfy $\om \circ \zeta = \mu|_{C_0}$.}
\end{equation}
    and such that $[C:C_0]$, $[D:D_0]$ are both finite, or both infinite.

\item\label{def.equiv-2-cocycle-bis.2} We write $\Om \approx\ve \Om'$ if $\Om \sim \Om'$ with both $C_0 < C$, $D_0 < D$ of finite index.
\end{enuma}
\end{definition}

Recall the homomorphism $\Theta : H^2(G,C) \to \Hom(\Chat,H^2(G,\T))$ introduced in \eqref{eq.hom-Theta}.

\begin{theorem}\label{thm.inherit-virtual-iso-superrigidity}
Let $0 \to C \to \Gtil \to G \to e$ be any central extension of countable groups with associated $2$-cocycle $\Om \in H^2(G,C)$.

Define $\cF_\sim$, resp.\ $\cF_{\approx\ve}$, as the class of countable groups that are isomorphic to a central extension of a finite index subgroup $G_0 < G$ by a countable abelian group $D$ such that the associated $2$-cocycle $\Om' \in H^2(G_0,D)$ satisfies $\Om' \sim \Om|_{G_0}$, resp.\ $\Om' \approx\ve \Om|_{G_0}$.

Let $\Lambdatil$ be any countable group.
\begin{enuma}
\item\label{thm.inherit-virtual-iso-superrigidity.1} If $\Lambdatil$ is virtually isomorphic\slash commensurable to a group in $\cF_\sim$, there exists a nonzero\slash faithful bifinite $L(\Lambdatil)$-$L(\Gtil)$-bimodule.
\end{enuma}
Assume that $C$ is torsion free.
\begin{enuma}[resume]
\item\label{thm.inherit-virtual-iso-superrigidity.2} If $G$ is icc and satisfies conditions \ref{cond.iso.1prime}, \ref{ncond.iso.1primeprime}, \ref{ncond.iso.2prime} and \ref{ncond.iso.3}, also the converse of \ref{thm.inherit-virtual-iso-superrigidity.1} holds: there exists a nonzero\slash faithful bifinite $L(\Lambdatil)$-$L(\Gtil)$-bimodule iff $\Lambdatil$ is virtually isomorphic\slash commensurable to a group in $\cF_\sim$.

\item\label{thm.inherit-virtual-iso-superrigidity.3} If $\Theta(\Om|_{G_0}) \in \Hom(\Chat,H^2(G_0,\T))$ has finite kernel for every finite index subgroup $G_0 < G$, then $\cF_\sim = \cF_{\approx\ve}$.

\item\label{thm.inherit-virtual-iso-superrigidity.4} If condition \ref{ncond.iso.4} below holds, then every group in $\cF_{\approx\ve}$ is virtually isomorphic to $\Gtil$.

\item\label{thm.inherit-virtual-iso-superrigidity.5} If conditions \ref{ncond.iso.4} and \ref{ncond.iso.5} below hold, then every group in $\cF_{\approx\ve}$ is commensurable to $\Gtil$.
\end{enuma}
\end{theorem}

So when all conditions  \ref{cond.iso.1prime}, \ref{ncond.iso.1primeprime}, \ref{ncond.iso.2prime}, \ref{ncond.iso.3}, \ref{ncond.iso.4}, \ref{ncond.iso.5} and the finiteness of $\Ker \Theta(\Om|_{G_0})$ hold, we have W$^*$-superrigidity up to virtual isomorphisms: for an arbitrary countable group $\Lambdatil$, there exists a nonzero\slash faithful bifinite $L(\Gtil)$-$L(\Lambdatil)$-bimodule if and only if $\Gtil$ is virtually isomorphic\slash commensurable to $\Lambdatil$.

One can also rephrase Theorem \ref{thm.inherit-virtual-iso-superrigidity} in the following way. The groups that are virtually isomorphic\slash commensurable to resp.\ $\Gtil$, a group in $\cF_{\approx\ve}$, a group in $\cF_\sim$, and the groups whose group von Neumann algebra admits a nonzero\slash faithful bifinite bimodule with $L(\Gtil)$, form larger and larger families of groups, and the rigidity results \ref{thm.inherit-virtual-iso-superrigidity.2}-\ref{thm.inherit-virtual-iso-superrigidity.5} say which of these families are equal.

In Theorem \ref{thm.inherit-virtual-iso-superrigidity}, we use the following cohomological conditions on $G$ and $C$. Both are virtual counterparts of the assumption $\Ext^1(G\ab,C) = 0$ that appeared in Theorem \ref{thm.inherit-iso-superrigidity}.

\begin{enumlist}[start=4]
\item\label{ncond.iso.4} If $G_0 < G$ has finite index and $D$ is a torsion free abelian group that is commensurable to $C$, then for every $\Om \in \Ext^1(G_{0,\text{\rm ab}},D)$, there exists a finite index subgroup $G_1 < G_0$ such that $\Om \circ \pi = 0$ in $\Ext^1(G_{1,\text{\rm ab}},D)$, where $\pi : G_{1,\text{\rm ab}} \to G_{0,\text{\rm ab}}$ is the natural homomorphism.

\item\label{ncond.iso.5} If $G_0 < G$ is a finite index subgroup and $E$ is a finite abelian group, then for every $\Om \in \Ext^1(G_{0,\text{\rm ab}},E)$, there exists a finite index subgroup $G_1 < G_0$ such that $\Om \circ \pi = 0$ in $\Ext^1(G_{1,\text{\rm ab}},E)$, where $\pi : G_{1,\text{\rm ab}} \to G_{0,\text{\rm ab}}$ is the natural homomorphism.
\end{enumlist}

By Lemma \ref{lem.extend-characters-C}, the finiteness of the kernel of $\Theta(\Om|_{G_0})$ is equivalent to the finiteness of the image of the restriction homomorphism $\Hom(\Gtil_0,\T) \to \Chat$, where $\Gtil_0 < \Gtil$ is the preimage of $G_0 < G$. In Proposition \ref{prop.new-no-go-result-1.2}, we prove that this finiteness of the kernel of $\Theta(\Om|_{G_0})$ is a necessary condition for Theorem \ref{thm.inherit-virtual-iso-superrigidity.3} to hold.

In Proposition \ref{prop.new-no-go-result-3}, we illustrate the essential roles of conditions \ref{ncond.iso.4} and \ref{ncond.iso.5}. We give examples where the conclusions of Theorem \ref{thm.inherit-virtual-iso-superrigidity.2} and \ref{thm.inherit-virtual-iso-superrigidity.3} hold, but the conclusion of \ref{thm.inherit-virtual-iso-superrigidity.4} fails. We also give examples where the conclusions of Theorem \ref{thm.inherit-virtual-iso-superrigidity.2}, \ref{thm.inherit-virtual-iso-superrigidity.3} and \ref{thm.inherit-virtual-iso-superrigidity.4} hold, but the conclusion of \ref{thm.inherit-virtual-iso-superrigidity.5} fails.

\begin{proof}[{Proof of Theorem \ref{thm.inherit-virtual-iso-superrigidity.1}}]
By Lemma \ref{lem.from-virtual-iso-to-bimodule}, it suffices to prove the following: if $\Lambdatil$ is a central extension of a finite index subgroup $G_0 < G$ by a countable abelian group $D$ such that the associated $2$-cocycle $\Om' \in H^2(G_0,D)$ satisfies $\Om|_{G_0} \sim \Om'$, then there exists a faithful bifinite $L(\Gtil)$-$L(\Lambdatil)$-bimodule.

Define $\Gtil_0 < \Gtil$ as the finite index subgroup given by the preimage of $G_0 < G$. By Lemma \ref{lem.from-virtual-iso-to-bimodule}, there exists a faithful bifinite $L(\Gtil)$-$L(\Gtil_0)$-bimodule. It thus suffices to prove that there exists a faithful bifinite $L(\Gtil_0)$-$L(\Lambdatil)$-bimodule.

Take an isomorphism $\zeta : C_0 \to D_0$ between subgroups $C_0 < C$ and $D_0 < D$ such that $\mu \circ \Om|_{G_0} = \om \circ \Om'$ in $H^2(G_0,\T)$ whenever $\mu \in \Chat$, $\om \in \Dhat$ satisfy $\om \circ \zeta = \mu|_{C_0}$, and such that $[C:C_0]$, $[D:D_0]$ are both finite, or both infinite.

When both $n=[C:C_0]$ and $m=[D:D_0]$ are finite, we find, as in the proof of Theorem \ref{thm.inherit-iso-superrigidity.1}, a tracial von Neumann algebra $Q$ such that $L(\Gtil_0) \cong \C^n \ot Q$ and $L(\Lambdatil) \cong \C^m \ot Q$. In particular, there exists a faithful bifinite $L(\Gtil_0)$-$L(\Lambdatil)$-bimodule.

When both $[C:C_0]$ and $[D:D_0]$ are infinite, we find, as in the proof of Theorem \ref{thm.inherit-iso-superrigidity.1}, a tracial von Neumann algebra $Q$ such that $L(\Gtil_0) \cong L^\infty([0,1]) \ovt Q \cong L(\Lambdatil)$. Again, there exists a faithful bifinite $L(\Gtil_0)$-$L(\Lambdatil)$-bimodule.
\end{proof}

\begin{proof}[{Proof of Theorem \ref{thm.inherit-virtual-iso-superrigidity.2}, part 1}]
Write $M = L(\Gtil)$, $P = L(\Lambdatil)$ and assume that $H$ is a nonzero bifinite $P$-$M$-bimodule. We prove in part~1 that $\Lambdatil$ is virtually isomorphic to a group in $\cF_\sim$.
Denote by $z \in \cZ(P)$ the left support of $H$. By Lemma \ref{lem.intertwine-virtual-center-to-center} and condition \ref{ncond.iso.3}, we find that $L(\Lambdatil\fc)z \prec^f \cZ(P)$ inside $P$. Denote by $v_s$, $s \in \Lambdatil$, the unitaries generating $L(\Lambdatil)$. Define $\Sigma = \{s \in \Lambdatil \mid v_s z = z \}$. By Proposition \ref{prop.structure-non-icc.1} and \ref{prop.structure-non-icc.2}, we get that $\Sigma$ is a finite normal subgroup of $\Lambdatil$ and we find a finite index subgroup $\Lambdatil_1 < \Lambdatil$ containing $\Sigma$ such that the group $\Lambdatil_0 := \Lambdatil_1 / \Sigma$ satisfies $\Lambdatil_{0,\text{\rm fc}} = \cZ(\Lambdatil_0)$.

Since $v_s z = z$ for all $s \in \Sigma$ and $\Lambdatil_1 < \Lambdatil$ has finite index, we can view $H$ as a nonzero bifinite $L(\Lambdatil_0)$-$L(\Gtil)$-bimodule. Since $\Lambdatil_0$ is virtually isomorphic to $\Lambdatil$, we may replace $\Lambdatil$ by $\Lambdatil_0$ and assume from the start that $\Lambdatil\fc = \cZ(\Lambdatil)$. Denote $D = \cZ(\Lambdatil)$ and define $\Lambda = \Lambdatil / D$. We have found the central extension $0 \to D \to \Lambdatil \to \Lambda \to e$. We denote the associated $2$-cocycle by $\Om' \in H^2(\Lambda,D)$.

Denote by $r \in \cZ(M)$ the right support of $H$. By Proposition \ref{prop.iso-on-center}, we may replace $H$ by a nonzero $P$-$M$-subbimodule and assume that there is a $*$-isomorphism $\al : \cZ(M) r \to \cZ(P) z$ such that
\begin{equation}\label{eq.left-is-right-on-center}
\al(a r) \cdot \xi = \xi \cdot a \quad\text{for all $a \in \cZ(M)$ and $\xi \in H$.}
\end{equation}

Since $G$ is an icc group, we get that $\Gtil\fc = \cZ(\Gtil) = C$, so that $\cZ(M) = L(C) = L^\infty(\Chat)$. Since $\Lambdatil\fc = \cZ(\Lambdatil) = D$, we similarly have that $\cZ(P) = L(D) = L^\infty(\Dhat)$. We thus write $r = 1_\cU$ and $z = 1_\cV$, where $\cU \subset \Chat$ and $\cV \subset \Dhat$ are nonnegligible measurable subsets. The $*$-isomorphism $\al$ gives rise to the nonsingular isomorphism $\Delta : \cU \to \cV$ such that $\al(F) = F \circ \Delta^{-1}$ for every $F \in L^\infty(\cU) = \cZ(M)r$.

As explained in \eqref{eq.concrete-direct-integral}, we have the direct integral decompositions
$$M = \int^\oplus_{\Chat} L_{\mu \circ \Om}(G) \, d\mu \quad\text{and}\quad P = \int^\oplus_{\Dhat} L_{\om \circ \Om'}(\Lambda) \, d\om \; ,$$
where we integrate w.r.t.\ the Haar measures on $\Chat$ and $\Dhat$. We write $M_\mu = L_{\mu \circ \Om}(G)$ and $P_\om = L_{\om \circ \Om'}(\Lambda)$. Since $\cZ(M) = L^\infty(\Chat)$ and $\cZ(P) = L^\infty(\Dhat)$, almost all $M_\mu$ and $P_\om$ are factors. Using \eqref{eq.left-is-right-on-center}, we have a corresponding direct integral decomposition
$$H = \int^\oplus_{\cU} H_\mu \, d\mu$$
in which every $H_\mu$ is a $P_{\Delta(\mu)}$-$M_\mu$-bimodule. Since $H$ is finitely generated as a right Hilbert $M$-module, we have in particular that $\dim_M(H) < +\infty$, where $\dim_M$ denotes the $M$-dimension of $H$ w.r.t.\ the canonical trace $\tau_{\Gtil}$ on $M$. Since the direct integral decomposition \eqref{eq.concrete-direct-integral} is trace preserving, we find that
$$\int_\cU \dim_{M_\mu}(H_\mu) \, d\mu = \dim_M(H) < +\infty \; ,$$
so that $\dim_{M_\mu}(H_\mu) < +\infty$ for a.e.\ $\mu \in \cU$. Since $M_\mu$ is a factor, it follows that $H_\mu$ is finitely generated as a right Hilbert $M_\mu$-module for a.e.\ $\mu \in \cU$. We similarly find that $H_\mu$ is finitely generated as a left Hilbert $P_{\Delta(\mu)}$-module for a.e.\ $\mu \in \cU$. So, $H_\mu$ is a bifinite $P_{\Delta(\mu)}$-$M_\mu$-bimodule for a.e.\ $\mu \in \cU$. Removing null sets from $\cU$ and $\cV$, we have thus found nonnegligible Borel sets $\cU \subset \Chat$, $\cV \subset \Dhat$ and a nonsingular isomorphism $\Delta : \cU \to \cV$ such that for every $\mu \in \cU$, there exists a nonzero bifinite $L_{\Delta(\mu) \circ \Om'}(\Lambda)$-$L_{\mu \circ \Om}(G)$-bimodule.

By condition \ref{cond.iso.1prime}, this means that for every $\mu \in \cU$, there exist a finite index subgroup $\Lambda_\mu < \Lambda$ and a group homomorphism $\phi_\mu : \Lambda_\mu \to G$ such that $\Ker \phi_\mu$ is finite, $\phi_\mu(\Lambda_\mu) < G$ has finite index and
\begin{equation}\label{eq.is-of-finite-type}
(\mu \circ \Om \circ \phi_\mu) \; \overline{(\Delta(\mu) \circ \Om')} \quad\text{is a finite type $2$-cocycle on $\Lambda_\mu$.}
\end{equation}
Using a single $\mu \in \cU$, we see that $\Lambda$ is virtually isomorphic to $G$. By condition \ref{ncond.iso.2prime} and Lemma \ref{lem.automatic-icc} below, we can choose a finite index subgroup $\Lambda_0 < \Lambda$ that is icc.

By condition \ref{ncond.iso.2prime}, we can enumerate the triples $(G_n,\Lambda_n,\delta_n)$, $n \geq 1$, of all finite index subgroups $G_n < G$, $\Lambda_n < \Lambda_0$ and group isomorphisms $\delta_n : G_n \to \Lambda_n$. We then define for all $n \in \N$, the Borel set $\cU_n \subset \cU$ by
$$\cU_n = \bigl\{ \mu \in \cU \bigm| \mu \circ \Om|_{G_n} = \Delta(\mu) \circ \Om' \circ \delta_n \;\;\text{in $H^2(G_n,\T)$} \bigr\} \; .$$
We prove that $\bigcup_{n \in \N} \cU_n = \cU$.

Fix $\mu \in \cU$. Using the notation $\phi_\mu : \Lambda_\mu \to G$ introduced above and the icc subgroup $\Lambda_0 < \Lambda$, it follows that $\Lambda_0 \cap \Ker \phi_\mu = \{e\}$, so that $\phi_\mu$ restricts to an isomorphism between the finite index subgroup $\Lambda_0 \cap \Lambda_\mu$ of $\Lambda_0$ and the finite index subgroup $\phi_\mu(\Lambda_0 \cap \Lambda_\mu)$ of $G$. Denote by $\delta_\mu$ the inverse of this restriction of $\phi_\mu$.

By \eqref{eq.is-of-finite-type}, the $2$-cocycle $(\mu \circ \Om) \, \overline{(\Delta(\mu) \circ \Om' \circ \delta_\mu)}$ is of finite type on the finite index subgroup $\phi_\mu(\Lambda_0 \cap \Lambda_\mu)$ of $G$. By condition \ref{ncond.iso.1primeprime}, we can choose a finite index subgroup $G_\mu < \phi_\mu(\Lambda_0 \cap \Lambda_\mu)$ such that
$$\mu \circ \Om|_{G_\mu} = \Delta(\mu) \circ \Om' \circ \delta_\mu|_{G_\mu} \quad\text{in $H^2(G_\mu,\T)$.}$$
Take $n \in \N$ such that $(G_n,\Lambda_n,\delta_n) = (G_\mu,\delta_\mu(G_\mu),\delta_\mu|_{G_\mu})$. It follows that $\mu \in \cU_n$. Since $\mu \in \cU$ was arbitrary, we have proven that $\bigcup_{n \in \N} \cU_n = \cU$.

There thus exists an $n \in \N$ such that $\cU_n$ is nonnegligible. Fixing any $n \in \N$ for which $\cU_n$ is nonnegligible, we now replace $\Lambda_0$ by the smaller group $\Lambda_n$ and we put $G_0 = G_n$. We also replace $\cU$ by the nonnegligible subset $\cU_n$, and replace $\cV$ by $\Delta(\cU_n)$.

We have thus found finite index subgroups $G_0 < G$, $\Lambda_0 < \Lambda$, a group isomorphism $\delta_0 : G_0 \to \Lambda_0$, nonnegligible Borel sets $\cU \subset \Chat$, $\cV \subset \Dhat$ and a nonsingular isomorphism $\Delta : \cU \to \cV$ such that
$$\mu \circ \Om|_{G_0} = \Delta(\mu) \circ \Om' \circ \delta_0 \quad\text{in $H^2(G_0,\T)$, for all $\mu \in \cU$.}$$
We now proceed as in the proof of Theorem \ref{thm.inherit-iso-superrigidity.2}. We thus define the continuous group homomorphisms
$$\lambda : \Chat \to H^2(G_0,\T) : \lambda(\mu) = \mu \circ \Om|_{G_0} \quad\text{and}\quad \rho : \Dhat \to H^2(\Lambda_0,\T) : \rho(\eta) = \eta \circ \Om'|_{\Lambda_0} \; .$$
We also define the isomorphism $\al : H^2(G_0,\T) \to H^2(\Lambda_0,\T) : \psi \mapsto \psi \circ \delta_0^{-1}$. By definition, $\al(\lambda(\mu)) = \rho(\Delta(\mu))$ for all $\mu \in \cU$.

Note that $(\al \circ \lambda)^{-1}(\rho(\Dhat))$ is a closed subgroup of $\Chat$ that contains $\cU$, so that it is nonnegligible and thus open. Since $C$ is torsion free, $\Chat$ is connected and we conclude that $\al(\lambda(\Chat)) \subset \rho(\Dhat)$. Since $\cV \subset \rho^{-1}(\al(\lambda(\Chat)))$, we similarly find that $\rho^{-1}(\al(\lambda(\Chat)))$ is an open subgroup of $\Dhat$, which is thus equal to $\widehat{D/F}$ for some finite subgroup $F < D$. Write $D_1 = D/F$ and note that $\cV \subset \widehat{D_1}$.

Define the subgroup $C_0 < C$ such that the kernel of $\lambda$ equals $\widehat{C/C_0}$. Define the faithful homomorphisms $\lambda_0 : \widehat{C_0} \to H^2(G_0,\T)$ such that $\lambda_0(\mu|_{C_0}) = \lambda(\mu)$ for all $\mu \in \Chat$. By construction, $(\al \circ \lambda_0)^{-1} \circ \rho$ defines a surjective continuous homomorphism of $\widehat{D_1}$ onto $\widehat{C_0}$ whose dual is a faithful group homomorphism $\zeta : C_0 \to D_1$ satisfying $\mu \circ \Om|_{G_0} = \om \circ \Om' \circ \delta_0$ in $H^2(G_0,\T)$ whenever $\mu \in \Chat$, $\om \in \widehat{D_1}$ and $\om \circ \zeta = \mu|_{C_0}$. Since the kernel of $\lambda$ equals $\widehat{C/C_0}$, it follows in particular that $\om \in \widehat{D_1}$ satisfies $\rho(\om) = 1$ iff $\om$ is equal to $1$ on $\zeta(C_0)$.

Note that $C_0 < C$ has finite index iff the restriction of $\lambda$ to $\cU$ is finite-to-one. Similarly, $D_0 < D_1$ has finite index iff the restriction of $\rho$ to $\cV \subset \widehat{D_1}$ is finite-to-one. Since $\al(\lambda(\mu)) = \rho(\Delta(\mu))$ for all $\mu \in \cU$, we conclude that $[C:C_0]$ and $[D:D_0]$ are either both finite, or both infinite.

Define the $2$-cocycle $\Om'_0 \in H^2(G_0,D_1)$ by composing $\Om' \circ \delta_0$ with the quotient homomorphism $D \to D_1$. We have $\Om'_0 \sim \Om|_{G_0}$. Since $\Lambdatil$ is virtually isomorphic to the central extension defined by $\Om'_0$, the proof is complete.
\end{proof}

\begin{proof}[{Proof of Theorem \ref{thm.inherit-virtual-iso-superrigidity.2}, part 2}]
We still write $M = L(\Gtil)$, $P = L(\Lambdatil)$ and now assume that $H$ is a faithful bifinite $P$-$M$-bimodule. We prove in part~2 that $\Lambdatil$ is commensurable to a group in $\cF_\sim$. We start by repeating part~1 of the proof. In the first paragraph, we get that $z=1$, so that in the second paragraph, after replacing $\Lambdatil$ by a finite index subgroup, we get that $\Lambdatil\fc = \cZ(\Lambdatil)$. We again denote this abelian group by $D$ and define $\Lambda = \Lambdatil/D$. We find the central extension $0 \to D \to \Lambdatil \to \Lambda \to e$ and denote the associated $2$-cocycle by $\Om' \in H^2(\Lambda,D)$.

As in the proof of part~1, after a further passage to a finite index subgroup in $\Lambdatil$, we may assume that the group $\Lambda$ is icc. Also, by starting from a decomposition of $\bim{P}{H}{M}$ using Proposition \ref{prop.iso-on-center}, the proof of part~1 gives us, indexed by $i \in \N$, nonnegligible Borel sets $\cU_i \subset \Chat$, $\cV_i \subset \Dhat$, nonsingular isomorphisms $\Delta_i : \cU_i \to \cV_i$, finite index subgroups $G_i < G$, $\Lambda_i < \Lambda$ and group isomorphisms $\delta_i : G_i \to \Lambda_i$ such that
\begin{equation}\label{eq.equality-with-Delta-i}
\mu \circ \Om|_{G_i} = \Delta_i(\mu) \circ \Om' \circ \delta_i \;\;\text{in $H^2(G_i,\T)$, for all $i \in \N$ and all $\mu \in \cU_i$,}
\end{equation}
and such that $\bigcup_i \cU_i$ and $\bigcup_i \cV_i$ have a complement of measure zero, where the latter follows from the faithfulness of the bimodule $\bim{P}{H}{M}$. Take $n$ such that the normalized Haar measure of $\cV_0 := \bigcup_{i=1}^n \cV_i$ in $\Dhat$ is strictly larger than $1/2$. Define the finite index subgroup $\Lambda_0 < \Lambda$ by $\Lambda_0 = \bigcap_{i=1}^n \Lambda_i$. Define the continuous homomorphism $\rho : \Dhat \to H^2(\Lambda_0,\T) : \rho(\om) = \om \circ \Om'|_{\Lambda_0}$.

Fix $i \in \{1,\ldots,n\}$. Since $\Lambda_0 \subset \Lambda_i$ and $\delta_i : G_i \to \Lambda_i$ is an isomorphism, we have the well-defined continuous group homomorphism
$$\be_i : \Chat \to H^2(\Lambda_0,\T) : \be_i(\mu) = \mu \circ \Om \circ \delta_i^{-1}|_{\Lambda_0} \; .$$
Since $\cU_i \subset \be_i^{-1}(\rho(\Dhat))$ and $C$ is torsion free, we conclude as above that $\be_i(\Chat) \subset \rho(\Dhat)$. Since $\cV_i \subset \rho^{-1}(\be_i(\Chat))$, it similarly follows that $\rho^{-1}(\be_i(\Chat))$ is an open subgroup of $\Dhat$. Applying $\rho$, we conclude that $\be_i(\Chat) \subset \rho(\Dhat)$ is an open subgroup. Since $\Chat$ is connected, also $\be_i(\Chat)$ is connected. It follows that the connected component $K$ of the identity in $\rho(\Dhat)$ is an open subgroup of $\rho(\Dhat)$ and that $\be_i(\Chat) = K$.

Since $\cV_i \subset \rho^{-1}(\be_i(\Chat))$, we conclude that $\cV_i \subset \rho^{-1}(K)$ for every $i \in \{1,\ldots,n\}$. So, $\rho^{-1}(K)$ is an open subgroup of $\Dhat$ whose normalized Haar measure is strictly larger than $1/2$. This means that $\rho^{-1}(K) = \Dhat$, and thus $\rho(\Dhat) = \be_i(\Chat)$ for all $i \in \{1,\ldots,n\}$.

We take $i=1$ and define $G_0 = \delta_1^{-1}(\Lambda_0)$. Since $\Lambda_0 < \Lambda_1$ has finite index, also $G_0 < G_1$ has finite index, so that $G_0 < G$ has finite index. Define the subgroup $C_0 < C$ such that $\widehat{C/C_0}$ is the kernel of the continuous group homomorphism $\lambda : \Chat \to H^2(G_0,\T) : \lambda(\mu) = \mu \circ \Om|_{G_0}$. Since $\rho$ is a surjective homomorphisms of $\Dhat$ onto $\be_1(\Chat)$, there is a unique faithful group homomorphism $\zeta : C_0 \to D$ such that $\mu \circ \Om|_{G_0} = \om \circ \Om' \circ \delta_1|_{G_0}$ in $H^2(G_0,\T)$ whenever $\mu \in \Chat$ and $\om \in \Dhat$ satisfy $\mu|_{C_0} = \om \circ\zeta$. Write $D_0 = \zeta(C_0)$.

In the same way as at the end of the proof of part~1, we conclude that $[C:C_0]$ and $[D:D_0]$ are either both finite, or both infinite. Define $\Om'_0 \in H^2(G_0,D)$ by $\Om'_0 = \Om' \circ \delta_1|_{G_0}$. Then $\Om|_{G_0} \sim \Om'_0$. By construction, $\Lambdatil$ is commensurable to the central extension defined by $\Om'_0$, so that the proof is complete.
\end{proof}

\begin{proof}[{Proof of Theorem \ref{thm.inherit-virtual-iso-superrigidity.3}}]
Take a central extension $\Lambdatil$ of a finite index subgroup $G_0 < G$ by a countable abelian group $D$ such that the associated $2$-cocycle $\Om' \in H^2(G_0,D)$ satisfies $\Om|_{G_0} \sim \Om'$. It suffices to prove that $\Om|_{G_0} \approx\ve \Om'$.

Take $C_0 < C$, $D_0 < D$ and a group isomorphism $\zeta : C_0 \to D_0$ as in Definition \ref{def.equiv-2-cocycle-bis.1}. It follows from \eqref{eq.cocycle-equal-bis} that $\mu \in \Chat$ belongs to the kernel of $\Theta(\Om|_{G_0})$ whenever $\mu$ is equal to $1$ on $C_0$. Since $\Theta(\Om|_{G_0})$ has finite kernel, $C_0 < C$ has finite index. Since $[C:C_0]$, $[D:D_0]$ are supposed to be either both finite, or both infinite, it follows that also $D_0 < D$ has finite index. So, $\Om|_{G_0} \approx\ve \Om'$.
\end{proof}

\begin{proof}[{Proof of Theorem \ref{thm.inherit-virtual-iso-superrigidity.4} and \ref{thm.inherit-virtual-iso-superrigidity.5}}]
Take a central extension $\Lambdatil$ of a finite index subgroup $G_0 < G$ by a countable abelian group $D$ such that the associated $2$-cocycle $\Psi \in H^2(G_0,D)$ satisfies $\Om|_{G_0} \approx\ve \Psi$. Assume that condition \ref{ncond.iso.4} holds. We prove that $\Lambdatil$ is virtually isomorphic to $\Gtil$. We then assume that also condition \ref{ncond.iso.5} holds and prove that $\Lambdatil$ is commensurable to $\Gtil$.

By the universal coefficient theorem \eqref{eq.univ-coeff}, condition \ref{ncond.iso.4}, resp.\ \ref{ncond.iso.5}, can be reformulated in the following way: if $G_0 < G$ is a finite index subgroup and $\Phi \in H^2(G_0,D)$ is a $2$-cocycle with values in a torsion free abelian group $D$ that is commensurable with $C$, resp.\ a finite abelian group $D$, and if $\om \circ \Phi = 1$ in $H^2(G_0,\T)$ for all $\om \in \Dhat$, there exists a finite index subgroup $G_1 < G$ such that $\Phi|_{G_1} = 0$ in $H^2(G_1,D)$.

Since $\Om|_{G_0} \approx\ve \Psi$, we can take an isomorphism $\zeta : C_0 \to D_0$ between finite index subgroups $C_0 < C$ and $D_0 < D$ such that \eqref{eq.cocycle-equal-bis} holds. Denote by $D\tor < D$ the subgroup of elements of finite order. Since $C$ is torsion free, also $C_0$ is torsion free, so that $\zeta(C_0) \cap D\tor = \{0\}$. Since $\zeta(C_0) < D$ has finite index, it follows that $D\tor$ is finite. So $D\tor$ is a finite normal subgroup of $\Lambdatil$ and $\Lambdatil/D\tor$ can be viewed as the central extension of $G_0$ by $D_1 := D/D\tor$, with $2$-cocycle $\Psi_1 \in H^2(G_0,D_1)$ given by composing $\Psi$ with the quotient homomorphism $D \to D_1$. The composition of $\zeta$ with the quotient homomorphism remains a faithful homomorphism $\zeta_1 : C_0 \to D_1$. Note that $D_1$ is torsion free and virtually isomorphic to $C$. By construction, $\mu \circ \Om|_{G_0} = \om \circ \Psi_1$ in $H^2(G_0,\T)$ whenever $\mu \in \Chat$ and $\om \in \widehat{D_1}$ satisfy $\om \circ \zeta_1 = \mu|_{C_0}$.

Define $D_2 = \zeta_1(C_0)$. Since $\om \circ \Psi_1 = 1$ in $H^2(G_0,\T)$ whenever $\om \in \widehat{D_1}$ is equal to $1$ on $D_2$, we obtain a well-defined continuous homomorphism $\al : \widehat{D_2} \to H^2(G_0,\T)$ satisfying $\al(\om|_{D_2}) = \om \circ \Psi_1$ for every $\om \in \widehat{D_1}$. By the surjectivity of $\Theta$ in the universal coefficient theorem \eqref{eq.univ-coeff}, we can choose $\Psi_0 \in H^2(G_0,D_2)$ such that $\om \circ \Psi_0 = \al(\om)$ for all $\om \in \widehat{D_2}$.

The difference $\Psi_1 - \Psi_0$ defines a $2$-cocycle $\Psi_2 \in H^2(G_0,D_1)$ satisfying $\om \circ \Psi_2 = 1$ for all $\om \in \widehat{D_1}$. Since $D_1$ is torsion free and virtually isomorphic to $C$, condition \ref{ncond.iso.4} gives us a finite index subgroup $G_1 < G_0$ such that $\Psi_2|_{G_1} = 0$ in $H^2(G_1,D_1)$. This means that $\Psi_0|_{G_1} = \Psi_1|_{G_1}$ in $H^2(G_1,D_1)$.

Since $\mu \circ \Om|_{G_0}=1$ for all $\mu \in \Chat$ that are equal to $1$ on $C_0$, after a similar application of \eqref{eq.univ-coeff} and condition \ref{ncond.iso.4}, and after replacing $G_1$ by a smaller finite index subgroup, we also find $\Om_0 \in H^2(G_1,C_0)$ such that $\Om|_{G_1} = \Om_0$ in $H^2(G_1,C)$.

Note that $\zeta_1 : C_0 \to D_2$ is an isomorphism. By construction, $\om \circ \Psi_0 = (\om \circ \zeta_1) \circ \Om_0$ in $H^2(G_1,\T)$ for all $\om \in \widehat{D_2}$. After a third application of condition \ref{ncond.iso.4} and replacing $G_1$ by a smaller finite index subgroup, we get that $\zeta_1^{-1} \circ \Psi_0 = \Om_0$ in $H^2(G_1,C_0)$. So the central extension of $G_1$ by $C_0$ defined by $\Om_0$ is isomorphic to the central extension of $G_1$ by $D_2$ defined by $\Psi_0$. The first of these groups is isomorphic to a finite index subgroup of $\Gtil$, while the second of these groups is isomorphic to a finite index subgroup of $\Lambdatil/D\tor$. So $\Lambdatil$ and $\Gtil$ are virtually isomorphic.

Finally assume that also condition \ref{ncond.iso.5} holds. Define the finite abelian group $E = D/D_0$ and define $\Psi' \in H^2(G_0,E)$ by composing $\Psi$ with the quotient homomorphism $\pi : D \to E$. Since $\om \circ \Psi = 1$ in $H^2(G_0,\T)$ for all $\om \in \Dhat$ that are equal to $1$ on $D_0$, we get that $\om \circ \Psi' = 1$ in $H^2(G_0,\T)$ for all $\om \in \widehat{E}$. By condition \ref{ncond.iso.5}, we find a finite index subgroup $G_1 < G_0$ such that $\Psi'|_{G_1} = 0$ in $H^2(G_1,E)$.

Choosing a representative $\Psi \in Z^2(G_0,D)$, it follows that $\pi \circ \Psi|_{G_1} = \partial \vphi$ for some map $\vphi : G_1 \to E$. Choose a map $\psi : G_1 \to D$ such that $\pi \circ \psi = \vphi$. Then $\Psi_1 := \Psi|_{G_1} - \partial \psi$ belongs to $Z^2(G_1,D_0)$. Define $\Lambdatil_1$ as the central extension of $G_1$ by $D_0$ with $2$-cocycle $\Psi_1$. By construction, $\Lambdatil_1$ is isomorphic with a finite index subgroup of $\Lambdatil$.

We still have that $\Psi_1 \approx\ve \Om|_{G_1}$. Since $\Psi_1 \in H^2(G_1,D_0)$ and $D_0$ is torsion free, in the proof of (d) above, the torsion part is zero and we conclude that $\Lambdatil_1$ is commensurable to $\Gtil$.
\end{proof}

\begin{remark}\label{rem.relax-v-iso}
We make a similar remark as in \ref{rem.relax-iso} about relaxing conditions \ref{cond.iso.1prime} and \ref{ncond.iso.2prime}, which will be used in the proof of Proposition \ref{prop.new-no-go-result-3}. Again, given a concrete central extension $0 \to C \to \Gtil \to G \to e$ with associated $2$-cocycle $\Om \in H^2(G,C)$, we only used that condition \ref{cond.iso.1prime} holds for the specific $2$-cocycles in $H^2(G,\T)$ of the form $\mu \circ \Om$, $\mu \in \Chat$.

Secondly, in our applications of condition \ref{ncond.iso.2prime}, it would be sufficient that $G$ has only countable many finite index subgroups and that we have a countable family $(L_n,R_n,\delta_n)$, $n \in \N$, of finite index subgroups $L_n,R_n < G$ and isomorphisms $\delta_n : L_n \to R_n$ such that for every triple $(G_0,G_1,\delta)$ of finite index subgroups $G_0,G_1 < G$ and isomorphism $\delta : G_0 \to G_1$, there exists an $n \in \N$ and an automorphism $\al \in \Aut G$ such that $\al(G_0) = L_n$, $\delta = \delta_n \circ \al|_{G_0}$ and $\mu \circ \Om \circ \al = \mu \circ \Om$ in $H^2(G,\T)$ for all $\mu \in \Chat$.
\end{remark}

In the proof of Theorem \ref{thm.inherit-virtual-iso-superrigidity}, we made use of the group theoretic Lemma \ref{lem.automatic-icc} whose proof relies on the following lemma.

\begin{lemma}\label{lem.finite-index-in-abelianization}
Let $G$ be a countable group that has only countably many finite index subgroups. If $G_0 < G$ is a finite index subgroup and $G_{0,\text{\rm ab}}$ denotes its abelianization, which we write additively, then for every $n \in \N$, $n \cdot G_{0,\text{\rm ab}}$ is a finite index subgroup of $G_{0,\text{\rm ab}}$.
\end{lemma}
\begin{proof}
Since also $G_0$ has only countably many finite index subgroups, also its quotient $G_{0,\text{\rm ab}}$ has only countably many finite index subgroups. We may thus assume that $G$ is abelian, write $G$ additively and prove that for every $n \in \N$, $n \cdot G$ has finite index in $G$. We prove this by induction on $n$. The statement is trivial for $n=1$. So assume that $n > 1$ and that $[G:k \cdot G] < \infty$ for all integers $1 \leq k < n$ and all abelian groups $G$ with countably many finite index subgroups. We prove that $[G:n \cdot G] < \infty$ for all abelian groups $G$ with countably many finite index subgroups.

Take a prime number $p$ such that $n = p k$ with $k \in \N$. Define $G_0 = k \cdot G$. By the induction hypothesis, $G_0 < G$ has finite index. So, $G_0$ is an abelian group with countably many finite index subgroups. Since $n \cdot G = p \cdot G_0$, it suffices to prove that $[G_0:p \cdot G_0] < \infty$. Assume the contrary and define $K = G_0 / p \cdot G_0$. Then $K$ is an infinite abelian group in which every nontrivial element has order $p$. This implies that $K \cong (\Z/p\Z)^{(\N)}$. For every $n \in \N$, denote by $a_n \in K$ the generator $1 \in \Z/p\Z$ in position $n$. For every subset $\cF \subset \N$, define the group homomorphism $\vphi_\cF : K \to \Z/p\Z$ by
$$\vphi_\cF(a_n) = \begin{cases} 1 &\;\;\text{if $n \in \cF$,}\\
0 &\;\;\text{if $n \not\in \cF$.}\end{cases}$$
Then, $(\Ker \vphi_\cF)_{\cF \subset \N}$ is an uncountable family of index $p$ subgroups of $K$, whose preimages in $G_0$ provide an uncountable family of index $p$ subgroups of $G_0$, which is impossible. So the lemma is proven.
\end{proof}

\begin{lemma}\label{lem.automatic-icc}
Let $0 \to D \to \Lambdatil \to \Lambda \to e$ be any central extension of countable groups such that $\Lambdatil\fc = \cZ(\Lambdatil) = D$. Assume that $\Lambda$ is virtually isomorphic to an icc group that admits only countably many finite index subgroups. Then $\Lambda$ admits a finite index subgroup that is icc.
\end{lemma}
\begin{proof}
Take an icc group $G$ with only countably many finite index subgroups such that $\Lambda$ is virtually isomorphic to $G$. Finite index subgroups of $G$ are still icc and therefore have no nontrivial finite normal subgroups. We can thus choose a finite index subgroup $\Lambda_1 < \Lambda$ and a group homomorphism $\phi : \Lambda_1 \to G$ such that $F = \Ker \phi$ is finite and $\phi(\Lambda_1) < G$ has finite index.

Define $\Lambda_0 = C_{\Lambda_1}(F)$. Since $F$ is a finite normal subgroup of $\Lambda_1$, we have that $\Lambda_0 < \Lambda_1$ has finite index. Write $F_0 = F \cap \Lambda_0$. Since $G$ is icc and $\phi$ restricts to an isomorphism between $\Lambda_0 / F_0$ and a finite index subgroup of $G$, it suffices to prove that $F_0 = \{e\}$.

We write the group operation in the abelian group $D$ additively. Denote by $\Psi \in H^2(\Lambda,D)$ the $2$-cocycle associated with the central extension $0 \to D \to \Lambdatil \to \Lambda \to e$. Since $F_0$ and $\Lambda_0$ commute, the map
$$\eta : F_0 \times \Lambda_0 \to D : \eta(\si,s) = \Psi(\si,s) - \Psi(s,\si)$$
is a bihomomorphism, meaning that for fixed $\si_0 \in F_0$ and $s_0 \in \Lambda_0$, the maps $\si \mapsto \eta(\si,s_0)$ and $s \mapsto \eta(\si_0,s)$ are group homomorphisms.
Since $F_0$ is finite and $\eta$ is a bihomomorphism, $\eta(F_0 \times F_0) \subset D$ is a finite subset of elements of finite order. Since $D$ is abelian, the subgroup $D_0 < D$ generated by $\eta(F_0 \times F_0)$ is still finite. We get the well-defined bihomomorphism $\eta_0 : F_0 \times \Lambda_0/F_0 \to D/D_0 : \eta_0(\si,s F_0) =\eta(\si,s) + D_0$. Defining the finite index subgroup $G_0 < G$ as $G_0 = \phi(\Lambda_0)$, we find a unique bihomomorphism $\gamma : F_0 \times G_0 \to D/D_0$ such that $\gamma \circ (\id \times \phi) = \eta_0$.

As a bihomomorphism to an abelian group, $\gamma$ factorizes through $\gamma_0 : F_0 \times G_{0,\text{\rm ab}} \to D/D_0$. Let $n = |F_0|$. We write the group operation in $G_{0,\text{\rm ab}}$ additively and the one in $F_0$ multiplicatively. Since $\si^n = e$ for all $\si \in F_0$, it follows that
$$\gamma_0(\si,n \cdot g) = n \cdot \gamma_0(\si,g) = \gamma_0(\si^n,g) = 0 \quad\text{for all $\si \in F_0$ and $g \in G_{0,\text{\rm ab}}$.}$$
Since $G$ has only countably many finite index subgroups, it follows from Lemma \ref{lem.finite-index-in-abelianization} that $n \cdot G_{0,\text{\rm ab}}$ has finite index in $G_{0,\text{\rm ab}}$. Since $\gamma_0(\si,n \cdot g)=0$ for all $\si \in F_0$ and $g \in G_{0,\text{\rm ab}}$ and $\gamma_0$ is a bihomomorphism, we conclude that $\gamma_0(F_0 \times G_{0,\text{\rm ab}}) \subset D/D_0$ is a finite subset of elements of finite order. Since $D/D_0$ is an abelian group, it follows that the subgroup of $D/D_0$ generated by $\eta_0(F_0 \times \Lambda_0/F_0)$ is finite. Since $D_0$ is a finite subgroup of $D$, we conclude that the subgroup $D_1$ of $D$ generated by $\eta(F_0 \times \Lambda_0)$ is finite as well.

Let $\theta : \Lambda \to \Lambdatil$ be a lift such that $\theta(s_1)\theta(s_2) = \Psi(s_1,s_2) \theta(s_1 s_2)$ for all $s_1,s_2 \in \Lambda$. Fix $\si \in F_0$. By definition of $\eta$, we find that $\theta(s)^{-1} \theta(\si) \theta(s) \in D_1 \theta(\si)$ for all $s \in \Lambda_0$. Since every element of $D = \cZ(\Lambdatil)$ commutes with $\theta(\si)$, it follows that $s^{-1} \theta(\si) s \in D_1 \theta(\si)$ for all $s$ in the preimage $\Lambdatil_0$ of $\Lambda_0$ in $\Lambdatil$. Since $\Lambda_0 < \Lambda$ has finite index, also $\Lambdatil_0 < \Lambdatil$ has finite index. It follows that $\theta(\si) \in \Lambdatil\fc = D$. So, $\si = e$. Since $\si \in F_0$ is arbitrary, we have proven that $F_0 = \{e\}$.
\end{proof}

\section{Weakly compact actions and Bernoulli crossed products}\label{sec.weakly-compact}

We prove in this section that certain left-right wreath product groups $G$ satisfy condition \ref{ncond.iso.3}. This follows from two technical Lemmas \ref{lem.weak-compact-Bernoulli} and \ref{lem.left-right-wreath-product-satisfies-4}, whose proofs consist of repackaging several known results from \cite{Ioa06,OP07,PV12} and are thus written concisely.

A trace preserving group action $\cG \actson^\al (B,\tau)$ is said to be \emph{compact} if the closure of $\al(\cG)$ in $\Aut(B,\tau)$ is compact. In \cite[Definition 3.1]{OP07}, the notion of a weakly compact trace preserving action was introduced. Rather than recalling the precise definition, we mention here that it was proven in \cite[Proposition 3.2]{OP07} that every compact trace preserving action on an amenable tracial von Neumann algebra is weakly compact.

\begin{lemma}\label{lem.weak-compact-Bernoulli}
Let $\Gamma \actson I$ be any transitive action of a countable group $\Gamma$ on a countable set $I$. Let $(A_0,\tau_0)$ and $(N,\tau)$ be arbitrary tracial von Neumann algebras. Consider $(A,\tau) = (A_0,\tau_0)^I$ with the generalized Bernoulli action $\Gamma \actson A$ and put $M = A \rtimes \Gamma$. Write $\cM = N \ovt M$.

Let $p \in \cM$ be a projection and $B \subset p \cM p$ a von Neumann subalgebra such that $B \prec^f_\cM N \ovt A$. Let $\cG \subset \cN_{p\cM p}(B)$ be a subgroup such that the action $(\Ad v)_{v \in \cG}$ of $\cG$ on $B$ is weakly compact. Write $P = (B \cup \cG)\dpr$. Let $i_0 \in I$.

If $P$ is strongly nonamenable relative to $N \ovt (A \rtimes \Stab i_0)$, then $B \prec^f_\cM N \ot 1$.
\end{lemma}

\begin{proof}
Since we may replace $B$ and $\cG$ by $B q$ and $\cG q$ for any nonzero projection $q \in P' \cap p \cM p$, it suffices to prove that $B \prec_\cM N \ot 1$.

Define the malleable deformation of $\Gamma \actson A$ given by \cite[Proposition 2.3]{Ioa06}: we define $\Atil = (A_0 \ast L(\Z))^I$ w.r.t.\ the canonical traces and define $\theta_t \in \Aut(\Atil)$ as $\theta_t = (\Ad u_t)^I$, where $u_t \in \cU(L(\Z))$ is a concrete $1$-parameter group of unitaries. We extend $\theta_t$ to $\cMtil = N \ovt (\Atil \rtimes \Gamma)$ by acting as the identity on $N$ and $\Gamma$.

Since the action $(\Ad v)_{v \in \cG}$ of $\cG$ on $B$ is weakly compact, it follows from \cite[Definition 3.1]{OP07} that we can take a net of vectors $(\xi_n)_{n \in \cJ}$ in $L^2(B \ovt B\op)$ such that, viewing $\xi_n \in L^2(\cM \ovt \cM\op)$, we have
\begin{itemlist}
\item $\|(b \ot 1) \xi_n - (1 \ot b\op) \xi_n \|_2 \to 0$ for all $b \in B$,
\item $\|(v \ot \overline{v}) \xi_n - \xi_n (v \ot \overline{v})\|_2 \to 0$ for all $v \in \cG$,
\item $\langle (x \ot 1) \xi_n , \xi_n \rangle = \tau(x) = \langle (1 \ot x\op)\xi_n,\xi_n\rangle$ for all $x \in p\cM p$.
\end{itemlist}
Write $\cM_0 = N \ovt (A \rtimes \Stab i_0)$. By \cite[Proof of Corollary 4.3]{IPV10} (see also \cite[Lemma 3.2]{BV12}), the $\cM$-$\cM$-bimodule $L^2(\cMtil \ominus \cM)$ is weakly contained in $L^2(\cM) \ot_{\cM_0} L^2(\cM)$. Repeating the proof of \cite[Theorem 4.9]{OP07}, we first use that $P$ is strongly nonamenable relative to $\cM_0$ to conclude that
$$\lim_{t \to 0 \, , \, n \in \cJ} \|(\theta_t \ot \id)(\xi_n) - \xi_n\|_2 = 0 \; .$$
We then deduce that $\|\theta_t(b) - b\|_2 \to 0$ as $t \to 0$, uniformly on the unit ball $(B)_1$.

Now the proof of \cite[Theorem 3.6(ii)]{Ioa06} implies that $B \prec_\cM N \ovt A_0^{I_0}$ for some finite subset $I_0$. Assume that $B \not\prec_\cM N \ot 1$. Making $I_0$ smaller, we may assume that $I_0 \neq \emptyset$ and $B \not\prec_\cM N \ovt A_0^{I_1}$ for every proper subset $I_1 \subset I_0$. By \cite[Lemma 4.1.2]{IPV10}, we get that $P \prec_\cM N \ovt (A_0^{I_0} \rtimes \Stab I_0)$. This contradicts the strong nonamenability of $P$ relative to $\cM_0$ and the lemma is proven.
\end{proof}

\begin{lemma}\label{lem.left-right-wreath-product-satisfies-4}
Let $\Gamma$ be a weakly amenable, nonamenable, biexact group and let $\Sigma_0$ be a countable group. Consider the left-right action $\Gamma \times \Gamma \actson \Sigma = \Sigma_0^{(\Gamma)}$ and let $\Sigma_1 < \Sigma$ be a globally invariant subgroup. Then $G = \Sigma_1 \rtimes (\Gamma \times \Gamma)$ satisfies condition \ref{ncond.iso.3}.
\end{lemma}
\begin{proof}
Define $G_1 = \Sigma_1 \rtimes (\Gamma \times e)$, $G_2 = \Sigma_1 \rtimes (e \times \Gamma)$ and $G_3 = \Sigma_1 \rtimes \delta(\Gamma)$ where $\delta(g) = (g,g)$. Since $\Gamma$ is nonamenable, the subgroups $G_i < G$ are not co-amenable.

Take a tracial von Neumann algebra $(N,\tau)$, a nonzero projection $p \in N \ovt L(G)$, a von Neumann subalgebra $B \subset p(N \ovt L(G))p$ and a subgroup $\cG \subset \cN_{p(N \ovt L(G)) p}(B)$. Assume that $B$ is amenable and that the action $(\Ad v)_{v \in \cG}$ of $\cG$ on $B$ is compact. Write $P = (B \cup \cG)\dpr$ and assume that $P$ is strongly nonamenable relative to $N \ovt L(G_i)$ for all $i \in \{1,2,3\}$.

Because $P$ is strongly nonamenable relative to $N \ovt L(G_1)$ and relative to $N \ovt L(G_2)$, it follows from \cite[Theorem 1.4]{PV12} that $B \prec^f N \ovt L(\Sigma_1)$. We define $\Gtil = \Sigma \rtimes (\Gamma \times \Gamma)$ and $\Gtil_3 = \Sigma \rtimes \delta(\Gamma)$. Since $B \prec^f N \ovt L(\Sigma_1)$ inside $N \ovt L(G)$, a fortiori $B \prec^f N \ovt L(\Sigma)$ inside $N \ovt L(\Gtil)$. Since $P$ is strongly nonamenable relative to $N \ovt L(G_3)$ inside $N \ovt L(G)$, it follows from \cite[Lemma 2.2]{DV24} that $P$ is strongly nonamenable relative to $N \ovt L(\Gtil_3)$ inside $N \ovt L(\Gtil)$.

By \cite[Proposition 3.2]{OP07}, the action $(\Ad v)_{v \in \cG}$ of $\cG$ on $B$ is weakly compact. Since the stabilizers of the left-right action $\Gamma \times \Gamma \actson \Gamma$ are conjugate to $\delta(\Gamma)$, it thus follows from Lemma \ref{lem.weak-compact-Bernoulli} that $B \prec^f N \ot 1$ inside $N \ovt L(\Gtil)$. By \cite[Lemma 2.2]{DV24}, we conclude that also $B \prec^f N \ot 1$ inside $N \ovt L(G)$.
\end{proof}

\section{\boldmath Elementary $2$-cohomology computations}\label{sec.elementary-2-cohom-comput}

As above, we write the group operation in an abelian group additively, with the circle $\T$ being the exception. Given groups $\Gamma_1,\Gamma_2$ and an abelian group $C$, we denote by $\BHom(\Gamma_1,\Gamma_2,C)$ the abelian group of bihomomorphisms $\gamma : \Gamma_1 \times \Gamma_2 \to C$, i.e.\ the group of maps $\gamma : \Gamma_1 \times \Gamma_2 \to C$ with the property that $g \mapsto \gamma(g,h_0)$ and $h \mapsto \gamma(g_0,h)$ are group homomorphisms for all $g_0 \in \Gamma_1$ and $h_0 \in \Gamma_2$.

When $\Gamma \actson^\al \Sigma$ is an action of a group $\Gamma$ by automorphisms of a group $\Sigma$ and when $C$ is an abelian group, we consider the abelian group $\cZ(\Gamma \actson \Sigma,C)$ of maps $\eta : \Gamma \times \Sigma \to C$ satisfying
$$\eta(g,ab) = \eta(g,a) + \eta(g,b) \quad\text{and}\quad \eta(gh,a) = \eta(g,\al_h(a)) + \eta(h,a) \quad\text{for all $g,h \in \Gamma$, $a,b \in \Sigma$.}$$
We define the subgroup $\cB(\Gamma \actson \Sigma,C) < \cZ(\Gamma \actson \Sigma,C)$ of maps of the form $\eta(g,a) = \vphi(\al_g(a)) - \vphi(a)$ for some $\vphi \in \Hom(\Sigma,C)$, and we denote by $\cH(\Gamma \actson \Sigma,C)$ the quotient.

Recall that an abelian group $C$ is called $p$-divisible, where $p$ is a prime number, if for every $c \in C$, there exists a $d \in C$ with $p d = c$.

The following elementary lemma is surely well known, but we did not find a reference that includes all the explicit formulas that we need later on. So we provide a complete and easy proof.

\begin{lemma}\phantomsection\label{lem.basic-cohomology}
\begin{enuma}
\item\label{lem.basic-cohomology.1} Let $\Gamma \actson^\al \Sigma$ be an action of a group $\Gamma$ by automorphisms of a group $\Sigma$. Let $C$ be an abelian group. Consider the semidirect product $G = \Sigma \rtimes \Gamma$ with $(a,g) \cdot (b,h) = (a\al_g(b),gh)$ for all $a,b \in \Sigma$, $g, h \in \Gamma$. Then
    \begin{equation}\label{eq.cohom-short-exact}
    0 \to H^2(\Gamma,C) \oplus \cH(\Gamma \actson \Sigma,C) \overset{\Psi}{\longrightarrow} H^2(G,C) \overset{\res}{\longrightarrow} H^2(\Sigma,C) \; ,
    \end{equation}
    where $(\Psi(\om,\eta))((a,g),(b,h)) = \om(g,h) + \eta(g,b)$ and $\res$ is the restriction to $\Sigma$, is an exact sequence of well-defined group homomorphisms.

\item\label{lem.basic-cohomology.2} Let $\Sigma$ and $C$ be abelian groups. Assume that $\Sigma$ can be written as the union of an increasing sequence of subgroups $\Sigma_n < \Sigma$ such that $\Sigma_1 = \{0\}$ and for every $n \in \N$, the group $\Sigma_{n+1}/\Sigma_n$ is either isomorphic to $\Z$, or isomorphic to $\Z/N_n\Z$ with $C$ being $p$-divisible for every prime divisor $p$ of $N_n$.
    Then $\Ext^1(\Sigma,C) = 0$ and the homomorphism
\begin{equation}\label{eq.hom-for-abelian}
\Xi : H^2(\Sigma,C) \to \BHom(\Sigma,\Sigma,C) : (\Xi(\Om))(a,b) = \Om(a,b) - \Om(b,a)
\end{equation}
is well-defined and faithful.

\item\label{lem.basic-cohomology.3} If $\Gamma_1,\Gamma_2$ are groups and $C$ is an abelian group, then
\begin{align*}
& \Theta_1 : H^2(\Gamma_1,C) \oplus H^2(\Gamma_2,C) \oplus \BHom(\Gamma_1,\Gamma_2,C) \to H^2(\Gamma_1 \times \Gamma_2,C) : (\om_1,\om_2,\gamma) \mapsto \Om\\
& \hspace{3cm}\text{with}\;\; \Om((g_1,g_2),(h_1,h_2)) = \om_1(g_1,h_1) + \om_2(g_2,h_2) + \gamma(g_2,h_1) \; ,\\[1ex]
& \Theta_2 : H^2(\Gamma_1 \ast \Gamma_2 , C) \to H^2(\Gamma_1,C) \oplus H^2(\Gamma_2,C) : \Theta_2(\Om) = (\Om|_{\Gamma_1},\Om|_{\Gamma_2})
\end{align*}
are well-defined group isomorphisms.
\end{enuma}
\end{lemma}

Note that the group $H^2(\Sigma,C)$ appearing at the right of \eqref{eq.cohom-short-exact} is in general too large for the restriction $\res$ to be surjective. Also the result in \eqref{eq.hom-for-abelian} is not sharp, as one can describe the image of $\Xi$. But the formulations above suffice for our purposes and have a short proof.

\begin{proof}
(a) Whenever $\om \in Z^2(\Gamma,C)$ and $\eta \in \cZ(\Gamma \actson \Sigma,C)$, one turns $\Gtil = C \times (\Sigma \rtimes \Gamma)$ into a group by
$$(c,a,g) \cdot (d,b,h) = (c+d+\om(g,h)+\eta(g,b),a \al_g(b),gh) \quad\text{for all $c,d \in C$, $a,b \in \Sigma$, $g,h \in \Gamma$.}$$
With $q(c,a,g) = (a,g)$ and $c \mapsto (c,e,e)$, we obtain the central extension $0 \to C \to \Gtil \to G \to e$. By construction, the associated $2$-cocycle is $\Psi(\om,\eta)$. It is easy to check that this central extension is split if and only if $\om \in B^2(\Gamma,C)$ and $\eta \in \cB(\Gamma \actson \Sigma,C)$. So, $\Psi$ is a well-defined faithful group homomorphism.

To conclude the proof of (a), take $\Om \in H^2(G,C)$ such that the restriction of $\Om$ to $\Sigma$ is a trivial $2$-cocycle. We have to prove that $\Om$ lies in the image of $\Psi$. Denote by $\Gtil$ the central extension corresponding to $\Om$, with quotient map $q : \Gtil \to G$. We view $\Sigma$ and $\Gamma$ as subgroups of $G$.

Since $\res(\Om)$ is trivial, we can choose a homomorphism $\theta : \Sigma \to \Gtil$ such that $q \circ \theta = \id$. We choose any map $\vphi : \Gamma \to \Gtil$ satisfying $q \circ \vphi = \id$ and $\vphi(e) = e$. Define $\om \in Z^2(\Gamma,C)$ by $\vphi(g)\vphi(h) = \om(g,h) \, \vphi(gh)$ for all $g,h \in \Gamma$. Define the map $\eta : \Gamma \times \Sigma \to C$ by $\vphi(g) \theta(a) = \eta(g,a) \, \theta(\al_g(a)) \vphi(g)$. It is easy to check that $\eta \in \cZ(\Gamma \actson \Sigma,C)$ and that $\Om = \Psi(\om,\eta)$.

(b) First assume that $0 \to C \to B \to \Sigma \to 0$ is an extension with $B$ abelian and denote by $q : B \to \Sigma$ the quotient homomorphism. We have to prove that there exists a group homomorphism $\theta : \Sigma \to B$ satisfying $q \circ \theta = \id$. We write the group operations in $C$, $B$ and $\Sigma$ additively.

We inductively construct group homomorphisms $\theta_n : \Sigma_n \to B$ such that $q \circ \theta_n = \id$ and $\theta_{n+1}|_{\Sigma_n} = \theta_n$. Once this is done, we can uniquely define $\theta$ such that $\theta|_{\Sigma_n} = \theta_n$ for all $n$, so that (b) is proven.

We start by defining $\theta_1 : \Sigma_1 = \{0\} \to B$ trivially. Assume that $\theta_1,\ldots,\theta_n$ have been constructed. If $\Sigma_{n+1}/\Sigma_n$ is an infinite cyclic group, we can take $a \in \Sigma_{n+1}$ such that $a+\Sigma_n$ generates $\Sigma_{n+1}/\Sigma_n$. Choose an arbitrary $b \in B$ such that $q(b) = a$. Since $k a \not\in \Sigma_n$ for all $k \in \Z \setminus \{0\}$,
$$\theta_{n+1} : \Sigma_{n+1} \to B : \theta_{n+1}(k a + d) = k b + \theta_n(d) \quad\text{for all $k \in \Z$, $d \in \Sigma_n$,}$$
is a well-defined group homomorphism.

If $\Sigma_{n+1}/\Sigma_n$ is a finite cyclic group of order $N$, by assumption, $C$ is $p$-divisible for every prime divisor of $N$. We choose $a \in \Sigma_{n+1}$ such that $a+\Sigma_n$ generates $\Sigma_{n+1}/\Sigma_n$. We choose an arbitrary $b \in B$ such that $q(b) = a$. Since $N a \in \Sigma_n$, we find that $N b - \theta_n(N a) \in C$. Since every element of $C$ can be divided by $N$, we can choose $c \in C$ such that $N c = N b - \theta_n(N a)$. We now unambiguously define the group homomorphism
$$\theta_{n+1} : \Sigma_{n+1} \to B : \theta_{n+1}(k a + d) = k(b-c) + \theta_n(d) \quad\text{for all $k \in \Z$, $d \in \Sigma_n$.}$$
In both cases, the homomorphisms $\theta_n$ combine into a well-defined group homomorphism $\theta : \Sigma \to B$. So, we have proven that $\Ext^1(\Sigma,C) = 0$. It is easy to check that $\Xi$ is a well-defined group homomorphism and that every $\Om \in H^2(\Sigma,C)$ with $\Xi(\Om)=0$ defines an abelian extension. Since we have already proven that $\Ext^1(\Sigma,C) = 0$, it follows that $\Xi$ is faithful.

(c) It is immediate that $\Theta_1$ is a well-defined group homomorphism. Also
\begin{multline*}
\Psi : H^2(\Gamma_1 \times \Gamma_2,C) \to H^2(\Gamma_1,C) \oplus H^2(\Gamma_2,C) \oplus \BHom(\Gamma_1,\Gamma_2,C) : \\
\Psi(\Om) = (\Om|_{\Gamma_1 \times e},\Om|_{e \times \Gamma_2},\gamma) \quad\text{where $\gamma(g,h) = \Om((e,g),(h,e)) - \Om((h,e),(e,g))$}
\end{multline*}
is a well-defined group homomorphism. By construction, $\Psi \circ \Theta_1 = \id$. It thus suffices to prove that $\Psi$ is faithful. Take $\Om$ in $H^2(\Gamma_1 \times \Gamma_2,C)$ such that $\Psi(\Om) = 0$ and denote by $0 \to C \to G \to \Gamma_1 \times \Gamma_2 \to e$ the corresponding central extension, with quotient homomorphism $q : G \to \Gamma_1 \times \Gamma_2$.

Since the restrictions of $\Om$ to $\Gamma_1$ and $\Gamma_2$ are trivial, we can choose homomorphisms $\theta_i : \Gamma_i \to G$ such that $q \circ \theta_i = \id$. Since $\Om((e,g),(h,e)) - \Om((h,e),(e,g))=0$, the images of $\theta_1$ and $\theta_2$ commute, so that we obtain the well-defined group homomorphism $\theta : \Gamma_1 \times \Gamma_2 \to G : \theta(g,h) = \theta_1(g) \theta_2(h)$. By construction, $q \circ \theta = \id$, so that $\Om$ must be a trivial $2$-cocycle.

Also the group homomorphism $\Theta_2$ is well-defined. To prove the faithfulness of $\Theta_2$, assume that $\Om \in H^2(\Gamma_1 \ast \Gamma_2,C)$ is such that $\Theta_2(\Om) = 0$. In the corresponding central extension, we find homomorphic lifts on $\Gamma_1$ and $\Gamma_2$, because the restrictions of $\Om$ to $\Gamma_1$ and $\Gamma_2$ are trivial. By freeness, they uniquely combine into a homomorphic lift on $\Gamma_1 \ast \Gamma_2$.

To prove the surjectivity of $\Theta_2$, choose $\om_i \in H^2(\Gamma_i,C)$ with corresponding central extensions $0 \to C \to G_i \to \Gamma_i \to e$. One turns the amalgamated free product $G = G_1 \ast_C G_2$ into a central extension whose corresponding $2$-cocycle $\Om$ satisfies $\Theta_2(\Om) = (\om_1,\om_2)$.
\end{proof}

The following more ad hoc lemma is an immediate consequence of Lemma \ref{lem.basic-cohomology} and will in particular be used to prove that our groups satisfy condition \ref{ncond.iso.1primeprime}.

\begin{lemma}\label{lem.cohomology-semidirect-order-2}
Let $\Gamma \actson^\al \Sigma$ be an action of a group $\Gamma$ on a countable abelian group $\Sigma$ in which every nontrivial element has order $2$. Write $G = \Sigma \rtimes \Gamma$. Let $C$ be an abelian group that is $2$-divisible. Then
$$\Psi : H^2(\Gamma,C) \to H^2(G,C) : (\Psi(\om))((a,g),(b,h)) = \om(g,h)$$
is a faithful group homomorphism and the following holds.
\begin{enuma}
\item\label{lem.cohomology-semidirect-order-2.1} If $C$ has no elements of order $2$, then $\Psi$ is a group isomorphism.
\item\label{lem.cohomology-semidirect-order-2.2} If $\Om \in H^2(G,\T)$ is a finite type $2$-cocycle, there exist finite index subgroups $\Sigma_0 < \Sigma$ and $\Gamma_0 < \Gamma$ such that $\al_g(\Sigma_0) = \Sigma_0$ for all $g \in \Gamma_0$ and a finite type $2$-cocycle $\om_0 \in H^2(\Gamma_0,\T)$ such that, writing $G_0 = \Sigma_0 \rtimes \Gamma_0$, we have $\Om|_{G_0} = \Psi(\om_0)$ in $H^2(G_0,\T)$.
\end{enuma}
\end{lemma}
\begin{proof}
We can write $\Sigma$ as the union of an increasing sequence of subgroups $\Sigma_n$ such that $\Sigma_1 = \{0\}$ and each $\Sigma_{n+1}/\Sigma_n$ has order $2$.

(a) Since every nontrivial element of $\Sigma$ has order $2$, while $C$ has no elements of order $2$, $\BHom(\Sigma,\Sigma,C) = 0$ and $\cZ(\Gamma \actson \Sigma,C) = 0$. By Lemmas \ref{lem.basic-cohomology.1} and \ref{lem.basic-cohomology.2}, the homomorphism $\Psi$ is a group isomorphism.

(b) Let $\Om \in H^2(G,\T)$ be a finite type $2$-cocycle. Choose a projective representation $\pi : G \to \cU(d)$ such that $\pi(g)\pi(h) = \Om(g,h) \pi(gh)$ for all $g,h \in G$. Define the bihomomorphism $\gamma : \Sigma \times \Sigma \to \T$ by $\gamma(a,b) = \Om(a,b) \overline{\Om(b,a)}$. Since $\pi(a) \pi(b) = \gamma(a,b) \pi(b) \pi(a)$ for all $a,b \in \Sigma$, conjugating by $\pi(g)$ for $g \in \Gamma$, we conclude that
\begin{equation}\label{eq.gamma-invariant}
\gamma(\al_g(a),\al_g(b)) = \gamma(a,b) \quad\text{for all $a,b \in \Sigma$, $g \in \Gamma$.}
\end{equation}
Define the subgroup $\Sigma_0 < \Sigma$ by $\Sigma_0 = \{a \in \Sigma \mid \forall b \in \Sigma: \gamma(a,b) = 1\}$. By \eqref{eq.gamma-invariant}, we have that $\al_g(\Sigma_0) = \Sigma_0$ for all $g \in \Gamma$. By \cite[Lemma 2.5]{DV24}, the subgroup $\Sigma_0 < \Sigma$ has finite index.

Since $\gamma(a,b) = 1$ for all $a,b \in \Sigma_0$ and since $\T$ is $2$-divisible, it follows from Lemma \ref{lem.basic-cohomology.2} that $\Om|_{\Sigma_0} = 1$ in $H^2(\Sigma_0,\T)$. We may thus assume that the restriction of $\pi$ to $\Sigma_0$ is a representation. As in the proof of \ref{lem.basic-cohomology.1}, define $\eta \in \cZ(\Gamma \actson \Sigma_0,\T)$ such that
\begin{equation}\label{eq.my-equivariance}
\pi(g) \pi(a) \pi(g)^* = \eta(g,a) \, \pi(\al_g(a)) \quad\text{for all $g \in \Gamma$, $a \in \Sigma_0$.}
\end{equation}
Since $\pi|_{\Sigma_0}$ is a finite dimensional unitary representation of the abelian group $\Sigma_0$, we find a finite subset $\cF \subset \widehat{\Sigma}_0$ and a direct sum decomposition $\C^d = \bigoplus_{\mu \in \cF} H_\mu$ with $H_\mu \subset \C^d$ nonzero subspaces such that $\pi(a) \xi = \mu(a) \xi$ for all $a \in \Sigma_0$ and $\xi \in H_\mu$.

Since $\eta \in \cZ(\Gamma \actson \Sigma_0,\T)$, the formula $(g^{-1} \cdot \mu)(a) = \eta(g,a) \, \mu(\al_g(a))$ defines an action of $\Gamma$ on $\widehat{\Sigma}_0$. By \eqref{eq.my-equivariance}, we get that $g \cdot \cF = \cF$ for all $g \in \Gamma$. Fix any $\mu \in \cF$ and define the finite index subgroup $\Gamma_0 < \Gamma$ by $\Gamma_0 = \{g \in \Gamma \mid g \cdot \mu = \mu\}$. By definition, $\eta(g,a) = \mu(a) \, \overline{\mu(\al_g(a))}$ for all $g \in \Gamma_0$ and $a \in \Sigma_0$. This means that $\eta \in \cB(\Gamma_0 \actson \Sigma_0,\T)$.

Write $G_0 = \Sigma_0 \rtimes \Gamma_0$. Then Lemma \ref{lem.basic-cohomology.1} provides a $2$-cocycle $\om_0 \in H^2(\Gamma_0,\T)$ such that $\Om|_{G_0} = \Psi(\om_0)$ in $H^2(G_0,\T)$. Since $\Om$ is of finite type, also $\om_0$ is of finite type.
\end{proof}

\begin{lemma}\label{lem.good-properties-ZP}
Fix a set of prime numbers $\cP$ and denote by $C = \Z[\cP^{-1}]$ the subring of $\Q$ generated by $\{p^{-1} \mid p \in \cP\}$.
\begin{enuma}
\item\label{lem.good-properties-ZP.1} If $N \in \N$, then $N \cdot C < C$ is a finite index subgroup whose index is the largest positive integer that divides $N$ and has no prime divisor in $\cP$.
\item\label{lem.good-properties-ZP.2} If $D < C$ is a finite index subgroup, $[C:D]$ has no prime divisor in $\cP$ and $D = [C:D] \cdot C$.
\item\label{lem.good-properties-ZP.3} If $D$ is a torsion free abelian group that is virtually isomorphic to $C$, then $D \cong C$.
\item\label{lem.good-properties-ZP.4} If $B$ is an abelian group that is $p$-divisible for every $p \in \cP$, then $H^2(C,B) = 0$.
\end{enuma}
\end{lemma}
\begin{proof}
(a) When $p \in \cP$, clearly $p \cdot C = C$. So it suffices to prove that $p \cdot C < C$ has index $p$ for every prime number $p \not\in \cP$.

Every element $c \in C$ can be uniquely written as an irreducible fraction $c = a/b$ with $a \in \Z$, $b \in \N$ and all prime divisors of $b$ belonging to $\cP$. It then follows immediately that the subsets $i + p \cdot C$, $i \in \{0,1,\ldots,p-1\}$ are disjoint. Also for any such $c = a/b$, we have that $p$ does not divide $b$, so that we can choose $\alpha \in \Z$ and $i \in \{0,1,\ldots,p-1\}$ such that $a = i b + \alpha p$. Then, $a/b \in i + p \cdot C$. So, the index of $p \cdot C$ equals $p$.

(b) We prove this statement by induction on the index $N = [C:D]$. If $N = 1$, there is nothing to prove. Next assume that $N > 1$ and that the statement holds for all subgroups with index strictly smaller than $N$. Write $N = p k$ with $p$ a prime number. Since $C/D$ is a finite abelian group of order $N$ and $p$ divides $N$, the group $\Z/p\Z$ is a quotient of $C/D$. We can thus choose a surjective group homomorphism $\vphi : C \to \Z / p\Z$ with $D < \Ker \vphi$.

If $p \in \cP$, we have for every $c \in C$ that $\vphi(c) = p \cdot \vphi(p^{-1}\cdot c) = 0$, contradicting that $\vphi$ is surjective. So, $p \not\in \cP$. Because $p \not\in \cP$ by (a), the subgroup $p \cdot C < C$ has index $p$. Also, $p \cdot C \subset \Ker \vphi$. Since also $\Ker \vphi < C$ has index $p$, it follows that $\Ker \vphi = p \cdot C$.

So, $D < p \cdot C$ is a subgroup of index $k$. We can then view $p^{-1} \cdot D$ as a subgroup of $C$ of index $k$. By the induction hypothesis, $k$ has no prime divisor in $\cP$ and $p^{-1} \cdot D = k \cdot C$. We have proven that $D = N \cdot C$ and that $N$ has no prime divisors in $\cP$.

(c) From (b), we already get that every finite index subgroup of $C$ is isomorphic with $C$.

We next prove that when $D$ is a torsion free abelian group that contains $C$ as a finite index subgroup, then $D \cong C$. Write $N = [D:C]$. Since $N \cdot a = 0$ for every $a \in D/C$, we get that $N \cdot D \subset C$. So, $\theta : D \to C : x \mapsto N \cdot x$ is a well-defined group homomorphism. Since $D$ is torsion free, $\theta$ is faithful. By construction, $N \cdot C$ is contained in the image of $\theta$. Since $N \cdot C < C$ has finite index, $D$ is isomorphic to a finite index subgroup of $C$. By the result above, it follows that $D \cong C$.

(d) Take $\Om \in H^2(C,B)$. Define the subgroups $C_1 = \{0\}$ and $C_2 = \Z$ of $C$. Choose a sequence $p_3,p_4,p_5,\ldots$ of prime numbers in $\cP$, with lots of repetitions, such that $C$ is the union of the subgroups $C_n = (p_3 \cdots p_n)^{-1} \Z$, $n \geq 3$. Since $B$ is $p$-divisible for every $p \in \cP$, by Lemma \ref{lem.basic-cohomology.2}, it suffices to prove that $\gamma := \Theta(\Om)$ is zero.

By \eqref{eq.hom-for-abelian}, $\gamma : C \times C \to B$ is a bihomomorphism satisfying $\gamma(a,a) = 0$ for all $a \in C$. For $n \geq 3$, denote by $c_n = (p_3 \cdots p_n)^{-1}$ the generator of the cyclic group $C_n$. Since $\gamma(c_n,c_n) = 0$ and $\gamma$ is a bihomomorphism, also $\gamma(k c_n, r c_n) = 0$ for all $k,r \in \Z$. So, $\gamma$ is $0$ on $C_n \times C_n$ for all $n$. This means that $\gamma = 0$.
\end{proof}

\section{\boldmath A concrete W$^*$-superrigidity theorem for central extensions}

In Section \ref{sec.proof-main}, we deduce Theorem \ref{thm.main-A} from the following more precise result.

\begin{theorem}\label{thm.more-general}
Let $\cP_1,\cP_2,\cQ$ be sets of prime numbers with $2 \in \cQ$ and $\cP_1 \cup \cP_2 \subset \cQ$. Consider the subrings $\Gamma_i = \Z[\cP_i^{-1}]$ and $C = \Z[\cQ^{-1}]$ of $\Q$. Define $\Gamma = \Gamma_1 \ast \Gamma_2$ and define the left-right wreath product group $G = (\Z/2\Z)^{(\Gamma)} \rtimes (\Gamma \times \Gamma)$.

Denote by $\Lambda = \Gamma_1 \times \Gamma_2$ the abelianization of $\Gamma$. Let $\zeta : G \to \Lambda \times \Lambda$ be the natural surjective homomorphism. The following homomorphism is an isomorphism:
\begin{multline}\label{eq.iso-of-cohom}
\BHom(\Lambda,\Lambda,C) \to H^2(G,C) : \gamma \mapsto \Om_\gamma \\
\text{with}\quad \Om_\gamma(g,h) = \gamma(b,a') \quad\text{when $\zeta(g)=(a,b)$ and $\zeta(h) = (a',b')$.}
\end{multline}
For every $\gamma \in \BHom(\Lambda,\Lambda,C)$, we denote by $G_\gamma$ the central extension of $G$ by $C$ given by $\Om_\gamma$.

\begin{enuma}
\item\label{thm.more-general.1} If $\gamma(\Lambda \times \Lambda)$ generates $C$, then $G_\gamma$ satisfies the isomorphism W$^*$-superrigidity property of Theorem \ref{thm.main-A.1}.

\item\label{thm.more-general.2} If $\gamma(\Lambda \times \Lambda)$ generates a finite index subgroup of $C$, then $G_\gamma$ satisfies the virtual isomorphism W$^*$-superrigidity properties of Theorem \ref{thm.main-A.2} and \ref{thm.main-A.3}.

\item\label{thm.more-general.3} We have that $G_\gamma \cong G_{\gamma'}$ if and only if there exists a $\vphi \in \Aut \Lambda$ that can be lifted to $\delta_0 \in \Aut \Gamma$, and nonzero integers $r,r'$ having all their prime divisors in $\cQ$ such that
    \begin{equation}\label{eq.equivalence-for-iso}
    r \, \gamma = r' \, \gamma' \circ (\vphi \times \vphi) \quad\text{or}\quad r \, \gamma = r' \, \gamma' \circ \si \circ (\vphi \times \vphi) \; ,
    \end{equation}
    where $\si(a,b) = (b,a)$ is the flip.

\item\label{thm.more-general.4} $G_\gamma$ is virtually isomorphic to $G_{\gamma'}$ iff $G_\gamma$ is commensurable to $G_{\gamma'}$ iff there exists a $\vphi \in \Aut \Lambda$ that can be lifted to $\delta_0 \in \Aut \Gamma$, and nonzero integers $r,r'$ such that \eqref{eq.equivalence-for-iso} holds.
\end{enuma}
\end{theorem}

Since the bihomomorphisms $\gamma : \Lambda \times \Lambda \to C$ are precisely given by $2 \times 2$ matrices over $C$ and since also the equivalence relations of (virtual) isomorphism appearing in \ref{thm.more-general.3} and \ref{thm.more-general.4} have a concrete matrix description, one can reformulate Theorem \ref{thm.more-general} in a more explicit form. This requires a case-by-case formulation depending on whether $\cP_i$ is empty or not, and whether $\cP_1$ is equal to $\cP_2$ or not. We postpone this to Section \ref{sec.proof-main}, where we also prove the statements appearing in Remark \ref{rem.main}.

In Proposition \ref{prop.new-no-go-result-3}, we will see that the assumptions $\cP_1 \cup \cP_2 \subset \cQ$ and $2 \in \cQ$ are essential for the validity of Theorem \ref{thm.more-general}.

As the first step in proving Theorem \ref{thm.more-general}, we show that the groups $G$ and $C$ appearing in the theorem satisfy all the conditions of our abstract W$^*$-superrigidity theorems.

\begin{lemma}\label{lem.about-all-our-conditions}
Let $\Gamma = \Gamma_1 \ast \Gamma_2$ be any free product of countable abelian groups with $|\Gamma_1| \geq 2$ and $|\Gamma_2| \geq 3$. Consider the left-right wreath product group $G = (\Z/2\Z)^{(\Gamma)} \rtimes (\Gamma \times \Gamma)$. Let $C$ be a torsion free abelian group.
\begin{enuma}
\item\label{lem.about-all-our-conditions.1} Conditions \ref{cond.iso.1}, \ref{cond.iso.1prime} and \ref{ncond.iso.3} hold.
\item\label{lem.about-all-our-conditions.2} If for all $i \in \{1,2\}$, $\Aut \Gamma_i$ is countable and $\Gamma_i$ has only countably many finite index subgroups, then conditions \ref{ncond.iso.2} and \ref{ncond.iso.2prime} hold.
\item\label{lem.about-all-our-conditions.3} $\Ext^1(G\ab,C) = 0$ if and only if $C$ is $2$-divisible and $\Ext^1(\Gamma_i,C) = 0$ for $i = 1,2$.
\item\label{lem.about-all-our-conditions.4} If for every finite index subgroup $S < \Gamma_i$, every finite type $2$-cocycle $\om \in H^2(S,\T)$ is trivial, then condition \ref{ncond.iso.1primeprime} holds.
\item\label{lem.about-all-our-conditions.5} If $C$ is $2$-divisible and if for every torsion free abelian group $D$ that is virtually isomorphic to $C$ and for every finite index subgroup $S < \Gamma_i$, we have that $\Ext^1(S,D) = 0$, then condition \ref{ncond.iso.4} holds.
\item\label{lem.about-all-our-conditions.6} If the groups $\Gamma_i$ are torsion free and if every finite index subgroup of $\Gamma_i$ has only finitely many index $2$ subgroups, then condition \ref{ncond.iso.5} holds.
\end{enuma}
\end{lemma}
\begin{proof}
As a preparation, we first describe the finite index subgroups of $G$. Write $\Sigma = (\Z/2\Z)^{(\Gamma)}$. Whenever $\Gamma_0 \lhd \Gamma$ is a finite index normal subgroup, define the finite index subgroup $\Sigma(\Gamma_0) < \Sigma$ as the kernel of the homomorphism
\begin{equation}\label{eq.hom-vphi}
\vphi_{\Gamma_0} : \Sigma \to (\Z/2\Z)^{(\Gamma/\Gamma_0)} : \bigl(\vphi_{\Gamma_0}(a)\bigr)_{g \Gamma_0} = \sum_{h \in \Gamma_0} a_{gh} \; .
\end{equation}
Define $G(\Gamma_0) := \Sigma(\Gamma_0) \rtimes (\Gamma_0 \times \Gamma_0)$. Then $G(\Gamma_0) < G$ has finite index. We claim that for every finite index subgroup $G_0 < G$, there exists a finite index normal subgroup $\Gamma_0 \lhd \Gamma$ such that $G(\Gamma_0) \subset G_0$.

Since $G_0 \cap (\Gamma \times e)$ has finite index in $\Gamma \times e$ and since $G_0 \cap (e \times \Gamma)$ has finite index in $e \times \Gamma$, we can choose a finite index normal subgroup $\Gamma_0 \lhd \Gamma$ such that $\Gamma_0 \times \Gamma_0 \subset G_0$. Denote by $\al : \Gamma \times \Gamma \actson \Sigma$ the left-right Bernoulli action. Then $\Sigma_0 := G_0 \cap \Sigma$ is a finite index subgroup of $\Sigma$ that is globally invariant under $\al_{(g,h)}$ for all $g,h \in \Gamma_0$. Replacing $\Gamma_0$ by a smaller finite index normal subgroup of $\Gamma$, we may assume that the automorphisms induced by $\al_{(g,h)}$ on the finite group $\Sigma/\Sigma_0$ are trivial. We prove that $\Sigma(\Gamma_0) < \Sigma_0$. Once this is proven, we get that $G(\Gamma_0) < G_0$ and the claim follows.

Choose an arbitrary $\om \in \widehat{\Sigma}$ that is equal to $1$ on $\Sigma_0$. It suffices to prove that $\om$ is equal to $1$ on $\Sigma(\Gamma_0)$. Since the action induced by $(\al_{(g,h)})_{(g,h) \in \Gamma_0 \times \Gamma_0}$ on $\Sigma/\Sigma_0$ is trivial and since $\om$ is equal to $1$ on $\Sigma_0$, we find that $\om \circ \al_{(g,h)} = \om$ for all $g,h \in \Gamma_0$. For every $k \in \Gamma$, denote by $\si_k \in \Sigma$ the element $1 \in \Z/2\Z$ in position $k$. Note that $\al_{(g,h)}(\si_k) = \si_{gkh^{-1}}$ for all $g,h,k \in \Gamma$. For every $k \in \Gamma$, we have that $\om(\si_k) \in \{\pm 1\}$. Define $J \subset \Gamma$ by $J = \{k \in \Gamma \mid \om(\si_k) = -1\}$. Since $\om \circ \al_{(g,h)} = \om$ for all $g,h \in \Gamma_0$, we get that $gJh = J$ for all $g,h \in \Gamma_0$. This means that $J = J_0 \Gamma_0$ for some subset $J_0 \subset \Gamma/\Gamma_0$. Denote by $\om_0$ the unique character on $(\Z/2\Z)^{(\Gamma/\Gamma_0)}$ that maps the generator in position $k \Gamma_0$ to $-1$ for all $k \in J_0$. By construction, $\om = \om_0 \circ \vphi_{\Gamma_0}$. So, $\om = 1$ on $\Sigma(\Gamma_0)$ and the claim is proven.

(a) Because every free product $\Gamma = \Gamma_1 \ast \Gamma_2$ of amenable groups $\Gamma_i$ with $|\Gamma_1| \geq 2$ and $|\Gamma_2| \geq 3$ belongs to class $\cC$, condition \ref{cond.iso.1} follows from \cite[Theorem 6.15]{DV24} and \ref{cond.iso.1prime} follows from \cite[Theorem A]{DV24}, while condition \ref{ncond.iso.3} follows from Lemma \ref{lem.left-right-wreath-product-satisfies-4}.

(b) Since both $\Aut \Gamma_i$ are countable, it follows from Lemma \ref{lem.countable-aut} that $\Aut \Gamma$ is countable. By \cite[Proposition 6.10]{DV24}, every isomorphism between finite index subgroups of $G$ is the restriction of an automorphism of $G$. To conclude the proof of (b), it thus suffices to prove that $G$ has only countably many finite index subgroups. By the introductory claim, it suffices to prove that $\Gamma$ has only countably many finite index subgroups. By the Kurosh subgroup theorem (see Section \ref{sec.kurosh}), for every finite index subgroup $\Gamma_0 < \Gamma$, there exist finitely many elements $a_i, b_j, c_k \in \Gamma$ and finite index subgroups $A_i < \Gamma_1$, $B_j < \Gamma_2$ such that $\Gamma_0$ is generated by $a_i A_i a_i^{-1}$, $b_j B_j b_j^{-1}$ and the elements $c_k$. Since we assumed that the groups $\Gamma_i$ have only countably many finite index subgroups, the proof of (b) is complete.

(c) Denote by $\zeta : \Gamma \to \Gamma_1 \times \Gamma_2$ the canonical abelianization homomorphism. Also consider the $(\Gamma \times \Gamma)$-invariant homomorphism $\vphi_{\Gamma} : \Sigma \to \Z/2\Z$ given by \eqref{eq.hom-vphi}. Then $\vphi_\Gamma$ and $\zeta \times \zeta$ define the abelianization $G\ab \cong \Z/2\Z \times \Gamma_1 \times \Gamma_2 \times \Gamma_1 \times \Gamma_2$, so that
$$\Ext^1(G\ab,C) = C/2\cdot C \oplus \Ext^1(\Gamma_1,C) \oplus \Ext^1(\Gamma_2,C) \oplus \Ext^1(\Gamma_1,C) \oplus \Ext^1(\Gamma_2,C) \; ,$$
from which (c) follows.

(d) Take a finite index subgroup $G_0 < G$ and a finite type $2$-cocycle $\Phi \in H^2(G_0,\T)$. We have to find a finite index subgroup $G_1 < G_0$ such that $\Phi|_{G_1} = 1$ in $H^2(G_1,\T)$. We may assume that $G_0$ is of the form $\Sigma_0 \rtimes (\Gamma_0 \times \Gamma_0)$, where $\Gamma_0 < \Gamma$ is a finite index subgroup and $\Sigma_0 < \Sigma$ is a finite index subgroup that is globally invariant under $\Gamma_0 \times \Gamma_0$. By Lemma \ref{lem.cohomology-semidirect-order-2.2}, after making $\Sigma_0$ and $\Gamma_0$ smaller, we may assume that $\Phi$ comes from a finite type $2$-cocycle $\om \in H^2(\Gamma_0 \times \Gamma_0,\T)$.

As in the proof of (b), by the Kurosh subgroup theorem, the group $\Gamma_0$ is a free product of copies of $\Z$ and finite index subgroups of the groups $\Gamma_i$. By Lemma \ref{lem.basic-cohomology.3} and our assumption that all finite type $2$-cocycles on finite index subgroups of $\Gamma_i$ are trivial, it follows that $\om$ is induced by an element $\gamma \in \BHom(\Gamma_0,\Gamma_0,\T)$. Since $\om$ is of finite type, it then follows from \cite[Lemma 2.5]{DV24} that there exists a finite index subgroup $\Gamma'_0 < \Gamma_0$ such that $\om|_{\Gamma'_0 \times \Gamma'_0}$ is trivial. With $G_1 = \Sigma_0 \rtimes (\Gamma'_0 \times \Gamma'_0)$, we get that $\Phi|_{G_1} = 1$.

(e) First note that given a torsion free abelian group $C$ and a finite index subgroup $C_0$, we have that $C_0$ is $2$-divisible if and only if $C$ is $2$-divisible. To prove this statement, denote $N = [C:C_0]$ and write $N = 2^k(2n+1)$ where $k,n \geq 0$ are integers. First assume that $C$ is $2$-divisible and choose $c \in C_0$. Since $C$ is $2$-divisible, we can take $d \in C$ such that $2^{k+1} \cdot d = c$. It follows that $(2n+1) \cdot c = 2N \cdot d$. Since $[C:C_0] = N$, the element $f := N \cdot d - n \cdot c$ belongs to $C_0$. By construction, $2 \cdot f = c$. Conversely, assume that $C_0$ is $2$-divisible and choose $c \in C$. Since $N \cdot c \in C_0$, take $d \in C_0$ such that $N \cdot c = 2^{k+1} \cdot d$. Since $C$ is torsion free, it follows that $(2n+1) \cdot c = 2 \cdot d$ and thus, $c = 2 \cdot (d-n\cdot c)$.

Fix a torsion free abelian group $D$ that is virtually isomorphic to $C$. Since $C$ and $D$ are torsion free, they thus have isomorphic finite index subgroups. By the previous paragraph, $D$ is $2$-divisible. Fix a finite index subgroup $G_0 < G$ and $\Psi \in \Ext^1(G_{0,\text{\rm ab}},D)$. To prove that condition \ref{ncond.iso.4} holds, we may assume that $G_0 = G(\Gamma_0)$ for some finite index normal subgroup $\Gamma_0 \lhd \Gamma$. Since $G_0 = \Sigma(\Gamma_0) \rtimes (\Gamma_0 \times \Gamma_0)$, we find that $G_{0,\text{\rm ab}}$ is the direct sum of some quotient $\Sigma_1$ of $\Sigma(\Gamma_0)$ with two copies of $\Gamma_{0,\text{\rm ab}}$. Every nontrivial element of $\Sigma_1$ has order $2$. Since $D$ is $2$-divisible, it follows from Lemma \ref{lem.basic-cohomology.2} that $\Ext^1(\Sigma_1,D) = 0$. As in the proof of (b), by the Kurosh subgroup theorem, $\Gamma_{0,\text{\rm ab}}$ is the direct sum of copies of $\Z$ and finite index subgroups of the groups $\Gamma_i$. So, also $\Ext^1(\Gamma_{0,\text{\rm ab}},D) = 0$. We have thus proven that $\Ext^1(G_{0,\text{\rm ab}},D)=0$, so that $\Psi = 0$.

(f) Fix a finite abelian group $E$, a finite index subgroup $G_0 < G$ and $\Psi \in \Ext^1(G_{0,\text{\rm ab}},E)$. As in the proof of (e), we may assume that $G_0 = G(\Gamma_0)$, where $\Gamma_0 \lhd \Gamma$ is a finite index normal subgroup.

Denote by $(\al_{(g,h)})_{(g,h) \in \Gamma_0 \times \Gamma_0}$ the action of $\Gamma_0 \times \Gamma_0$ on $\Sigma_0 := \Sigma(\Gamma_0)$. Denote by $K \subset \widehat{\Sigma_0}$ the closed subgroup of characters $\om$ satisfying $\om \circ \al_{(g,h)} = \om$ for all $(g,h) \in \Gamma_0 \times \Gamma_0$. Define $\Sigma_1$ as the quotient of $\Sigma_0$ by the closed subgroup of $\si \in \Sigma_0$ satisfying $\om(\si) = 1$ for all $\om \in K$. By definition, $G_{0,\text{\rm ab}} = \Sigma_1 \times \Gamma_{0,\text{\rm ab}} \times \Gamma_{0,\text{\rm ab}}$.

Since the groups $\Gamma_i$ are torsion free, it follows from the discussion in the proof of (e) that $\Gamma_{0,\text{\rm ab}}$ is torsion free. Since $\Ext^1(T,E) = 0$ for every torsion free abelian $T$ and finite abelian $E$ (see e.g.\ \cite[Section 52, statement (F)]{Fuc70}), we get that $\Ext^1(\Gamma_{0,\text{\rm ab}} \times \Gamma_{0,\text{\rm ab}},E) = 0$. In particular, the restriction of $\Psi$ to $A := \{0\} \times \Gamma_{0,\text{\rm ab}} \times \Gamma_{0,\text{\rm ab}}$ is equal to $0$ in $\Ext^1(A,E)$. To conclude the proof of (f), it suffices to show that $\Sigma_1$ is finite, since we can then define $G_1 < G_0$ as the preimage of $A$ under the abelianization homomorphism and conclude that $\Psi$ becomes trivial in $\Ext^1(G_{1,\text{\rm ab}},E)$.

Since the groups $\Gamma_i$ admit only finitely many index $2$ subgroups and since $\Gamma_0$ can be written as a free product of finitely many copies of $\Z$ and groups that are isomorphic to finite index subgroups of the $\Gamma_i$, we get that there are only finitely many homomorphisms $\Gamma_0 \to \Z/2\Z$. To prove that $\Sigma_1$ is finite, it then suffices to prove the following statement: consider a countable group $\Lambda$ that has only finitely many index $2$ subgroups; define $\Sigma_0$ as the kernel of
$$\vphi : (\Z/2\Z)^{(\Lambda)} \to \Z/2\Z : a \mapsto \sum_{g \in \Lambda} a_g \; ;$$
then the group of $(\Lambda \times \Lambda)$-invariant characters on $\Sigma_0$ is finite.

Let $\om_0 \in \widehat{\Sigma_0}$ be $(\Lambda \times \Lambda)$-invariant. Choose a character $\om$ on $\Sigma = (\Z/2\Z)^{(\Lambda)}$ that extends $\om_0$. Since every nontrivial element of $\Sigma$ has order $2$, we view $\om$ as a homomorphism $\om : \Sigma \to \Z/2\Z$. Since $\Sigma_0 < \Sigma$ has index $2$, either $\om$ is $(\Lambda \times \Lambda)$-invariant or the orbit of $\om$ under the $(\Lambda \times \Lambda)$-action has two elements.

As above, for every $k \in \Lambda$, we consider $\si_k \in \Sigma$ given by the element $1 \in \Z/2\Z$ in position $k$. Define $J \subset \Lambda$ by $J = \{k \in \Lambda \mid \om(\si_k) = 1\}$. In the first case where $\om$ is $(\Lambda \times \Lambda)$-invariant, we find that $J = \Lambda J \Lambda$, so that $J=\emptyset$ or $J = \Lambda$. This means that $\om = 0$ or $\om = \vphi$, so that $\om_0 = 0$.

In the second case where $\om$ is invariant under an index $2$-subgroup $K < \Lambda \times \Lambda$, we find subgroups $\Lambda_0 < \Lambda$ and $\Lambda_1 < \Lambda$ of index at most $2$ such that $\Lambda_0 \times \Lambda_1 < K$. Then, $J = \Lambda_0 J \Lambda_1$. Since we are in the second case, $J \neq \emptyset$ and $J \neq \Lambda$. So both $\Lambda_i$ must be proper, index $2$ subgroups of $\Lambda$. It then also follows that $\Lambda_0 = \Lambda_1$ and $J = g_0\Lambda_0$ for an element $g_0 \in \Lambda$. This means that $\om(a) = \sum_{g \in \Lambda_0} a_{g_0 g}$ for all $a \in \Sigma$. Since $\Lambda$ has only finitely many index $2$ subgroups, we conclude that are only finitely many $(\Lambda \times \Lambda)$-invariant characters on $\Sigma_0$.
\end{proof}

Below, we use a few times the following observation, which already appeared in several versions above. If $G$ is a group containing a subgroup of the form $\Gamma_0 \times \Gamma_0$ and if $D$ is any abelian group, then
\begin{equation}\label{eq.easy-reduction-to-bihom}
\Xi_0 : H^2(G,D) \to \BHom(\Gamma_0,\Gamma_0,D) : (\Xi_0(\Om))(g,h) = \Om((e,g),(h,e)) - \Om((h,e),(e,g))
\end{equation}
is a well-defined group homomorphism. If $\Gamma_0 = \Gamma$, $D = C$, and we identify $\BHom(\Gamma,\Gamma,C) = \BHom(\Gamma\ab,\Gamma\ab,C)$, then the homomorphism $\Xi_0$ is the inverse of the isomorphism \eqref{eq.iso-of-cohom}.

\begin{proof}[{Proof of Theorem \ref{thm.more-general}}]
By Lemma \ref{lem.good-properties-ZP.2}, $\Gamma_i$ has only countably many finite index subgroups, at most one index $2$ subgroup and every finite index subgroup of $\Gamma_i$ is isomorphic with $\Gamma_i$. Since every automorphism of $\Gamma_i$ is given by multiplication with $a/b$ where $a,b \in \Z \setminus \{0\}$ have all their prime divisors in $\cP_i$, the automorphism groups $\Aut \Gamma_i$ are countable. By Lemma \ref{lem.good-properties-ZP.3}, every torsion free abelian group that is virtually isomorphic to $C$ is actually isomorphic to $C$. By Lemma \ref{lem.good-properties-ZP.4}, $\Ext^1(\Gamma_i,C) = 0$ and $H^2(\Gamma_i,\T) = 1$. So by Lemma \ref{lem.about-all-our-conditions}, all conditions \ref{cond.iso.1}, \ref{cond.iso.1prime}, \ref{ncond.iso.1primeprime}, \ref{ncond.iso.2}, \ref{ncond.iso.2prime}, \ref{ncond.iso.3}, \ref{ncond.iso.4} and \ref{ncond.iso.5} are satisfied and $\Ext^1(G\ab,C) = 0$.

By Lemma \ref{lem.good-properties-ZP.4}, $H^2(\Gamma_i,C) = 0$ for $i=1,2$. By Lemma \ref{lem.basic-cohomology.3}, also $H^2(\Gamma,C) = 0$. Combining Lemmas \ref{lem.cohomology-semidirect-order-2.1} and \ref{lem.basic-cohomology.3}, we find an isomorphism $H^2(G,C) \cong \BHom(\Gamma,\Gamma,C)$. Since $C$ is abelian, also $\BHom(\Gamma,\Gamma,C) \cong \BHom(\Lambda,\Lambda,C)$. All these isomorphisms are explicit and precisely combine into \eqref{eq.iso-of-cohom} being an isomorphism.

(a) Assume that $\gamma(\Lambda \times \Lambda)$ generates $C$. By Theorem \ref{thm.inherit-iso-superrigidity}, it only remains to prove that $\Theta(\Om_\gamma)$ is faithful. If $\mu \in \Chat$ belongs to the kernel of $\Theta(\Om_\gamma)$, we have that $\mu \circ \Om_\gamma = 1$ in $H^2(G,\T)$. Applying the map $\Xi_0$ in \eqref{eq.easy-reduction-to-bihom}, we conclude that $\mu \circ \gamma = 1$, so that $\mu = 1$.

(b) Assume that $\gamma(\Lambda \times \Lambda)$ generates a finite index subgroup of $C$. By Theorem \ref{thm.inherit-virtual-iso-superrigidity}, it only remains to prove that $\Theta(\Om_\gamma|_{G_0})$ has finite kernel for every finite index subgroup $G_0 < G$. Using the first paragraphs of the proof of Lemma \ref{lem.about-all-our-conditions}, we may assume that $G_0 = G(\Gamma_0)$, where $\Gamma_0 \lhd \Gamma$ is a finite index normal subgroup. Denote by $\zeta_0 : \Gamma \to \Lambda$ the abelianization map. When $\mu \in \Chat$ belongs to the kernel of $\Theta(\Om_\gamma|_{G_0})$, we apply the map $\Xi_0$ in \eqref{eq.easy-reduction-to-bihom} and conclude that $\mu = 1$ on $\gamma(\zeta_0(\Gamma_0) \times \zeta_0(\Gamma_0))$. It thus suffices to prove that $\gamma(\zeta_0(\Gamma_0) \times \zeta_0(\Gamma_0))$ generates a finite index subgroup of $C$.

By assumption, $\gamma(\Lambda \times \Lambda)$ generates a finite index subgroup of $C$, which by Lemma \ref{lem.good-properties-ZP.2} is of the form $N \cdot C$ for some integer $N \geq 1$. We write $\Lambda = \Z[\cP_1^{-1}] \times \Z[\cP_2^{-1}]$ additively. Since $\zeta_0(\Gamma_0) < \Lambda$ has finite index, Lemma \ref{lem.good-properties-ZP.2} also provides an integer $N_0 \geq 1$ such that $N_0 \cdot \Lambda < \zeta_0(\Gamma_0)$. It then follows that the subgroup of $C$ generated by $\gamma(\zeta_0(\Gamma_0) \times \zeta_0(\Gamma_0))$ contains $N_0^2 N \cdot C$ and therefore has finite index in $C$ by Lemma \ref{lem.good-properties-ZP.1}.

(c) First assume that $\deltatil : G_\gamma \to G_{\gamma'}$ is an isomorphism of groups. Since $\cZ(G_\gamma) = C$ and $\cZ(G_{\gamma'}) = C$, we have $\deltatil(C) = C$. We denote by $\rho \in \Aut C$ the restriction of $\deltatil$ to $C$. Since $C = \Z[\cQ^{-1}]$, we get that $\rho$ is given by multiplication with $r/r'$ where $r,r'$ are nonzero integers having all their prime divisors in $\cQ$.

We denote by $\delta \in \Aut G$, the automorphism induced by $\deltatil$ on the quotient $G = G_{\gamma}/C = G_{\gamma'}/C$. By \cite[Proposition 6.10]{DV24}, after composing $\deltatil$ with an inner automorphism, we find $\delta_0 \in \Aut \Gamma$ such that either $\delta(g,h) = (\delta_0(g),\delta_0(h))$ for all $g,h \in \Gamma$, or $\delta(g,h) = (\delta_0(h),\delta_0(g))$ for all $g,h \in \Gamma$. We denote by $\vphi \in \Aut \Lambda$ the automorphism induced by $\delta_0$ on the abelianization $\Lambda$ of $\Gamma$.

Since $\deltatil$ is an isomorphism of groups, we have that $\rho \circ \Om_{\gamma} = \Om_{\gamma'} \circ \delta$ in $H^2(G,C)$. Applying the map $\Xi_0$ of \eqref{eq.easy-reduction-to-bihom}, we get that \eqref{eq.equivalence-for-iso} holds.

To prove the converse, let $\delta_0 \in \Aut \Gamma$ with associated $\vphi \in \Aut \Lambda$ and let $r,r'$ be nonzero integers having all their prime divisors in $\cQ$ such that \eqref{eq.equivalence-for-iso} holds. Depending on which of both formulas hold, we define $\rho \in \Aut C$ by multiplication with $r/r'$, resp.\ $-r/r'$, and we define $\delta \in \Aut G$ by $\delta(\pi_e(a)) = \pi_e(a)$ for all $a \in \Z/2\Z$ and $\delta(g,h) = (\delta_0(g),\delta_0(h))$, resp.\ $\delta(g,h) = (\delta_0(h),\delta_0(g))$ for all $g,h \in \Gamma$. By definition, $\rho \circ \Om_\gamma = \Om_{\gamma'} \circ \delta$, so that $G_\gamma \cong G_{\gamma'}$.

(d) Since $G_{\gamma,\text{\rm fc}} = C$ and $C$ is torsion free, finite index subgroups of $G_\gamma$ have no nontrivial finite normal subgroups. The same holds for $G_{\gamma'}$. So $G_\gamma$ is virtually isomorphic to $G_{\gamma'}$ iff $G_\gamma$ is commensurable to $G_{\gamma'}$.

Next assume that $G_\gamma$ is commensurable to $G_{\gamma'}$. Take finite index subgroups $H < G_\gamma$, $H' < G_{\gamma'}$ and a group isomorphism $\deltatil : H \to H'$. Since $G_{\gamma,\text{\rm fc}} = C$ and $G_{\gamma',\text{\rm fc}} = C$, we get that $\deltatil(H \cap C) = H' \cap C$. We denote by $\rho$ the restriction of $\deltatil$ to $H \cap C$. By Lemma \ref{lem.good-properties-ZP.2}, $H \cap C = N \cdot C$ for some integer $N \geq 1$, so that $\rho$ must be given by multiplication with a nonzero rational number $r/r'$.

We denote by $\delta$ the isomorphism induced by $\deltatil$ between $G_0 := H / (H \cap C)$ and $G'_0 := H'/(H' \cap C)$. Note that $G_0$ and $G'_0$ are finite index subgroups of $G$. By \cite[Proposition 6.10]{DV24}, after composing $\deltatil$ with an inner automorphism, we find $\delta_0 \in \Aut \Gamma$ such that either $\delta(g,h) = (\delta_0(g),\delta_0(h))$ for all $(g,h) \in G_0 \cap (\Gamma \times \Gamma)$, or $\delta(g,h) = (\delta_0(h),\delta_0(g))$ for all $(g,h) \in G_0 \cap (\Gamma \times \Gamma)$. Passing to the abelianization, $\delta_0 \in \Aut \Gamma$ induces an automorphism $\vphi \in \Aut \Lambda$.

We have the central extensions $0 \to H \cap C \to H \to G_0 \to e$ and $0 \to H' \cap C \to H' \to G'_0 \to e$. We denote by $\Psi \in H^2(G_0,H \cap C)$ and $\Psi' \in H^2(G'_0,H' \cap C)$ the associated $2$-cocycles. Since $\deltatil : H \to H'$ is a group isomorphism that restricts to the isomorphism $\rho : H \cap C \to H' \cap C$, we find that $\rho \circ \Psi = \Psi' \circ \delta$ in $H^2(G_0,H' \cap C)$.

Since $\rho$ is given by multiplication with $r/r'$, we get in particular that $r \Psi = r' \Psi' \cap \delta$ in $H^2(G_0,C)$. By definition, we also have that $\Psi = \Om_\gamma$ in $H^2(G_0,C)$ and $\Psi' = \Om_{\gamma'}$ in $H^2(G'_0,C)$. We thus conclude that
\begin{equation}\label{eq.better-cohom}
r \Om_\gamma = r' \Om_{\gamma'} \circ \delta \quad\text{in $H^2(G_0,C)$.}
\end{equation}
Since $G_0 < G$ has finite index, we find a finite index subgroup $\Gamma_0 < \Gamma$ such that $\Gamma_0 \times \Gamma_0 < G_0$. We denote by $\Lambda_0 < \Lambda$ the image of $\Gamma_0$. Applying \eqref{eq.easy-reduction-to-bihom} to \eqref{eq.better-cohom}, it follows that \eqref{eq.equivalence-for-iso} holds on $\Lambda_0 \times \Lambda_0$. By Lemma \ref{lem.good-properties-ZP.2}, we find an integer $N_0 \geq 1$ such that $N_0 \cdot \Lambda < \Lambda_0$. It then follows that \eqref{eq.equivalence-for-iso} holds on $\Lambda \times \Lambda$.

Conversely, if $\vphi \in \Aut \Lambda$ can be lifted to $\delta_0 \in \Aut \Gamma$, and if $r,r'$ are nonzero integers such that \eqref{eq.equivalence-for-iso} holds, we proceed in the same way as in the proof of (c) and obtain that $G_\gamma$ is commensurable to $G_{\gamma'}$.
\end{proof}

\section{\boldmath Lack of W$^*$-superrigidity}\label{sec.lack-of-superrigidity}

In this section, we illustrate how all the cohomological conditions in Theorems \ref{thm.inherit-iso-superrigidity} and \ref{thm.inherit-virtual-iso-superrigidity} are essential. We start by illustrating the role of the kernel of $\Theta(\Om) : \Chat \to H^2(G,\T)$.

\begin{proposition}\label{prop.new-no-go-result-1}
Let $G$ be an icc group, $C$ a torsion free abelian group and $\Gtil$ a central extension of $G$ by $C$ with associated $2$-cocycle $\Om \in H^2(G,C)$.

\begin{enuma}
\item\label{prop.new-no-go-result-1.1} If $\Theta(\Om)$ is not faithful, there is a countable group $\Lambdatil$ such that $L(\Lambdatil) \cong L(\Gtil)$, but $\Lambdatil \not\cong \Gtil$.
More precisely, with the notation of \ref{thm.inherit-iso-superrigidity}, there is a group $\Lambdatil$ in $\cH_{\sim^1} \setminus \cH_\approx$.

\item\label{prop.new-no-go-result-1.2} If $\Theta(\Om|_{G_0})$ has infinite kernel for a finite index $G_0 < G$, there is a countable group $\Lambdatil$ and a faithful bifinite $L(\Lambdatil)$-$L(\Gtil)$-bimodule such that $\Lambdatil$ is not virtually isomorphic to $\Gtil$.

    More precisely, with the notation of \ref{thm.inherit-virtual-iso-superrigidity}, there exists a group in $\cF_{\sim}$ that is not virtually isomorphic to a group in $\cF_{\approx\ve}$.
\end{enuma}
\end{proposition}

\begin{proof}
We prove (a) and (b) in parallel. For the proof of (a), we define $G_0 := G$. For the proof of (b), we choose a finite index subgroup $G_0 < G$ such that $\Theta(\Om|_{G_0})$ has infinite kernel. We define the subgroup $C_0 < C$ such that $\mu \in \Chat$ belongs to the kernel of $\Theta(\Om|_{G_0})$ if and only if $\mu$ is equal to $1$ on $C_0$. In (a), we have that $C_0$ is a proper subgroup of $C$, while in (b), we have that $C_0 < C$ has infinite index.
Denote by $\al : \widehat{C_0} \to H^2(G_0,\T)$ the unique group homomorphism such that $\Theta(\Om|_{G_0}) = \al \circ \res$, where $\res : \Chat \to \widehat{C_0}$ is the restriction homomorphism. Note that $\al$ is continuous.

By the surjectivity of $\Theta$ in the universal coefficient theorem \eqref{eq.univ-coeff}, we can take $\Om_0 \in H^2(G_0,C_0)$ such that $\Theta(\Om_0) = \al$. Choose an abelian torsion group $T$ such that $|C/C_0| = |T|$, which might be $+\infty$. Define $\Om_1 \in H^2(G_0,C_0 \times T)$ by composing $\Om_0$ with the embedding $C_0 \to C_0 \times T$. Denote by $\Lambdatil$ the central extension of $G_0$ by $C_0 \times T$ defined by $\Om_1$.

By construction, $\Om|_{G_0} \sim^1 \Om_1$ in the sense of Definition \ref{def.equiv-2-cocycle}. Define $\Gtil_0 < \Gtil$ as the preimage of $G_0 < G$. By Theorem \ref{thm.inherit-iso-superrigidity.1}, we have that $L(\Lambdatil) \cong L(\Gtil_0)$. By construction, the torsion group $T$ sits in the center of $\Lambdatil$.

In (a), we have $G_0 = G$ and thus $\Gtil_0 = \Gtil$. We find that $\Lambdatil \in \cH_{\sim^1}$. Since the center of $\Lambdatil$ has a nontrivial torsion subgroup, while $G$ is icc and $C$ torsion free, $\Lambdatil$ is not isomorphic to a central extension of $G$ by $C$, so that $\Lambdatil$ does not belong to $\cH_\approx$. In particular, $\Lambdatil \not\cong \Gtil$.

To prove (b), we already have that $\Lambdatil$ belongs to $\cF_\sim$. By Theorem \ref{thm.inherit-virtual-iso-superrigidity.1}, there exists a faithful bifinite $L(\Lambdatil)$-$L(\Gtil)$-bimodule. To prove that $\Lambdatil$ is not virtually isomorphic to a group in $\cF_{\approx\ve}$, it suffices to prove the following statement: if $\Gamma$ is a countable group that is virtually isomorphic to a central extension $\Gamma'$ of a finite index subgroup $G_1 < G$ by a countable abelian group $D$ that is commensurable to $C$, then the torsion subgroup $T$ of $\cZ(\Gamma)$ is finite.

Since $D$ is commensurable to $C$ and $C$ is torsion free, the group $D\tor$ is finite. Replacing $\Gamma'$ by $\Gamma'/D\tor$, we may assume that $D$ is torsion free. Since $G$ is icc, no finite index subgroup of $\Gamma'$ has a nontrivial finite normal subgroup. Since $\Gamma$ is virtually isomorphic to $\Gamma'$, we thus find a finite index subgroup $\Gamma_0 < \Gamma$ and a group homomorphism $\theta : \Gamma_0 \to \Gamma'$ such that $\Ker \theta$ is finite and $\theta(\Gamma_0) < \Gamma'$ has finite index. Since $G$ is icc, $\theta(\cZ(\Gamma_0)) \subset D$. Since $D$ is torsion free, it follows that $\theta(T \cap \Gamma_0) = \{0\}$. Since $\Gamma_0 < \Gamma$ has finite index and $\Ker \theta$ is finite, we conclude that $T$ is finite.
\end{proof}

Our next no-go result shows that extensions by $C$ of a generalized wreath product group with base $\Z/n\Z$ will never be W$^*$-superrigid if $n \cdot C \neq C$. We use the equivalence relation $\approx$ from Definition \ref{def.equiv-2-cocycle}.

\begin{proposition}\label{prop.new-no-go-result-2}
Let $\Gamma$ be a countable group and $\Gamma \actson I$ a faithful transitive action such that $(\Stab_\Gamma i) \cdot j$ is infinite for all $i \neq j$ in $I$. Let $n \in \N$ and consider the generalized wreath product group $G = (\Z/n\Z)^{(I)} \rtimes \Gamma$.

If $C$ is an abelian group with $n \cdot C \neq C$, then no central extension $\Gtil$ of $G$ by $C$ is W$^*$-superrigid.

More precisely, for every $\Om \in H^2(G,C)$, there exists an $\Om' \in H^2(G,C)$ such that $\Om \approx \Om'$ and the corresponding central extensions $\Gtil$ and $\Lambdatil$ satisfy $L(\Gtil) \cong L(\Lambdatil)$ and $\Gtil \not\cong \Lambdatil$.
\end{proposition}

Before proving Proposition \ref{prop.new-no-go-result-2}, we need a lemma describing the automorphism group of a generalized wreath product group. Although starting with \cite{Hou62}, there is a large body of literature computing the automorphism group of such a generalized (i.e.\ permutational) wreath product group $\Sigma_0 \wr_I \Gamma$, we need a version of this that, as far as we know, is not available in the literature, but can be proven rather easily. The reason why we prove the following lemma in a very high generality, in particular without making any assumptions on the acting group $\Gamma$, is to obtain the no-go Proposition \ref{prop.new-no-go-result-2} in the largest possible generality.

\begin{lemma}\label{lem.aut-generalized-wreath}
Let $\Gamma$ be a countable group and $\Gamma \actson I$ a faithful transitive action such that $(\Stab_\Gamma i) \cdot j$ is infinite for all $i \neq j$ in $I$. Let $\Sigma_0$ be a nontrivial abelian group, write $\Sigma = \Sigma_0^{(I)}$ and consider the generalized wreath product group $G = \Sigma \rtimes \Gamma$.

For every automorphism $\delta \in \Aut G$ and every $i_0 \in I$, there exist $g_0 \in G$ and $\delta_0 \in \Aut \Sigma_0$ such that $(\Ad g_0) \circ \delta \circ \pi_{i_0} = \pi_{i_0} \circ \delta_0$, where $\pi_{i_0} : \Sigma_0 \to \Sigma$ is the embedding in position $i_0$.

In particular, every automorphism of $G$ globally preserves $\Sigma$.
\end{lemma}
\begin{proof}
Fix an automorphism $\delta \in \Aut G$.

{\bf Step 1~:} we prove that $\delta(\Sigma) = \Sigma$. Since also $\delta^{-1}$ is an automorphism of $G$, it suffices to prove that $\delta(\Sigma) \subset \Sigma$. Assume the contrary and take $a \in \Sigma$ and $g \in \Gamma \setminus \{e\}$ such that $a g \in \delta(\Sigma)$. Since the action of $\Gamma$ on $I$ is faithful, we can take $i_0 \in I$ with $g \cdot i_0 \neq i_0$. Since $\delta(\Sigma)$ is a normal subgroup of $G$, we get that $\pi_{i_0}(b) a g \pi_{i_0}(b)^{-1} \in \delta(\Sigma)$ for all $b \in \Sigma_0$. Write $j_0 = g \cdot i_0$. Since $a g \in \delta(\Sigma)$ and $\Sigma$ is abelian, also
$$\pi_{i_0}(b) \pi_{j_0}(b^{-1}) = \pi_{i_0}(b) a g \pi_{i_0}(b)^{-1} (ag)^{-1} \in \delta(\Sigma) \quad\text{for all $b \in \Sigma_0$.}$$
Define
$$S = \{(i,j) \in I \times I \mid i \neq j \;\; , \;\; \forall b \in \Sigma_0 : \pi_i(b) \pi_j(b^{-1}) \in \delta(\Sigma) \} \; .$$
Since $\delta(\Sigma)$ is a normal subgroup of $G$, it follows that $(h \cdot i, h \cdot j) \in S$ for all $(i,j) \in S$ and $h \in \Gamma$.

Denote by $\zeta : G \to \Gamma$ the natural quotient homomorphism. We claim that for every $h \in \zeta(\delta(\Sigma))$ and $(i,j) \in S$, we have $\{h \cdot i,h \cdot j\} = \{i,j\}$. Indeed, taking $b \in \Sigma_0 \setminus \{e\}$, we find that $\pi_i(b) \pi_j(b^{-1}) \in \delta(\Sigma)$. Since $h \in \zeta(\delta(\Sigma))$, we can take $c \in \Sigma$ such that $c h \in \delta(\Sigma)$. Since $\delta(\Sigma)$ is abelian, $ch$ commutes with $\pi_i(b) \pi_j(b^{-1})$. Since also $\Sigma$ is abelian, it follows that
$$\pi_{h \cdot i}(b) \pi_{h \cdot j}(b^{-1}) = (ch) \pi_i(b) \pi_j(b^{-1}) (ch)^{-1} = \pi_i(b) \pi_j(b^{-1}) \; .$$
Considering the support of the left and right hand side, it follows that $\{h \cdot i,h \cdot j\} = \{i,j\}$, so that the claim is proven.

Choose any $i \in I$ and $h \in \zeta(\delta(\Sigma))$. Since $S$ is nonempty and invariant under the diagonal action of $\Gamma$ and since the action of $\Gamma$ on $I$ is transitive, we find $j \neq i$ such that $(i,j) \in S$. For every $s \in \Stab_\Gamma i$, we also get that $(i,s \cdot j) = (s \cdot i,s \cdot j) \in S$. Since $(\Stab_\Gamma i) \cdot j$ is infinite, we can thus take $k \in I$ such that $(i,k) \in S$ and such that the three elements $i,j,k$ are distinct. By the claim above, $\{h \cdot i,h \cdot j\} = \{i,j\}$ and $\{h \cdot i,h\cdot k\} = \{i,k\}$. Since $i,j,k$ are distinct, it follows that $h \cdot i = i$. Since this holds for all $i \in I$ and the action $\Gamma \actson I$ is faithful, it follows that $h = e$. So, $\zeta(\delta(\Sigma)) = \{e\}$. This means that $\delta(\Sigma) \subset \Sigma$. So the assumption $\delta(\Sigma) \not\subset \Sigma$ implies that $\delta(\Sigma) \subset \Sigma$, which is absurd. This concludes the proof of step~1.

{\bf Step 2~:} for every $i \in I$ and $b_0 \in \Sigma_0 \setminus \{e\}$, there exists a $j \in I$ such that $\delta(\pi_i(b_0)) \in \pi_j(\Sigma_0 \setminus \{e\})$.

Define the finite nonempty subset $\cF_i \subset I$ as the support of the element $\delta(\pi_i(b_0)) \in \Sigma \setminus \{e\}$. The centralizer of $\pi_i(b_0)$ in $G$ equals $\Sigma \cdot \Stab_\Gamma(i)$. Since the support of $\delta(\pi_i(b_0))$ equals $\cF_i$, the centralizer of $\delta(\pi_i(b_0))$ in $G$ is contained in $\Sigma \cdot \Norm_\Gamma(\cF_i)$, where
$$\Norm_\Gamma(\cF_i) = \{h \in \Gamma \mid h \cdot \cF_i = \cF_i\} \; .$$
So, $\delta(\Sigma \cdot \Stab_\Gamma(i)) \subset \Sigma \cdot \Norm_\Gamma(\cF_i)$. Choose $j \in \cF_i$. Since $\cF_i$ is finite, $\Norm_\Gamma(\cF_i) \cap \Stab_\Gamma(j)$ has finite index in $\Norm_\Gamma(\cF_i)$. We can thus choose a finite index subgroup $\Lambda_0 < \Stab_\Gamma(i)$ such that $\delta(\Sigma \cdot \Lambda_0) \subset \Sigma \cdot \Stab_\Gamma(j)$.

Applying the same reasoning to $j \in I$ and $\delta^{-1}$, we also find a finite index subgroup $\Lambda_1 < \Stab_\Gamma(j)$ and a $k \in I$ such that $\delta^{-1}(\Sigma \cdot \Lambda_1) \subset \Sigma \cdot \Stab_\Gamma(k)$. Making $\Lambda_0$ smaller, we may assume that $\delta(\Sigma \cdot \Lambda_0) \subset \Sigma \cdot \Lambda_1$. Applying $\delta^{-1}$ it follows that $\Sigma \cdot \Lambda_0 \subset \Sigma \cdot \Stab_\Gamma(k)$. Since $\Lambda_0 < \Stab_\Gamma(i)$ has finite index and $\Stab_\Gamma(i) \cdot r$ is infinite for all $r \neq i$, we must have that $k = i$.

Write $\Lambda_2 = \Norm_\Gamma(\cF_i) \cap \Lambda_1$. We have $\delta(\Sigma \cdot \Lambda_0) \subset \Sigma \cdot \Lambda_2$. Applying $\delta^{-1}$, we get that
\begin{equation}\label{eq.chain-inclusions}
\Sigma \cdot \Lambda_0 \subset \delta^{-1}(\Sigma \cdot \Lambda_2) \subset \delta^{-1}(\Sigma \cdot \Lambda_1) \subset \Sigma \cdot \Stab_\Gamma(i) \; .
\end{equation}
Since $\Lambda_0 < \Stab_\Gamma(i)$ has finite index, also $\Sigma \cdot \Lambda_0 < \Sigma \cdot \Stab_\Gamma(i)$ has finite index. It then follows from \eqref{eq.chain-inclusions} that $\Sigma \cdot \Lambda_2 < \Sigma \cdot \Lambda_1$ has finite index, so that $\Lambda_2 < \Lambda_1$ has finite index. Since $\Stab_\Gamma(j) \cdot r$ is infinite for all $r \neq j$, it follows that $\cF_i = \{j\}$. This means that $\delta(\pi_i(b_0)) \in \pi_j(\Sigma_0 \setminus \{e\})$ and concludes the proof of step~2.

{\bf Step 3~:} for every $i \in I$, there exists a $j \in I$ such that $\delta(\pi_i(\Sigma_0)) = \pi_j(\Sigma_0)$. Choose any $b_0 \in \Sigma_0 \setminus \{0\}$. By the previous step, we can take $j \in I$ such that $\delta(\pi_i(b_0)) \in \pi_j(\Sigma_0 \setminus \{e\})$. Taking centralizers, it follows that $\delta(\Sigma \cdot \Stab_\Gamma(i)) = \Sigma \cdot \Stab_\Gamma(j)$. Taking once more the centralizer, we conclude that $\delta(\pi_i(\Sigma_0)) = \pi_j(\Sigma_0)$.

{\bf End of the proof.} Fix $i_0 \in I$. By the previous step, take $j \in I$ such that $\delta(\pi_{i_0}(\Sigma_0)) = \pi_j(\Sigma_0)$. Since $\Gamma \actson I$ is transitive, take $g_0 \in \Gamma$ such that $g_0 \cdot j = i_0$. Then, $(\Ad g_0) \circ \delta$ globally preserves $\pi_{i_0}(\Sigma_0)$, thus leading to an automorphism $\delta_0 \in \Aut \Sigma_0$ satisfying the conclusion of the lemma.
\end{proof}

When $\Gamma \actson I$ is a transitive action, $\Sigma_0 = \Z/n\Z$ with $n \geq 2$, $\Sigma = \Sigma_0^{(I)}$ and $G = \Sigma \rtimes \Gamma$ is the corresponding generalized wreath product group, the canonical quotient homomorphisms $\Gamma \to \Gamma\ab$ and $\Sigma \to \Sigma_0 : a \mapsto \sum_{i \in I} a_i$ combine into the homomorphism $G \to \Sigma_0 \times \Gamma\ab$ that identifies $G\ab = \Sigma_0 \times \Gamma\ab$.

Let $C$ be an abelian group. Composing the identification $C/n\cdot C = \Ext^1(\Sigma_0,C) \subset \Ext^1(G\ab,C)$ with the embedding $\Upsilon : \Ext^1(G\ab,C) \to H^2(G,C)$ given by \eqref{eq.univ-coeff}, we find a faithful group homomorphism
\begin{equation}\label{eq.my-hom-vphi}
\vphi : C/n\cdot C \to H^2(G,C) \; .
\end{equation}
Still identifying $C/n\cdot C = \Ext^1(\Sigma_0,C) = H^2(\Sigma_0,C)$ and denoting, for a given $i_0 \in I$, by $\pi_{i_0} : \Sigma_0 \to \Sigma \subset G$ the embedding in the $i_0$-th coordinate, we find the homomorphism
\begin{equation}\label{eq.my-hom-Phi}
\Phi : H^2(G,C) \to C/n\cdot C : \Phi(\Om) = \Om \circ \pi_{i_0} \; ,
\end{equation}
satisfying $\Phi \circ \vphi = \id$.

\begin{proof}[{Proof of Proposition \ref{prop.new-no-go-result-2}}]
Let $C$ be any abelian group such that $n \cdot C \neq C$ and let $\Gtil$ be an arbitrary central extension of $G$ be $C$ with associated $2$-cocycle $\Om \in H^2(G,C)$. Using the notation of \eqref{eq.my-hom-Phi}, define $c_0 = \Phi(\Om)$ and fix any nonzero element $c_1 \in C/n \cdot C$. If $c_0 = 0$, use \eqref{eq.my-hom-vphi} to define $\Om' = \Om + \vphi(c_1)$. If $c_0 \neq 0$, define $\Om' = \Om - \vphi(c_0)$.

In both cases, we have found $\Om' \in H^2(G,C)$ such that among the elements $\Phi(\Om)$, $\Phi(\Om')$ of $C/n \cdot C$, one is zero and the other is nonzero.

Denote by $\Lambdatil$ the central extension corresponding to $\Om'$. By construction (cf. \eqref{eq.univ-coeff}), $\mu \circ \vphi(c) = 1$ in $H^2(G,\T)$ for all $\mu \in \Chat$ and all $c \in C/n \cdot C$. It follows that $\mu \circ \Om' = \mu \circ \Om$ for all $\mu \in \Chat$, so that $\Om \approx \Om'$. It follows from \eqref{eq.central-decomp} that $L(\Gtil) \cong L(\Lambdatil)$.

It remains to prove that $\Lambdatil \not\cong \Gtil$. Assume that $\deltatil : \Gtil \to \Lambdatil$ is an isomorphism. Since $G$ has trivial center, $\deltatil(C) = C$. We denote by $\rho$ the restriction of $\deltatil$ to $C$. Since $\deltatil(C) = C$, we get the factor automorphism $\delta \in \Aut G$. Since $\deltatil$ is an isomorphism between the two central extensions, $\rho \circ \Om = \Om' \circ \delta$ in $H^2(G,C)$.

Fix $i_0 \in I$. By Lemma \ref{lem.aut-generalized-wreath}, after composing $\deltatil$ with an inner automorphism, we may assume that $\delta \circ \pi_{i_0} = \pi_{i_0} \circ \delta_0$ for some $\delta_0 \in \Aut(\Z/n\Z)$.

This automorphism $\delta_0 \in \Aut(\Z/n\Z)$ is given by multiplication with an invertible element $s \in \Z/n\Z$. One then checks that $\Phi(\Om' \circ \delta) = s^{-1} \cdot \Phi(\Om')$. Since $\Phi(\rho \circ \Om) = \rho(\Phi(\Om))$ and since $\rho \circ \Om = \Om' \circ \delta$, we conclude that $\rho(\Phi(\Om)) = s^{-1} \cdot \Phi(\Om')$. This is absurd because one of these two elements of $C/n \cdot C$ is zero and the other one is nonzero. We have thus proven that $\Lambdatil$ and $\Gtil$ are not isomorphic.
\end{proof}

In our final no-go theorem, we illustrate the role of the vanishing condition $\Ext^1(G\ab,C)=0$ in Theorem \ref{thm.inherit-iso-superrigidity.4} and the role of conditions \ref{ncond.iso.4} and \ref{ncond.iso.5} in Theorem \ref{thm.inherit-virtual-iso-superrigidity.4} and \ref{thm.inherit-virtual-iso-superrigidity.5}.

\begin{proposition}\label{prop.new-no-go-result-3}
Let $\cP$ be a set of prime numbers and $q$ a prime number. Consider the subring $C = \Z[\cP^{-1}]$ of $\Q$ and let $F$ be either $\Z[q^{-1}]$ or its quotient $\Z[q^{-1}]/\Z$. Define $\Gamma = C \ast F$ and the left-right wreath product group $G = (\Z/2\Z)^{(\Gamma)} \rtimes (\Gamma \times \Gamma)$.

Denote by $\zeta : G \to C \times C$ the surjective homomorphism given by composing the natural maps $G \to \Gamma \times \Gamma \to C \times C$. Let $\gamma : C \times C \to C$ be a bihomomorphism such that $\gamma(C \times C)$ generates the group $C$. Define $\Om_\gamma \in H^2(G,C)$ by $\Om_\gamma(g,h) = \gamma(b,a')$ when $\zeta(g)=(a,b)$ and $\zeta(h) = (a',b')$.

Define $\Gtil$ as the central extension of $G$ by $C$ given by $\Om_\gamma$. We use the notations of Theorems \ref{thm.inherit-iso-superrigidity} and \ref{thm.inherit-virtual-iso-superrigidity}.

\begin{enuma}
\item\label{prop.new-no-go-result-3.1} For every countable group $\Lambdatil$, the conclusions of \ref{thm.inherit-iso-superrigidity.2}, \ref{thm.inherit-iso-superrigidity.3}, \ref{thm.inherit-virtual-iso-superrigidity.2} and \ref{thm.inherit-virtual-iso-superrigidity.3} hold:
\begin{itemlist}
\item $p L(\Gtil) p \cong L(\Lambdatil)$ iff $p = 1$ and $\Lambdatil$ belongs to $\cH_\approx$~;
\item there is a nonzero\slash faithful bifinite $L(\Gtil)$-$L(\Lambdatil)$-bimodule iff $\Lambdatil$ is virtually isomorphic\slash commensurable to a group in $\cF_{\approx\ve}$.
\end{itemlist}
\item\label{prop.new-no-go-result-3.2} If $2 \not\in \cP$, $q \in \cP$ and $F = \Z[q^{-1}]$, then \ref{thm.inherit-iso-superrigidity.4} fails, but \ref{thm.inherit-virtual-iso-superrigidity.4} and \ref{thm.inherit-virtual-iso-superrigidity.5} hold: $\cH_\approx$ contains groups that are not isomorphic with $\Gtil$, but every group in $\cF_{\approx\ve}$ is commensurable to $\Gtil$.
\item\label{prop.new-no-go-result-3.3} If $q \not\in \cP$ and $F = \Z[q^{-1}]$, then \ref{thm.inherit-iso-superrigidity.4}, \ref{thm.inherit-virtual-iso-superrigidity.4} and \ref{thm.inherit-virtual-iso-superrigidity.5} all fail: $\cH_\approx$ contains a group that is not virtually isomorphic to $\Gtil$.
\item\label{prop.new-no-go-result-3.4} If $2,q \in \cP$ and $F = \Z[q^{-1}]/\Z$, then \ref{thm.inherit-iso-superrigidity.4} and \ref{thm.inherit-virtual-iso-superrigidity.4} hold, but \ref{thm.inherit-virtual-iso-superrigidity.5} fails: all groups in $\cH_\approx$ are isomorphic with $\Gtil$, all groups in $\cF_{\approx\ve}$ are virtually isomorphic to $\Gtil$, but $\cF_{\approx\ve}$ contains a group that is not commensurable to $\Gtil$.
\end{enuma}
\end{proposition}

Note that Proposition \ref{prop.new-no-go-result-3.2} provides groups $\Gtil$ that are not W$^*$-superrigid, but nevertheless satisfy the virtual W$^*$-superrigidity properties of Theorem \ref{thm.main-A.2} and \ref{thm.main-A.3}.

The groups $\Gtil$ in Proposition \ref{prop.new-no-go-result-3.3} do not satisfy any of the (virtual) W$^*$-superrigidity properties of Theorem \ref{thm.main-A}, but nevertheless Proposition \ref{prop.new-no-go-result-3.1} describes all countable groups $\Lambdatil$ for which $L(\Lambdatil)$ is (virtually) isomorphic to $L(\Gtil)$.

The groups $\Gtil$ in Proposition \ref{prop.new-no-go-result-3.4} satisfy the (virtual) W$^*$-superrigidity properties of Theorem \ref{thm.main-A.1} and \ref{thm.main-A.2}, but do not satisfy the virtual W$^*$-superrigidity property of Theorem \ref{thm.main-A.3}.

Finally note that when $2,q \in \cP$ and $F = \Z[q^{-1}]$, the groups $\Gtil$ in Proposition \ref{prop.new-no-go-result-3} satisfy all the W$^*$-superrigidity properties of Theorem \ref{thm.main-A}, because they are covered by Theorem \ref{thm.more-general}. So Proposition \ref{prop.new-no-go-result-3} also illustrates that the assumptions $\cP_1 \cup \cP_2 \subset \cQ$ and $2 \in \cQ$ in Theorem \ref{thm.more-general} are essential.

\begin{proof}
First note that $\Z[q^{-1}]/\Z$ has no proper finite index subgroups, because they would be given by proper finite index subgroups of $\Z[q^{-1}]$ containing $\Z$, which do not exist by Lemma \ref{lem.good-properties-ZP.2}.

(a) The proof is entirely similar to the proof of Theorem \ref{thm.more-general}. By Lemma \ref{lem.about-all-our-conditions.1}, the group $G$ satisfies conditions \ref{cond.iso.1}, \ref{cond.iso.1prime} and \ref{ncond.iso.3}.

When $A$ is a finite index subgroup of $C$ or $\Z[q^{-1}]$, we know from Lemma \ref{lem.good-properties-ZP.2} that $A$ is isomorphic to $C$ or $\Z[q^{-1}]$, so that Lemma \ref{lem.good-properties-ZP.4} says that $H^2(A,\T) = 1$. When $A$ is a finite index subgroup of $\Z[q^{-1}]/\Z$, we have that $A = \Z[q^{-1}]/\Z$, so that $A$ is a direct limit of cyclic groups. By the same argument in as in the proof of Lemma \ref{lem.good-properties-ZP.4}, we get that $H^2(A,\T) = 1$. By Lemma \ref{lem.about-all-our-conditions.4}, we get that the group $G$ satisfies condition \ref{ncond.iso.1primeprime}.

Since $\Z[q^{-1}]/\Z$ has an uncountable automorphism group, the group $G$ does not satisfy conditions \ref{ncond.iso.2} and \ref{ncond.iso.2prime} in case $F = \Z[q^{-1}]/\Z$, but we can see as follows that $G$ satisfies the weaker conditions appearing in Remarks \ref{rem.relax-iso} and \ref{rem.relax-v-iso}.

Every automorphism $\delta_0 \in \Aut F$ extends to an automorphism $\delta_1$ of $\Gamma = C \ast F$ by acting as the identity on $C$. Then $\delta_1 \times \delta_1$ uniquely extends to an automorphism $\delta_2$ of $G$ by acting as the identity on the copy of $\Z/2\Z$ in position $e \in \Gamma$. By construction, $\Om_\gamma \circ \delta_2 = \Om_\gamma$. Viewing in this way $\Aut F$ as a subgroup of $\Aut G$ and using that $\Aut C$ is countable, the proof of Lemma \ref{lem.countable-aut} shows that the set $\Aut G / \Aut F$ is countable. Since $\Z[q^{-1}]/\Z$ has no proper finite index subgroups and since the groups $C$ and $\Z[q^{-1}]$ have only countably many finite index subgroups by Lemma \ref{lem.good-properties-ZP.2}, the proof of Lemma \ref{lem.about-all-our-conditions.2} says that $G$ has only countably many finite index subgroups. So all conditions of Remarks \ref{rem.relax-iso} and \ref{rem.relax-v-iso} are satisfied.

The proof of Theorems \ref{thm.more-general.1} and \ref{thm.more-general.2} shows that $\Theta(\Om_\gamma)$ is faithful and that $\Theta(\Om_\gamma|_{G_0})$ has finite kernel for every finite index subgroup $G_0 < G$.

So (a) follows from Theorems \ref{thm.inherit-iso-superrigidity.2}, \ref{thm.inherit-iso-superrigidity.3}, \ref{thm.inherit-virtual-iso-superrigidity.2} and \ref{thm.inherit-virtual-iso-superrigidity.3}, combined with Remarks \ref{rem.relax-iso} and \ref{rem.relax-v-iso}.

(b) Since $2 \not\in \cP$, it follows from Proposition \ref{prop.new-no-go-result-2} that $\cH_\approx$ contains groups that are not isomorphic with $\Gtil$.

By Lemma \ref{lem.about-all-our-conditions.6}, the group $G$ satisfies condition \ref{ncond.iso.5}. We prove that, because $q \in \cP$, the group $G$ also satisfies condition \ref{ncond.iso.4}. Take a finite index subgroup $G_0 < G$ and an element $\Psi \in \Ext^1(G_{0,\text{\rm ab}},D)$, where $D$ is a torsion free abelian group that is commensurable to $C$. We have to prove that there exists a finite index subgroup $G_1 < G_0$ such that $\Psi \circ \pi = 0$ in $\Ext^1(G_{1,\text{\rm ab}},D)$, where $\pi : G_{1,\text{\rm ab}} \to G_{0,\text{\rm ab}}$ is the canonical homomorphism.

By Lemma \ref{lem.good-properties-ZP.3}, we may assume that $D=C$. Using the first paragraphs of the proof of Lemma \ref{lem.about-all-our-conditions} and the notation introduced there, we may assume that $G_0 = \Sigma(\Gamma_0) \rtimes (\Gamma_0 \times \Gamma_0)$, where $\Gamma_0 \lhd \Gamma$ is a finite index normal subgroup. From the proof of \ref{lem.about-all-our-conditions.6}, we get that $G_{0,\text{\rm ab}} = \Sigma_1 \times \Gamma_{0,\text{\rm ab}} \times \Gamma_{0,\text{\rm ab}}$, with $\Sigma_1$ finite. Defining $G_1$ as the preimage of $\{0\} \times \Gamma_{0,\text{\rm ab}} \times \Gamma_{0,\text{\rm ab}}$, it thus suffices to prove that $\Ext^1(\Gamma_{0,\text{\rm ab}},C) = 0$.

In the proof of Lemma \ref{lem.good-properties-ZP.2}, we have seen that $\Gamma_{0,\text{\rm ab}}$ is a direct sum of copies of $\Z$ and finite index subgroups of $C$ and $F$. By Lemma \ref{lem.good-properties-ZP.2} and \ref{lem.good-properties-ZP.4}, using at this point that $q \in \cP$, we find that $\Ext^1(\Gamma_{0,\text{\rm ab}},C) = 0$.

So $G$ satisfies conditions \ref{ncond.iso.4} and \ref{ncond.iso.5}, and the conclusion follows from Theorem \ref{thm.inherit-virtual-iso-superrigidity.4} and \ref{thm.inherit-virtual-iso-superrigidity.5}.

(c) We start by constructing a sufficiently nontrivial abelian extension $0 \to C \to B_\infty \to F \to 0$. Fix a $q$-adic integer $b \in \Z_q$ that does not belong to $\Q$ inside $\Q_q$. Uniquely define, for every $n \geq 0$, the elements $b_n \in \{0,1,\ldots,q-1\}$ such that $b \in b_0 + q b_1 + \cdots + q^k b_k + q^{k+1} \Z_q$ for all $k \geq 0$. Define $B_n = C \times \Z$ and consider the faithful homomorphisms
$$\vphi_n : B_n \to B_{n+1} : \vphi_n(x,y) = (x + b_n y , q y) \; .$$
Define $B_\infty$ as the direct limit of this sequence of abelian groups. Since the canonical embeddings $C \to B_n : x \mapsto (x,0)$ are compatible with the maps $\vphi_n$, we view $C$ in this way as a subgroup of $B_\infty$. The homomorphisms $\theta_n : B_n \to F : \theta_n(x,y) = q^{-n} y$ satisfy $\theta_{n+1} \circ \vphi_n = \theta_n$ and thus combine into a well-defined surjective homomorphism $\theta : B_\infty \to F$, whose kernel is equal to $C$. So we get the abelian extension $0 \to C \to B_\infty \to F \to 0$ with associated $2$-cocycle $\Psi_1 \in \Ext^1(F,C) = H^2\sym(F,C)$.

We claim that for every finite index subgroup $F_0 < F$ and nonzero integer $r \in \Z$, $r \cdot \Psi_1|_{F_0}$ is nonzero in $\Ext^1(F_0,C)$. By Lemma \ref{lem.good-properties-ZP.2}, we can take an integer $N \geq 1$ such that $F_0 = N \cdot F$. Assume that $r \cdot \Psi_1|_{F_0} = 0$ in $\Ext^1(F_0,C)$. This means that there exists a group homomorphism $\psi : F_0 \to B_\infty$ satisfying $\theta(\psi(z)) = r z$ for all $z \in F_0$. Define $c_0 := \psi(N)$, so that $\theta(c_0) = r N$. For every $k \geq 1$, the element $N \in F_0$ is divisible by $q^k$. Since $\psi$ is a group homomorphism, $c_0$ is divisible by $q^k$ in $B_\infty$.

Viewing every $B_n$ as a subgroup of $B_\infty$, we have that $B_\infty = \bigcup_n B_n$. Fix $n \in \N$ such that $c_0 \in B_n$. We thus represent $c_0$ by the element $(x_0,q^n r N) \in B_n$ for some $x_0 \in C$. Since $C = \Z[\cP^{-1}]$, we can take an integer $\kappa \geq 1$ such that $\kappa x_0 \in \Z$. Define $c_1 = \kappa c_0$. Then $c_1$ is still divisible by $q^k$ for every $k \geq 1$. Also, $c_1$ belongs to $B_n$ and is represented by the element $(x_1,y_1)$, where $x_1 = \kappa x_0 \in \Z$ and $y_1 = \kappa q^n r N$, so that $y_1 \neq 0$.

Fix $k \geq 0$. Since $c_1$ is divisible by $q^{k+1}$, for all $l \geq 1$ large enough, the image of $(x_1,y_1)$ in $B_{n+l+1}$ is divisible by $q^{k+1}$ inside $B_{n+l+1}$. This implies that for all $l \geq 1$ large enough, the integer
$$x_1 + (b_n + q b_{n+1} + \cdots + q^l b_{n+l}) y_1$$
is divisible by $q^{k+1}$ inside $C$, and thus also inside $\Z$ because $q \not\in \cP$. It follows that
\begin{equation}\label{eq.almost-in-Qq}
q^{k+1} \;\;\text{divides}\;\; x_1 + (b_n + q b_{n+1} + \cdots + q^k b_{n+k}) y_1 \quad\text{for all $k \geq 0$.}
\end{equation}
Define the integer $d = b_0 + q b_1 + \cdots + q^{n-1} b_{n-1}$. Then \eqref{eq.almost-in-Qq} says that $x_1 + q^{-n}(b-d)y_1 = 0$ in $\Q_q$. Since $y_1$ is a nonzero integer, it follows that $b \in \Q$, contrary to our assumption. So the claim is proven.

Define the quotient homomorphism $\zeta_1 : G \to F$ by composing $G \to \Gamma \times \Gamma \to F \times F \to F$, where in the last step, we project onto the first coordinate. Viewing $\Psi_1$ as a symmetric $2$-cocycle on $F$, we define the symmetric $2$-cocycle $\Psi \in H^2(G,C)$ by $\Psi = \Psi_1 \circ \zeta_1$. Since $\Psi_1$ is defined by an element of $\Ext^1(F,C)$, we have that $\mu \circ \Psi_1 = 1$ in $H^2(F,\T)$ for all $\mu \in \Chat$. So we also have that $\mu \circ \Psi = 1$ in $H^2(G,\T)$ for all $\mu \in \Chat$. Define $\Lambdatil$ as the extension of $G$ by $C$ associated with the $2$-cocycle $\Om_\gamma + \Psi$. Since by construction, $\mu \circ \Psi = 1$ in $H^2(G,\T)$ for all $\mu \in \Chat$, we have that $\Om_\gamma + \Psi \approx \Om_\gamma$, so that $\Lambdatil$ belongs to $\cH_\approx$.

To conclude the proof of (c), we have to show that $\Lambdatil$ is not virtually isomorphic to $\Gtil$. Assume the contrary. Since $G$ is icc and $C$ is torsion free, no finite index subgroup of $\Lambdatil$ or $\Gtil$ has a nontrivial finite normal subgroup. So, $\Lambdatil$ and $\Gtil$ are commensurable. Reasoning as in the proof of Theorem \ref{thm.more-general.4}, we find a finite index subgroup $G_0 < G$, an automorphism $\delta \in \Aut G$ and nonzero integers $r,s$ such that
\begin{equation}\label{eq.good-cohom-here}
r \cdot (\Om_\gamma + \Psi)|_{G_0} = s \cdot \Om_\gamma \circ \delta|_{G_0} \quad\text{in $H^2(G_0,C)$.}
\end{equation}
By \cite[Proposition 6.10]{DV24}, after composing $\delta$ with an inner automorphism, we find $\al \in \Aut \Gamma$ such that either $\delta(g,e) = (\al(g),e)$ for all $g \in \Gamma$, or $\delta(g,e) = (e,\al(g))$ for all $g \in \Gamma$. Since $q \not\in \cP$, by the argument in the proof of Lemma \ref{lem.countable-aut}, we may further compose $\delta$ with an inner automorphism and assume that $\al(F) = F$. Define the finite index subgroup $F_0 < F$ such that $F_0 \times \{e\} = (F \times \{e\}) \cap G_0$. Restricting \eqref{eq.good-cohom-here} to $F_0 \times \{e\}$ now implies that $r \cdot \Psi_1|_{F_0} = 0$. This contradicts the statement proven above.

(d) The group $F$ has no proper finite index subgroups. Since $q \in \cP$, it follows from Lemma \ref{lem.basic-cohomology.2} that $\Ext^1(F,C) = 0$. By Lemma \ref{lem.about-all-our-conditions.3}, we get that $\Ext^1(G\ab,C) = 0$. By Lemma \ref{lem.about-all-our-conditions.5}, also condition \ref{ncond.iso.4} holds. So the conclusions of \ref{thm.inherit-iso-superrigidity.3} and \ref{thm.inherit-virtual-iso-superrigidity.4} hold, meaning that every group in $\cH_\approx$ is isomorphic with $\Gtil$ and every group in $\cF_{\approx\ve}$ is virtually isomorphic to $\Gtil$.

Write $E = \Z/q\Z$, which we view as a subgroup of $F$. Then $E$ is the kernel of the map $F \to F : x \mapsto q x$. This defines the abelian extension $0 \to E \to F \to F \to 0$, that we denote as $\Psi_1 \in \Ext^1(F,E)$. Using the same quotient homomorphism $\zeta_1 : G \to F$ as in the proof of (c), we define $\Psi \in H^2(G,E)$ by $\Psi = \Psi_1 \circ \zeta_1$. Define $\Om_1 \in H^2(G,C \times E)$ as $\Om_1 = \Om_\gamma \oplus \Psi$. Define $\Lambdatil$ as the central extension of $G$ by $C \times E$ associated with $\Om_1$. By construction, $\Om_1 \approx\ve \Om_\gamma$, so that $\Lambdatil$ belongs to $\cF_{\approx\ve}$.

To conclude the proof of (d), assume by contradiction that $\Lambdatil_0 < \Lambdatil$ is a finite index subgroup that is isomorphic to a finite index subgroup of $\Gtil$. Define the finite index subgroup $G_0 < G$ as the image of $\Lambdatil_0$ under the quotient homomorphism $\Lambdatil \to G$. Define the finite index subgroup $D < C \times E$ by $D = (C \times E) \cap \Lambdatil_0$. By construction, $\Lambdatil_0$ is a central extension of $G_0$ by $D$. We denote by $\Om_0 \in H^2(G_0,D)$ the associated $2$-cocycle. We also have by construction that $\Om_0 = \Om_1|_{G_0}$ in $H^2(G_0,C \times E)$.

Since $\Lambdatil_0$ is isomorphic to a finite index subgroup of $\Gtil$, the center $\cZ(\Lambdatil_0)$ is torsion free. So, $D$ intersects $\{0\} \times E$ trivially and the projection $\lambda : C \times E \to C$ on the first coordinate restricts to an isomorphism between $D$ and a finite index subgroup of $C$. By Lemma \ref{lem.good-properties-ZP.2}, $D \cong C$. In particular, $D$ is $p$-divisible.

Denote by $\rho : C \times E \to E$ the projection on the second coordinate. Since $D$ is $p$-divisible and $E = \Z/p\Z$, we get that $\rho(D) = 0$. It follows that $\rho \circ \Om_0 = 0$. Since $\Om_0 = \Om_1|_{G_0}$ in $H^2(G_0,C \times E)$, it follows that $\rho \circ \Om_1|_{G_0} = 0$ in $H^2(G_0,E)$. This means that $\Psi|_{G_0} = 0$ in $H^2(G_0,E)$. Since $F$ has no proper finite index subgroups, we have $F \times \{e\} \subset G_0$. Further restricting $\Psi$ to $F \times \{e\}$, we get that $\Psi_1 = 0$ in $\Ext^1(F,E)$, which is absurd.
\end{proof}

\section{Proof of Theorem \ref{thm.main-A} and Remark \ref{rem.main}}\label{sec.proof-main}

Theorem \ref{thm.main-A} and Remark \ref{rem.main} are immediate consequences of Theorem \ref{thm.more-general} and Proposition \ref{prop.new-no-go-result-1}, once we have proven the following elementary lemma.

\begin{lemma}\label{lem.liftable-aut}
For $i \in \{1,2\}$, let $\cP_i$ be a set of prime numbers and put $C_i = \Z[\cP_i^{-1}]$ with unit group $C_i^\times$. Define $\Gamma = C_1 \ast C_2$ with abelianization $\Lambda = C_1 \times C_2$.
\begin{enuma}
\item\label{lem.liftable-aut.1} If $\cP_2$ is a proper subset of $\cP_1$, every automorphism $\vphi \in \Aut \Lambda$ is of the form $\vphi(a,b) = (ca + eb,d b)$ for some $c \in C_1^\times$, $e \in C_1$ and $d \in C_2^\times$.

    If $\cP_2 = \emptyset$, all these automorphisms can be lifted to an automorphism of $\Gamma$. If $\cP_2 \neq \emptyset$, such an automorphism can be lifted to an automorphism of $\Gamma$ if and only if $e=0$.

\item\label{lem.liftable-aut.2} If both $\cP_1 \setminus \cP_2$ and $\cP_2 \setminus \cP_1$ are nonempty, every automorphism $\vphi \in \Aut \Lambda$ is of the form $\vphi(a,b) = (ca,d b)$ for some $c \in C_1^\times$ and $d \in C_2^\times$. All these automorphisms can be lifted to an automorphism of $\Gamma$.

\item\label{lem.liftable-aut.3} If $\cP_1 = \cP_2 \neq \emptyset$, we write $\cP = \cP_i$ and $C = C_i$. The automorphisms of $\Lambda$ are given by multiplication with a matrix $A \in C^{2 \times 2}$ with determinant in $C^\times$. Such an automorphism $\vphi$ can be lifted to an automorphism of $\Gamma$ if and only if it is either of the form $\vphi(a,b) = (c a,d b)$ or the form $\vphi(a,b) = (d b,c a)$ for some $c,d \in C^\times$.
\end{enuma}
\end{lemma}

\begin{proof}
To avoid confusion, we explicitly denote by $\eta_i : C_i \to C_1 \ast C_2$ the natural embedding.

(a) Fix a prime $p \in \cP_1 \setminus \cP_2$ and fix $\vphi \in \Aut \Lambda$. Since $\vphi(a,0)$ is $p$-divisible for every $a \in C_1$, we must have that $\vphi(C_1,0) \subset C_1 \times 0$. Since the same holds for $\vphi^{-1}$, it follows that $\vphi$ has the form described in (a).

If $\cP_2 = \emptyset$, we lift $\vphi$ to the unique automorphism $\delta \in \Aut \Gamma$ satisfying $\delta(\eta_1(a)) = \eta_1(ca)$ for all $a \in C_1$ and $\delta(\eta_2(b)) = (\eta_1(e)\eta_2(d))^b$ for all $b \in \Z = C_2$.

Next assume that $\cP_2 \neq \emptyset$. If $e = 0$, we define $\delta_i \in \Aut C_i$ by $\delta_1(a) = ca$ and $\delta_2(b) = db$. Then $\delta = \delta_1 \ast \delta_2$ is a lift of $\vphi$.

Conversely, if $\vphi$ lifts to an automorphism $\delta \in \Aut \Gamma$, by the Kurosh subgroup theorem (see Section \ref{sec.kurosh}), $\delta(\eta_2(C_2))$ is contained in a conjugate of $\eta_i(C_i)$ for some $i \in \{1,2\}$. It follows that $\vphi(C_2) \subset C_i$. Given the form of $\vphi$, this means that $i=2$ and $e=0$.

(b) The same argument as in the proof of (a) shows that every automorphism $\vphi$ of $\Lambda$ is of the form stated in (b). As in (a), such an automorphism always lifts to an automorphism of the form $\delta = \delta_1 \ast \delta_2$.

(c) Let $\vphi \in \Aut \Lambda$. Writing elements of $C$ as column matrices, we define the matrix $A \in C^{2 \times 2}$ with first column $\vphi(1,0)$ and second column $\vphi(0,1)$. It follows that $\vphi(a,b) = A \cdot \bigl(\begin{smallmatrix} a \\ b \end{smallmatrix}\bigr)$. Using the matrix associated with $\vphi^{-1}$, it follows that $\det A \in C^\times$.

If $\vphi$ has one of the two forms described in (c), we can lift $\vphi$ to an automorphism $\delta \in \Aut \Gamma$ that is either of the form $\delta_1 \ast \delta_2$ or $\sigma \circ (\delta_1 \ast \delta_2)$, where $\sigma \in \Aut(C \ast C)$ flips the two free factors.

Conversely, if an arbitrary automorphism $\vphi$ given by a matrix $A \in C^{2 \times 2}$ lifts to an automorphism $\delta \in \Aut \Gamma$, by the Kurosh subgroup theorem (see Section \ref{sec.kurosh}), we find $i,j \in \{1,2\}$ such that $\delta(\eta_1(C))$ is contained in a conjugate of $\eta_i(C)$ and $\delta(\eta_2(C))$ is contained in a conjugate of $\eta_j(C)$. It follows that both columns of the matrix $A$ have only one nonzero entry. This forces $\vphi$ to have one of the forms described in (c).
\end{proof}

To prove Theorem \ref{thm.main-A} and Remark \ref{rem.main}, it now suffices to put all pieces together.

\begin{proof}[{Proof of Theorem \ref{thm.main-A}}]
Denote by $\cJ$ the set of all square free integers $n \geq 1$. We use the notation of Theorem \ref{thm.more-general} with $\cP_1 = \cQ = \cP$ and $\cP_2 = \emptyset$. So, $C = \Z[\cP^{-1}]$, $\Gamma = C \ast \Z$ and $\Lambda = C \times \Z$. We define for every $n \in \cJ$ the bihomomorphism $\gamma_n : \Lambda \times \Lambda \to C$ by $\gamma_n((a,b),(a',b')) = aa' + n bb'$.

It follows from Theorem \ref{thm.more-general.1}, resp.\ \ref{thm.more-general.2} that the associated central extensions $G_{\gamma_n}$ satisfy the W$^*$-super\-rigidity properties of Theorem \ref{thm.main-A.1}, resp.\ \ref{thm.main-A.2}.

Combining Theorem \ref{thm.more-general.4} and Lemma \ref{lem.liftable-aut.1}, it follows that for distinct elements $n \in \cJ$, the corresponding groups $G_{\gamma_n}$ are not virtually isomorphic.
\end{proof}

\begin{proof}[{Proof of Remark \ref{rem.main}}]
We continue to use the notation introduced at the start of the proof of Theorem \ref{thm.main-A}.

First assume that $\cP$ is a proper set of prime numbers and fix a prime $p \not\in \cP$.

(a) For every $n \in \cJ$, define the bihomomorphism $\gamma'_n = p \cdot \gamma_n$, whose image equals the proper finite index subgroup $p \cdot C < C$. It follows from Theorem \ref{thm.more-general.2} that the associated central extensions $G_{\gamma'_n}$ satisfy the virtual isomorphism W$^*$-superrigidity property of Theorem \ref{thm.main-A.2}. By Proposition \ref{prop.new-no-go-result-1.1}, these groups $G_{\gamma'_n}$ are not isomorphism W$^*$-superrigid. By construction, the groups $G_{\gamma'_n}$ are virtually isomorphic to $G_{\gamma_n}$. Since we proved above that the groups $G_{\gamma_n}$ are mutually not virtually isomorphic, the same holds for $G_{\gamma'_n}$.

(b) Denote by $\cK$ the set of all integers $n \geq 1$ having all their prime divisors outside $\cP$. For every integer $n \in \cK$, define the bihomomorphism $\tau_n : \Lambda \times \Lambda \to C : \tau_n((a,b),(a',b')) = n b b'$, with associated central extension $G_{\tau_n}$. Since the image of $\tau_n$ equals $n \Z < C$, it follows from Proposition \ref{prop.new-no-go-result-1.2} that there exist countable groups $\Lambda_n$ such that $\Lambda_n$ is not virtually isomorphic to $G_{\tau_n}$, but nevertheless $L(\Lambda_n) \cong L(G_{\tau_n})$. Combining Theorem \ref{thm.more-general.3} and Lemma \ref{lem.liftable-aut.1}, the groups $G_{\tau_n}$ are mutually nonisomorphic. Note however that the groups $G_{\tau_n}$ are all virtually isomorphic.

Next assume that $\cP$ contains all prime numbers, so that $C = \Q$.

(c) By the isomorphism \eqref{eq.iso-of-cohom}, every central extension of $G$ by $C$ is given by a bihomomorphism $\Lambda \times \Lambda \to \Q$. Except when $\gamma((a,b),(a',b')) = q bb'$ for some $q \in \Q$, the image of $\gamma$ equals $\Q$, so that by Theorem \ref{thm.more-general.1} and \ref{thm.more-general.2}, the corresponding central extension satisfies both the isomorphism and virtual isomorphism W$^*$-superrigidity property. In the exceptional case, the image of $\gamma$ equals $\Z q$ and has infinite index in $\Q$, so that by Proposition \ref{prop.new-no-go-result-1.2}, the corresponding extension $\Gtil_q$ is not W$^*$-superrigid in any sense. By Theorem \ref{thm.more-general.3}, all the groups $\Gtil_q$ with $q \neq 0$ are isomorphic.

We finally take $\Gamma = C \ast C$, and thus $\Lambda = C \times C$, with $C = \Z[\cP^{-1}]$ and $2 \in \cP$.

(c) Since every nonzero group homomorphism $C \to C$ has an image of finite index, the image of any nonzero bihomomorphism $\gamma : C \times C \to C$, generates a finite index subgroup of $C$. It then follows from Theorem \ref{thm.more-general.2} and the isomorphism \eqref{eq.iso-of-cohom} that every nontrivial central extension of $G$ by $C$ satisfies virtual isomorphism W$^*$-superrigidity.

If $p$ is a prime that does not belong to $\cP$, we can as above consider bihomomorphisms of the form $\gamma = p \gamma_0$ to obtain infinitely many nonisomorphic central extensions that are not isomorphism W$^*$-superrigid.

(e) When $C = \Q$ every nonzero bihomomorphism $\Lambda \times \Lambda \to \Q$ is surjective. It follows from Theorem \ref{thm.more-general} that every nontrivial central extension of $G$ by $\Q$ is both isomorphism W$^*$-superrigid and virtual isomorphism W$^*$-superrigid.
\end{proof}

\end{document}